\documentclass[11pt,a4paper,reqno]{amsart}
\usepackage[T1]{fontenc}
\usepackage[utf8]{inputenc}
\usepackage{lmodern}
\usepackage[a4paper,margin=30mm]{geometry}
\usepackage{amsmath,amssymb,amsthm,mathtools}
\usepackage{microtype}
\usepackage[unicode,hidelinks]{hyperref}
\usepackage{bookmark}
\hypersetup{
  pdftitle={The modulo 9 Kanade--Russell identities and their Nahm-sum duals},
  pdfauthor={Ernest X. W. Xia},
  pdfsubject={Rogers--Ramanujan type identities, generalized Nahm sums, and duality},
  pdfkeywords={Kanade--Russell identities, Rogers--Ramanujan identities, Nahm sums, duality, divided differences, residues}
}

\newtheorem{theorem}{Theorem}[section]
\newtheorem{lemma}[theorem]{Lemma}
\theoremstyle{remark}
\newtheorem*{remark}{Remark}
\numberwithin{equation}{section}
\allowdisplaybreaks[2]
\title[Modulo 9 Kanade--Russell identities and Nahm-sum duals]
{The modulo 9 Kanade--Russell identities and their Nahm-sum duals}
\author{Ernest X. W. Xia}
\address{School of Mathematical Sciences, Suzhou University of Science and Technology, Suzhou 215009, Jiangsu, P. R. China}
\email{ernestxwxia@163.com}

\begin{document}
	
	\begin{abstract}
		Kanade and Russell initiated a family of conjectural
		Rogers--Ramanujan type identities of moduli $9$ and $12$, which
		ultimately comprised five modulo $9$ identities and eleven
		modulo $12$ identities.  The eleven modulo $12$ conjectures were
		subsequently settled through the work of Bringmann,
		Jennings-Shaffer, and Mahlburg and of Rosengren.  In this paper
		we prove all five modulo $9$ Kanade--Russell sum-product identities,
		four individual generalized Nahm-sum dual identities, and a product
		formula for the natural dual companion of the fifth
		Kanade--Russell identity, which is expressed as a linear combination
		of two negative-mixed-term generalized Nahm sums.  The first three
		individual dual identities settle Conjecture~3.6 of Wang and Wang,
		while the fourth proves the corresponding conjecture of Li and Wang.
			Our results also connect directly with the recent Dynkin-diagram
		framework of Sun and Wang for generalized Nahm sums.  They identified
		the rank-two pairs $(T_1,G_2)$ and $(G_2,T_1)$ as unresolved cases
		whose modularity would follow, respectively, from the first
		modulo $9$ Kanade--Russell identity and its Wang-Wang dual.
		The present results prove precisely these two required identities
		and hence establish the corresponding modularity statements
		unconditionally.
	\end{abstract}
	
 \maketitle

\setcounter{tocdepth}{1}
\tableofcontents

\section{Introduction}

Two of the most important results in the theory of $q$-series are the classical 
Rogers--Ramanujan identities which state that
\begin{align}
 \mathcal G(q):=\sum_{r\ge0}\frac{q^{r^2}}{(q;q)_r}
&=\frac1{(q,q^4;q^5)_\infty}, \label{RR-1}\\
 \mathcal H(q):=\sum_{r\ge0}\frac{q^{r^2+r}}{(q;q)_r}
&=\frac1{(q^2,q^3;q^5)_\infty}.\label{RR-2}
\end{align}
Here
and throughout this paper we   use standard $q$-series notation: 
\[
(a;q)_0:=1,\qquad
(a;q)_n:=\prod_{j=0}^{n-1}(1-aq^j)\quad(n\ge1),
\qquad
(a;q)_\infty:=\prod_{j=0}^{\infty}(1-aq^j).
\]
In addition, for any positive
integer $m$, we write 
\[
(a_1,a_2,\ldots,a_m;q)_\infty:=(a_1;q)_\infty (a_2;q)_\infty \cdots (a_m;q)_\infty . 
\]
Infinite products and analytic identities are considered for $|q|<1$.
When a formula contains negative powers of $q$, we initially assume
$0<|q|<1$ and justify its extension at $q=0$ separately. We also use
formal power series, with limits understood coefficientwise.

 The identities were first discovered and proved by
 Rogers~\cite{rogers} in 1894 and were independently rediscovered
 by Ramanujan before 1913; Ramanujan's proof was published in
 1919~\cite{Ramanujan}.  Schur~\cite{schur} obtained independent
 proofs and their classical partition-theoretic interpretation in
 1917.

  These analytic identities have elegant partition-theoretic interpretations.
 The first asserts that the number of partitions of a nonnegative integer \(N\) in which adjacent parts differ by at least \(2\) equals the number of partitions of \(N\) into parts congruent to \(1\) or \(4\) modulo \(5\). The second gives the analogous equality when the adjacent parts differ by at least \(2\) and the smallest part is at least \(2\), with the corresponding product counting partitions into parts congruent to \(2\) or \(3\) modulo \(5\). A further structural interpretation was provided by Lepowsky and Milne \cite{lepowsky-milne}, who identified the product sides as principally specialized characters of level-three standard modules for the affine Lie algebra \(A_1^{(1)}\). Lepowsky and Wilson \cite{lepowsky-wilson}  then developed vertex-operator and \(Z\)-algebra methods to obtain representation-theoretic proofs.  During the twentieth century, the work of   Slater~\cite{slater}, Gordon~\cite{gordon}, Andrews~\cite{Andrews-1,Andrews-2}, and others led to numerous analytic and combinatorial generalizations, including the Bailey-pair method and the Andrews--Gordon identities. For a comprehensive introduction to the history, theory, and various
 extensions of the Rogers--Ramanujan identities, including many
 identities of Rogers--Ramanujan type, we refer to
 Sills~\cite{sills}.

Motivated by such partition identities, Kanade and
Russell~\cite{Kanade--Russell} developed the Maple package \texttt{IdentityFinder} to
search for product generating functions for partitions with prescribed
initial and difference conditions. Their method enumerates partitions subject to prescribed
initial and difference conditions and then uses Euler's algorithm to
detect possible infinite-product representations.    In their original
 paper, Kanade and Russell \cite{Kanade--Russell} proposed six conjectural identities, four
 of modulus \(9\) and two of modulus \(12\). A fifth modulo \(9\) companion was later
 recorded in Russell's doctoral thesis \cite{russell-thesis}, while
 nine additional modulo \(12\) conjectures appeared in subsequent work
 of Kanade and Russell \cite{staircases}.  Kur\c{s}ung\"oz \cite{kursungoz} 
 constructed positive Andrews--Gordon-type sum representations for
 the four original modulo \(9\) conjectures, together with alternative
 sum sides for the two original modulo \(12\) conjectures; the double-sum form of the fifth modulo \(9\)
 identity was subsequently given by Uncu and Zudilin
 \cite{uncu-zudilin}.   For the eleven modulo $12$ identities,
 Bringmann, Jennings--Shaffer, and Mahlburg~\cite{bringmann} proved
 seven and reduced the other four to basic hypergeometric identities.
 Rosengren~\cite{rosengren} proved the remaining four and gave new
 proofs of five of the previously established identities. In the
 modulo $9$ setting, Tsuchioka~\cite{tsuchioka} developed a
 vertex-operator reformulation involving $D_4^{(3)}$. Reflection of
 finite versions of these identities was studied by Uncu and
 Zudilin~\cite{uncu-zudilin} and Konenkov~\cite{konenkov}.

Our first aim is to prove the five modulo $9$ sum--product identities.
Define
\begin{equation}\label{eq:S}
 S(a,b):=\sum_{r,s\ge0}
 \frac{q^{r^2+3rs+3s^2+ar+bs}}
 {(q;q)_r(q^3;q^3)_s}.
\end{equation}
\begin{theorem}[The five modulo nine Kanade--Russell identities] 
\label{thm:main}
The following identities hold:
\begin{align}
 S(0,0)&=\frac1{(q,q^3,q^6,q^8;q^9)_\infty}=:P_1(q),\label{eq:kr1}\\
 S(1,3)&=\frac1{(q^2,q^3,q^6,q^7;q^9)_\infty}=:P_2(q),\label{eq:kr2}\\
 S(2,3)&=\frac1{(q^3,q^4,q^5,q^6;q^9)_\infty}=:P_3(q),\label{eq:kr3}
\end{align}
and
\begin{align}
 S(1,2)&=\frac1{(q^2,q^3,q^5,q^8;q^9)_\infty}=:P_4(q),\label{eq:kr4}\\
 S(1,4)+qS(2,4)&=\frac1{(q,q^4,q^6,q^7;q^9)_\infty}=:P_5(q).
 \label{eq:kr5}
\end{align}
\end{theorem}

The arguments of the $P_i$ are suppressed when no ambiguity arises.
The sum representations in the first four identities were obtained
by Kur\c{s}ung\"oz~\cite{kursungoz}, whereas the double-sum
representation in the fifth identity was subsequently given by
Uncu and Zudilin~\cite{uncu-zudilin}.

Our second aim is to establish four individual identities with the
opposite mixed term, together with a dual companion for the fifth
Kanade--Russell identity. For integer parameters $a,b$, write
\begin{equation}\label{eq:dual-sum}
 \mathfrak N(a,b):=\sum_{r,s\ge0}
 \frac{q^{r^2-3rs+3s^2+ar+bs}}
 {(q;q)_r(q^3;q^3)_s}.
\end{equation}
Positive definiteness of the quadratic form makes this a well-defined
formal Laurent series for every fixed pair $a,b$. All four sums in
the following theorem are power series.

\begin{theorem}[The four dual identities]\label{thm:dual-main}\label{Th-2}
With $P_1,P_2,P_3$ as in Theorem~\ref{thm:main},
\begin{align}
 \mathfrak N(0,0)&=P_1^2+qP_2P_3,\label{eq:ww1}\\
 \mathfrak N(-1,3)&=P_2^2+P_1P_3,\label{eq:ww2}\\
 \mathfrak N(1,0)&=P_1P_2-qP_3^2,\label{eq:ww3}\\
 \mathfrak N(0,1)&=
 \frac{(q^6;q^9)_\infty}
 {(q,q^2,q^2,q^4,q^5,q^5;q^6)_\infty}.
 \label{eq:li-wang-conjecture}
\end{align}
\end{theorem}

Equations~\eqref{eq:ww1}--\eqref{eq:ww3} settle Conjecture~3.6 of Wang and Wang~\cite{wangwang}; their original product forms are
recovered in Section~7. Equation~\eqref{eq:li-wang-conjecture}
was conjectured by Li and Wang~\cite[Conjecture~6.4]{liwang}. Thus
Theorem~\ref{thm:dual-main} contains three Wang--Wang   identities and
one Li--Wang identity.

The fifth Kanade--Russell identity has a slightly different dual
behavior, since its sum side is itself a linear combination. This leads
to the following additional identity.

\begin{theorem}[A dual companion of the fifth Kanade--Russell identity]
\label{thm:fifth-dual-companion}
For $|q|<1$,
\begin{equation}\label{eq:fifth-dual-companion}
 \mathfrak N(0,2)+q^2\mathfrak N(-2,5)
 =\frac{(q^3;q^9)_\infty}
 {(q;q^3)_\infty^2(q^2;q^3)_\infty}.
\end{equation}
\end{theorem}

We explain precisely the duality that relates these four individual sums to the
first four identities of Theorem~\ref{thm:main}. An ordinary Nahm sum
has the form
\[
 f_{\mathcal A,\boldsymbol b,\kappa}(q)
 =\sum_{\boldsymbol m\in\mathbb Z_{\ge0}^{\ell}}
 \frac{q^{\frac12\boldsymbol m^T\mathcal A\boldsymbol m
                 +\boldsymbol b^T\boldsymbol m+\kappa}}
 {\prod_{i=1}^{\ell}(q;q)_{m_i}},
\]
where $\mathcal A$ is symmetric and positive definite. Zagier's
duality conjecture~\cite{zagier} proposed that modularity is preserved
under the involution
\[
 (\mathcal A,\boldsymbol b,\kappa)\longmapsto
 \left(\mathcal A^{-1},\mathcal A^{-1}\boldsymbol b,
 \tfrac12\boldsymbol b^T\mathcal A^{-1}\boldsymbol b
                  -\tfrac{\ell}{24}-\kappa\right).
\]

Mizuno~\cite{mizuno} generalized the usual Nahm-sum setting to
symmetrizable matrices.  More precisely, he considered generalized
Nahm sums of the form
\begin{equation}
	\widetilde f_{A,\boldsymbol b,\kappa,D}(q)
	=
	\sum_{\boldsymbol m\in\mathbb Z_{\geq0}^{\ell}}
	\frac{
		q^{\frac12\boldsymbol m^{T}AD\boldsymbol m
			+\boldsymbol b^{T}\boldsymbol m+\kappa}
	}{
		\prod_{i=1}^{\ell}(q^{d_i};q^{d_i})_{m_i}
	},
	\qquad
	D=\operatorname{diag}(d_1,\ldots,d_\ell),
	\label{eq:generalized-nahm}
\end{equation}
where the $d_i$ are positive integers and $AD$ is symmetric and
positive definite. 
   The corresponding parameter involution is
\begin{align}
A^*=A^{-1},
\quad
\boldsymbol b^*=A^{-1}\boldsymbol b,
\quad
\kappa^*
=
\frac12
\boldsymbol b^{T}(AD)^{-1}\boldsymbol b
-\frac1{24}\operatorname{tr}D-\kappa,
\quad
D^*=D. \label{1-16}
\end{align}

Recent work of Wang and his collaborators has led to substantial
progress on the explicit evaluation and modularity of low-rank Nahm
sums, including Zagier's rank-two and rank-three examples, tadpole
Nahm sums, and several families of generalized Nahm sums
\cite{cao-rosengren-wang,milas-wang,shi-wang,
	wang-rank-two,wang-rank-three}.
On the duality side, Cao and Wang \cite{cao-wang-lift3,cao-wang-lift4} developed lift-dual constructions
that produce new modular examples, while Shi and Wang \cite{shi-wang-rank3-duals} studied Nahm
sums dual to Zagier's rank-three examples.

Mizuno's examples also motivated further product evaluations by Wang
and Wang~\cite{wang-wang-rank-three,wangwang}, including
the conjectures addressed here.  The duality principle is not valid
in full generality: L.~Wang~\cite{wang-duality} constructed
explicit counterexamples to Zagier's duality expectation and to
Mizuno's generalized duality conjecture.  We therefore prove the four
individual dual identities and the fifth dual companion directly; no
 general preservation of modularity is assumed.

For the sums in this paper,
\[
 \mathcal A_+=\begin{pmatrix}2&1\\3&2\end{pmatrix},\qquad
 \mathcal D=\begin{pmatrix}1&0\\0&3\end{pmatrix},\qquad
 \mathcal A_-:=\mathcal A_+^{-1}
 =\begin{pmatrix}2&-1\\-3&2\end{pmatrix}.
\]
Although $\mathcal A_+$ and $\mathcal A_-$ themselves are not
symmetric, the products
\[
\mathcal A_+D
=
\begin{pmatrix}
	2&3\\
	3&6
\end{pmatrix},
\qquad
\mathcal A_-D
=
\begin{pmatrix}
	2&-3\\
	-3&6
\end{pmatrix}
\]
are symmetric and positive definite, as required in the generalized
Nahm-sum framework.

The two quadratic forms are
$\frac12(r,s)\mathcal A_\pm\mathcal D(r,s)^T
 =r^2\pm3rs+3s^2$, and the linear parameters transform as
\[
 (a,b)^T\longmapsto(2a-b,-3a+2b)^T.
\]
Consequently the four correspondences, with the scalar normalizing
power $q^\kappa$ suppressed, are
\begin{equation}\label{eq:dual-correspondence}
 \begin{array}{c|c|c}
 \text{Kanade--Russell sum}&\text{dual sum}&\text{dual identity}\\ \hline
 S(0,0)&\mathfrak N(0,0)&\eqref{eq:ww1}\\
 S(1,3)&\mathfrak N(-1,3)&\eqref{eq:ww2}\\
 S(2,3)&\mathfrak N(1,0)&\eqref{eq:ww3}\\
 S(1,2)&\mathfrak N(0,1)&\eqref{eq:li-wang-conjecture}
 \end{array}
\end{equation}

The fifth identity is slightly different because its sum side is the
linear combination $S(1,4)+qS(2,4)$.  The linear parameters transform as
\[
 (1,4)^T\longmapsto(-2,5)^T,
 \qquad
 (2,4)^T\longmapsto(0,2)^T.
\]
The scalar parameter in \eqref{1-16} explains the
relative power of $q$.  For $S(1,4)$ we take $\kappa=0$ and obtain
$\kappa^*=13/6$, whereas the term $qS(2,4)$ has $\kappa=1$ and gives
$\kappa^*=1/6$.  After suppressing the common factor $q^{1/6}$, the
natural dual combination is therefore
\[
 \mathfrak N(0,2)+q^2\mathfrak N(-2,5).
\]
Theorem~\ref{thm:fifth-dual-companion}, proved in
Section~\ref{sec:fifth-dual}, gives its product evaluation.  Thus the
fifth identity also has a concrete dual companion, although unlike the
first four correspondences it is naturally expressed as a linear
combination of two generalized Nahm sums.

Following Sun and Wang \cite{sun-wang}, let $X$ and $Y$ be Dynkin diagrams from the
families under consideration, and let $C(X)$ denote the Cartan matrix
associated with $X$.  We write $D(X)$ for the diagonal symmetrizing
matrix, normalized so that its diagonal entries are positive integers
with greatest common divisor one and
$
C(X)D(X)
$
is symmetric.  If $r(X)$ denotes the rank of $X$, then $C(X)$ and
$D(X)$ are $r(X)\times r(X)$ matrices.  For a pair $(X,Y)$, Sun and
Wang associate the generalized Nahm-sum data
\[
\mathcal A(X,Y):=C(X)\otimes C(Y)^{-1},
\qquad
D(X,Y):=D(X)\otimes D(Y),
\]
where $\otimes$ denotes the Kronecker product.  
For the tadpole diagram $T_1$ and the $G_2$ diagram one has
\[
 C(T_1)=(1),\qquad D(T_1)=(1),\qquad
 C(G_2)=\begin{pmatrix}2&-1\\-3&2\end{pmatrix},\qquad
 D(G_2)=\begin{pmatrix}1&0\\0&3\end{pmatrix}.
\]
Consequently,
\[
 \mathcal A(T_1,G_2)=C(G_2)^{-1}=\mathcal A_+,\qquad
 \mathcal A(G_2,T_1)=C(G_2)=\mathcal A_-,
\]
while in both cases the diagonal matrix is precisely $\mathcal D$.
Thus the exchange of the two Dynkin diagrams realizes exactly the
matrix inversion $\mathcal A_+\leftrightarrow\mathcal A_-$ and hence
the passage from the positive mixed term $r^2+3rs+3s^2$ to the
negative mixed term $r^2-3rs+3s^2$ used in this paper.

 Sun and Wang~\cite{sun-wang} recorded seven rank-at-most-three cases
 whose modularity had not yet been established:
 \[
 (T_1,G_2),\quad (G_2,T_1),\quad (T_3,A_1),\quad
 (C_3,A_1),\quad (C_3,T_1),\quad (B_3,A_1),\quad (B_3,T_1).
 \]
 Subsequently, Shi and Wang~\cite{shi-wang-rank3-duals} proved the
 $(T_3,A_1)$ case, leaving six unresolved cases.  Among these, Sun and
 Wang \cite{sun-wang} had observed that the modularity of $(T_1,G_2)$ would follow from
 the first modulo nine Kanade--Russell identity, whereas that of
 $(G_2,T_1)$ would follow from the corresponding Wang--Wang dual
 identity.  These are precisely \eqref{eq:kr1} and \eqref{eq:ww1},
 respectively.  The present paper proves both required identities and
 therefore establishes the modularity of these two cases
 unconditionally.  Relative to the status recorded in these works, the
 remaining unresolved cases are
 \[
 (C_3,A_1),\qquad (C_3,T_1),\qquad
 (B_3,A_1),\qquad (B_3,T_1).
 \]
 
 Moreover, the Dynkin-diagram conjecture of Sun and Wang is formulated
 at zero linear parameter, whereas
 \eqref{eq:dual-correspondence} contains, in addition to the
 $\boldsymbol b=0$ pair, three nonzero individual linear-parameter correspondences
 under
 \[
 \boldsymbol b\longmapsto \mathcal A_+^{-1}\boldsymbol b.
 \]
 Thus the present results extend this particular
 $(T_1,G_2)\leftrightarrow(G_2,T_1)$ correspondence beyond the
 zero-$\boldsymbol b$ case.  The fifth dual companion is likewise
 compatible with the same parameter involution, with the scalar
 normalization accounting for the relative factor $q^2$. 
 Thus the present results establish the required $q$-series identities
 and modularity statements.  We do not address the further problem of
 identifying the corresponding rational conformal field theories.  

Our proofs are based on explicit recurrences and boundary uniqueness.
For the five Kanade--Russell identities, we construct product-side
sequences from ordinary divided differences of quadratic-factor
kernels.  Finite local residue sums and explicit algebraic
certificates establish the required recurrences, while coefficientwise
boundary estimates identify the sum-side and product-side solutions.
For the first three dual identities, two row-vector certificates
produce bilinear relations between the positive- and
negative-mixed-term Nahm sums.  For the fourth dual identity, the
finite $\tau=1$ certificate for the fourth Kanade--Russell identity is
transferred to a convergent direct-product kernel.  The same transfer
applied to the $\tau=2$ certificate proves the dual companion of the
fifth identity.  These arguments are direct and do not rely on any
general principle asserting that Nahm-sum duality preserves modularity.

The paper is organized as follows.  
Section~\ref{sec:preliminaries} collects the common lemmas.
Sections~\ref{sec:first-three}--\ref{sec:first-proof} prove the first
three Kanade--Russell identities, and
Sections~\ref{sec:cyclotomic}--\ref{sec:cyclotomic-proof} prove the
fourth and fifth. Section~\ref{sec:wang-wang} establishes the three
Wang--Wang dual identities of Conjecture~3.6, and
Section~\ref{sec:li-wang} proves the Li--Wang identity.
Section~\ref{sec:fifth-dual} proves the dual companion of the fifth
Kanade--Russell identity by transferring the $\tau=2$ certificate.
Section~10 contains concluding remarks, while
Appendix~\ref{app:coefficients} records the explicit coefficient data
for the cyclotomic certificates.

\section{Preliminaries}\label{sec:preliminaries}

%We use the empty-product convention $(u;q^d)_0=1$. Elementary cancellation
%gives
%\begin{equation}\label{eq:finite-product-shift}
% \begin{aligned}
% (u;q^d)_{n+1}&=(1-uq^{dn})(u;q^d)_n,\\
% (u;q^d)_\infty&=(u;q^d)_n(uq^{dn};q^d)_\infty,\\
% \frac{(u;q^d)_m}{(uq^{-dh};q^d)_m}
% &=\frac{(uq^{d(m-h)};q^d)_h}{(uq^{-dh};q^d)_h}
% \qquad(m,h\ge0).
% \end{aligned}
%\end{equation}
%The last equality is a rational identity, including $m<h$: cross
%multiplication gives the same product of $m+h$ consecutive factors.
%These formulas will be used to derive all kernel quotients, including
%the small orders at which shifted strings overlap.

For a formal Laurent series $f$, let $\operatorname{ord}_q f$ denote
its least exponent, with $\operatorname{ord}_q0=+\infty$. The notation
$f=O(q^L)$ means $f\in q^L\mathbb C[[q]]$. Here and throughout, $\mathbb C$ is the set 
 of complex numbers and  $\mathbb C[[q]]$ 
 denotes the ring of formal power series in \(q\) with complex coefficients.   For vectors it applies
componentwise. 

For every fixed pair of integers $a,b$, both quadratic polynomials 
$
r^2+3rs+3s^2+ar+bs$ and $
r^2-3rs+3s^2+ar+bs$ 
are bounded below on $\mathbb Z_{\geq0}^2$ and tend to $+\infty$
as $r+s\to\infty$.  Consequently, for fixed $a,b$, the sums
$S(a,b)$ and $\mathfrak N(a,b)$ are well-defined formal Laurent
series: only finitely many pairs $(r,s)$ can contribute to any
prescribed power of $q$.  In all instances of $S(a,b)$ used below,
the parameters are nonnegative, so these series are in
$\mathbb C[[q]]$.  Some dual sums involve negative linear parameters
and may initially be Laurent series; for example,
$\mathfrak N(-2,5)$ has lowest $q$-degree $-1$.

For every fixed $0<\rho<1$, the same positive-definite quadratic
growth, together with the lower bounds for
$(q;q)_r$ and $(q^3;q^3)_s$, gives absolute and locally uniform
convergence on compact subsets of the punctured disk
$0<|q|<1$.  Whenever the resulting series has no negative powers,
it extends analytically to $q=0$.

\begin{lemma}\label{lem:contiguous}
  For all nonnegative integers \(a\) and \(b\), direct
index shifts give
 \begin{align}
 S(a,b)-S(a+1,b)&=q^{a+1}S(a+2,b+3), \label{2-1}\\
 S(a,b)-S(a,b+3)&=q^{b+3}S(a+3,b+6). \label{2-2}
 \end{align}                                                    
\end{lemma}
\begin{proof}
Note that 
\begin{align}
	 S(a,b)-S(a+1,b)&=\sum_{r,s\ge0}
	\frac{q^{r^2+3rs+3s^2+ar+bs}  }
	{(q;q)_r(q^3;q^3)_s}-\sum_{r,s\ge0}
	\frac{q^{r^2+3rs+3s^2+(a+1)r+bs}  }
	{(q;q)_r(q^3;q^3)_s} \nonumber\\
	&=\sum_{r\geq 1,\ s\ge0}
	\frac{q^{r^2+3rs+3s^2+ar+bs} (1-q^r) }
	{(q;q)_r(q^3;q^3)_s} \nonumber\\
	&=\sum_{r\geq 1,\ s\ge0}
	\frac{q^{r^2+3rs+3s^2+ar+bs}   }
	{(q;q)_{r-1}(q^3;q^3)_s}
	\nonumber\\
	&=\sum_{r, s\ge0}
	\frac{q^{(r+1)^2+3(r+1)s+3s^2+a(r+1)+bs}   }
	{(q;q)_{r}(q^3;q^3)_s}  . \label{2-3}
\end{align}	
Combining \eqref{2-3} with 
\[
 (r+1)^2+3(r+1)s+3s^2+a(r+1)+bs
 =r^2+3rs+3s^2+(a+2)r+(b+3)s+a+1 
\]
 gives \eqref{2-1}.

Similarly,
\begin{align}
	S(a,b)-S(a,b+3)&=\sum_{r,s\ge0}
	\frac{q^{r^2+3rs+3s^2+ar+bs}  }
	{(q;q)_r(q^3;q^3)_s}-\sum_{r,s\ge0}
	\frac{q^{r^2+3rs+3s^2+ar+(b+3)s}  }
	{(q;q)_r(q^3;q^3)_s} \nonumber\\
	&=\sum_{r\geq 0,\ s\ge1}
	\frac{q^{r^2+3rs+3s^2+ar+bs} (1-q^{3s}) }
	{(q;q)_r(q^3;q^3)_s} \nonumber\\
	&=\sum_{r\geq 0,\ s\ge1}
	\frac{q^{r^2+3rs+3s^2+ar+bs}   }
	{(q;q)_{r}(q^3;q^3)_{s-1}}
	\nonumber\\
	&=\sum_{r, s\ge0}
	\frac{q^{r^2+3r(s+1)+3(s+1)^2+ar+b(s+1)}   }
	{(q;q)_{r}(q^3;q^3)_s}  . \label{2-4}
\end{align}	
By \eqref{2-4} and the following identity, 
\[
 r^2+3r(s+1)+3(s+1)^2+ar+b(s+1)
 =r^2+3rs+3s^2+(a+3)r+(b+6)s+b+3,
\]
we obtain \eqref{2-2}. 
\end{proof}

For distinct nodes $z_0,\ldots,z_n$, the ordinary divided difference is
\begin{equation}\label{eq:dd-definition}
 f[z_0,\ldots,z_n]:
 =\sum_{i=0}^n\frac{f(z_i)}{\prod_{j\ne i}(z_i-z_j)}.
\end{equation}
It is symmetric in the nodes and satisfies the product rule
\begin{equation}\label{eq:dd-product-rule}
 (fg)[z_0,\ldots,z_n]
 =\sum_{i=0}^n f[z_0,\ldots,z_i]g[z_i,\ldots,z_n],
\end{equation}
which  will be used in the proof of Lemma \ref{lem:dd-expansion}.

These formulas follow from polynomial interpolation, or by induction
from the recursive definition of divided differences; see \cite{deboor}.
For the affine reciprocal factor
\begin{align}
f_j(z)=\frac{1}{1-zq^{dj}+q^{c+2dj}},\label{a-1}
\end{align} 
subtraction of reciprocals and
induction give
\begin{equation}\label{eq:one-factor}
 f_j[z_i,\ldots,z_{i+s}]
 =\frac{q^{djs}}
 {\prod_{\ell=i}^{i+s}(1-z_\ell q^{dj}+q^{c+2dj})}.
\end{equation}

Throughout this paper, define 
\begin{equation}\label{eq:kernel}
	\Phi_{d,c}(z):=\prod_{j\ge0}
	(1-zq^{dj}+q^{c+2dj})^{-1},
\end{equation}
where  \(c\) and \(d\) are positive integers.  
Indeed,
\[
1-\left(t+\frac{q^c}{t}\right)q^{dj}+q^{c+2dj}
=(1-tq^{dj})\left(1-\frac{q^{c+dj}}{t}\right),
\]
and hence
\begin{align}\label{2-9}
\Phi_{d,c}(t+q^c/t)=\frac1{(t,q^c/t;q^d)_\infty}.
\end{align}

\begin{lemma} \label{lem:dd-expansion}
Let $c,d$ be positive integers. Suppose that
$\alpha_i=\zeta_iq^{r_i}$ and $\beta_i=\zeta_i^{-1}q^{c-r_i}$,
where $\zeta_i\in\mathbb C\setminus \{0\}$,  and  $r_i$ and $c-r_i$ are positive  integers. 
  Suppose that the nodes
$z_i=\alpha_i+\beta_i$ are distinct. Set
$d_i(j)=(1-\alpha_iq^{dj})(1-\beta_iq^{dj})$.
For $n\ge1$,
\begin{equation}\label{eq:dd-expansion}
 \begin{aligned}
 \Phi_{d,c}[z_0,\ldots,z_n]
 ={}&\sum_{0\le k_1\le\cdots\le k_n}
 q^{d(k_1+\cdots+k_n)}
 \left(\prod_{j=0}^{k_1}d_0(j)^{-1}\right)\\
 &\quad\times\left(\prod_{i=1}^{n-1}
             \prod_{j=k_i}^{k_{i+1}}d_i(j)^{-1}\right)
 \left(\prod_{j=k_n}^{\infty}d_n(j)^{-1}\right).
 \end{aligned}
\end{equation}
For $n=0$, the expression is the single evaluated infinite product.
In either case,
\begin{equation}\label{eq:dd-leading}
 \Phi_{d,c}[z_0,\ldots,z_n]
 =\frac1{(q^d;q^d)_n}+O(q^L),
\end{equation}
where 
$$
L:=\min\limits_{0\leq i \leq n}\{r_i,c-r_i\} .
$$
\end{lemma}

\begin{proof}
	All limits in this proof are initially interpreted coefficientwise as
	formal power series in \(q\), and all order estimates refer to the
	\(q\)-adic order. 
	 Define 
	\[
	F_N(z):=\prod_{j=0}^{N}f_j(z),
	\]
	where $f_j(z)$ is defined by \eqref{a-1}. 
	Since \(\alpha_i\beta_i=q^c\) and \(z_i=\alpha_i+\beta_i\),
	\begin{align*}
		1-z_iq^{dj}+q^{c+2dj}
		&=1-(\alpha_i+\beta_i)q^{dj}
		+\alpha_i\beta_i q^{2dj}\nonumber \\
		&=(1-\alpha_iq^{dj})(1-\beta_iq^{dj})
		=d_i(j).
	\end{align*}
	In particular, \(f_j(z_i)=d_i(j)^{-1}\). Moreover,
	\[
	d_i(j)^{-1}
	=
	\frac{1}{
		(1-\zeta_iq^{r_i+dj})
		(1-\zeta_i^{-1}q^{c-r_i+dj})}
	=1+O(q^{L+dj}).
	\]
	Consequently, all evaluated factors belong to
	\(\mathbb{C}[[q]]\) and have constant term one, and their
	infinite products are well defined coefficientwise.
	
	We first establish the finite product identity. The divided difference
	product rule \eqref{eq:dd-product-rule} gives, for \(N\ge1\),
	\[
	F_N[z_0,\ldots,z_n]
	=
	\sum_{r=0}^{n}
	F_{N-1}[z_0,\ldots,z_r]\,
	f_N[z_r,\ldots,z_n].
	\]
	Iterating this identity yields
	\[
	F_N[z_0,\ldots,z_n]
	=
	\sum_{0=m_0\le m_1\le\cdots\le m_{N+1}=n}
	\prod_{j=0}^{N}
	f_j[z_{m_j},\ldots,z_{m_{j+1}}].
	\]
	Here the \(j\)-th factor appears as
	$
	f_j[z_{m_j},\ldots,z_{m_{j+1}}],
$
	which is a divided difference of order \(m_{j+1}-m_j\).
  Formula \eqref{eq:one-factor} gives
	\[
	f_j[z_{m_j},\ldots,z_{m_{j+1}}]
	=
	\frac{q^{dj(m_{j+1}-m_j)}}
	{\displaystyle\prod_{\ell=m_j}^{m_{j+1}}d_\ell(j)}.
	\]
	This formula includes \(m_j=m_{j+1}\), when the divided
	difference is simply the value \(f_j(z_{m_j})\). Hence
	\[
	F_N[z_0,\ldots,z_n]
	=
	\sum_{0=m_0\le\cdots\le m_{N+1}=n}
	q^{d\sum_{j=0}^{N}j(m_{j+1}-m_j)}
	\prod_{j=0}^{N}
	\prod_{\ell=m_j}^{m_{j+1}}d_\ell(j)^{-1}.
	\]
	
We first prove the assertion for \(n\ge1\); the case \(n=0\) will be
treated separately at the end of the proof.  To each sequence
	\((m_0,\ldots,m_{N+1})\), associate the weakly increasing
	sequence
	\[
	0\le k_1\le\cdots\le k_n\le N
	\]
	in which the integer \(j\) occurs exactly \(m_{j+1}-m_j\)
	times. Equivalently,
	\[
	k_r=j
	\quad\Longleftrightarrow\quad
	m_j<r\le m_{j+1}.
	\]
	This correspondence is bijective: its inverse is
	\[
	m_j=\#\{r:1\le r\le n,\ k_r<j\},
	\qquad 0\le j\le N+1.
	\]
	It follows immediately that
	\[
	\sum_{j=0}^{N}j(m_{j+1}-m_j)
	=\sum_{r=1}^{n}k_r.
	\]
	For the purpose of describing the denominator, set
	\(k_0=0\) and \(k_{n+1}=N\). For
	\(0\le\ell\le n\) and \(0\le j\le N\), we have
	\[
	m_j\le\ell\le m_{j+1}
	\quad\Longleftrightarrow\quad
	k_\ell\le j\le k_{\ell+1}.
	\]
	Indeed, the first inequality says that at most \(\ell\)
	entries of the sequence are less than \(j\), whereas the
	second says that at least \(\ell\) entries are at most \(j\).
	The endpoint cases are covered by the conventions for
	\(k_0\) and \(k_{n+1}\). Therefore
	\[
	\prod_{j=0}^{N}
	\prod_{\ell=m_j}^{m_{j+1}}d_\ell(j)^{-1}
	=
	\prod_{\ell=0}^{n}
	\prod_{j=k_\ell}^{k_{\ell+1}}d_\ell(j)^{-1}.
	\]
	In particular, repeated values among the \(k_r\) are allowed;
	they represent   several divided difference operations applied to the same factor. 
	
	For \(\mathbf{k}=(k_1,\ldots,k_n)\) with \(k_n\le N\), define
	\[
	\begin{aligned}
	G_N(\mathbf{k}):
		={}&
		\left(\prod_{j=0}^{k_1}d_0(j)^{-1}\right)
		\left(\prod_{i=1}^{n-1}
		\prod_{j=k_i}^{k_{i+1}}d_i(j)^{-1}\right) 
		\left(\prod_{j=k_n}^{N}d_n(j)^{-1}\right).
	\end{aligned}
	\]
	We have proved the exact finite identity
	\[
	F_N[z_0,\ldots,z_n]
	=
	\sum_{0\le k_1\le\cdots\le k_n\le N}
	q^{d(k_1+\cdots+k_n)}G_N(\mathbf{k}).
	\]
	
	We next justify passage to the infinite product, keeping
	\(n\), \(c\), \(d\), and the nodes fixed. For every \(i\),
	\[
	\prod_{j=N+1}^{\infty}d_i(j)^{-1}
	=1+O(q^{L+d(N+1)}),
	\]
	and therefore
	\[
	\begin{aligned}
		\Phi_{d,c}(z_i)-F_N(z_i)
		&=
		F_N(z_i)
		\left(
		\prod_{j=N+1}^{\infty}d_i(j)^{-1}-1
		\right)\\
		&=O(q^{L+d(N+1)}).
	\end{aligned}
	\]
	Since the nodes are distinct, the Laurent-series denominators
	\[
	V_i=\prod_{\substack{0\le\ell\le n\\\ell\ne i}}
	(z_i-z_\ell)
	\]
	are nonzero. Set
	\[
	v:=\max_{0\le i\le n}\{\operatorname{ord}_q V_i \}.
	\]
	The Lagrange formula for divided differences now gives
	\[
	\begin{aligned}
		\Phi_{d,c}[z_0,\ldots,z_n]-F_N[z_0,\ldots,z_n]
		&=
		\sum_{i=0}^{n}
		\frac{\Phi_{d,c}(z_i)-F_N(z_i)}{V_i}\\
		&=O(q^{L+d(N+1)-v}).
	\end{aligned}
	\]
	Since \(v\) is independent of \(N\), the last order tends
	to infinity. Thus the left-hand side of the finite identity
	converges coefficientwise to the required divided difference.
	
	For an arbitrary weakly increasing tuple \(\mathbf{k}\), put
	\[
	\begin{aligned}
		G(\mathbf{k}):
		={}&
		\left(\prod_{j=0}^{k_1}d_0(j)^{-1}\right)
		\left(\prod_{i=1}^{n-1}
		\prod_{j=k_i}^{k_{i+1}}d_i(j)^{-1}\right)
		\left(\prod_{j=k_n}^{\infty}d_n(j)^{-1}\right).
	\end{aligned}
	\]
	Every \(G(\mathbf{k})\) belongs to \(\mathbb{C}[[q]]\).
	For a fixed nonnegative integer \(M\), a summand can
	contribute to the coefficient of \(q^M\) only if
	\[
	d(k_1+\cdots+k_n)\le M.
	\]
	There are only finitely many such tuples. Hence
	\[
	\mathcal{S}
	=
	\sum_{0\le k_1\le\cdots\le k_n}
	q^{d(k_1+\cdots+k_n)}G(\mathbf{k})
	\]
	is a well-defined formal power series.
	
	Write
	\[
	\mathcal{S}_N
	=
	\sum_{0\le k_1\le\cdots\le k_n\le N}
	q^{d(k_1+\cdots+k_n)}G_N(\mathbf{k}).
	\]
	If \(k_n\le N\), then
	\[
	\begin{aligned}
		G(\mathbf{k})-G_N(\mathbf{k})
		&=G_N(\mathbf{k})
		\left(
		\prod_{j=N+1}^{\infty}d_n(j)^{-1}-1
		\right)\\
		&=O(q^{L+d(N+1)}).
	\end{aligned}
	\]
	If \(k_n>N\), then
	\[
	\operatorname{ord}_q
	\left(q^{d(k_1+\cdots+k_n)}G(\mathbf{k})\right)
	\ge d(N+1).
	\]
	Consequently,
	\[
	\begin{aligned}
		\mathcal{S}-\mathcal{S}_N
		={}&
		\sum_{0\le k_1\le\cdots\le k_n\le N}
		q^{d(k_1+\cdots+k_n)}
		\bigl(G(\mathbf{k})-G_N(\mathbf{k})\bigr)\\
		&+
		\sum_{\substack{0\le k_1\le\cdots\le k_n\\k_n>N}}
		q^{d(k_1+\cdots+k_n)}G(\mathbf{k})\\
		={}&O(q^{L+d(N+1)})+O(q^{d(N+1)})\\
		={}&O(q^{d(N+1)}).
	\end{aligned}
	\]
	Thus the right-hand side also converges coefficientwise.
	Taking limits in the finite identity proves
	\[
	\Phi_{d,c}[z_0,\ldots,z_n]
	=
	\sum_{0\le k_1\le\cdots\le k_n}
	q^{d(k_1+\cdots+k_n)}G(\mathbf{k}),
	\]
	which is precisely \eqref{eq:dd-expansion}. This representation also
	shows that all negative Laurent coefficients in the
	Lagrange expression cancel.
	
	It remains to prove \eqref{eq:dd-leading}. Since each factor is
	\(1+O(q^{L+dj})\), the finite and infinite products defining
	\(G(\mathbf{k})\) satisfy
	\[
	G(\mathbf{k})=1+O(q^L).
	\]
	Using the coefficientwise finiteness established above, we obtain
	\[
	\begin{aligned}
		\Phi_{d,c}[z_0,\ldots,z_n]
		 -
		\sum_{0\le k_1\le\cdots\le k_n}
		q^{d(k_1+\cdots+k_n)}&
		=
		\sum_{0\le k_1\le\cdots\le k_n}
		q^{d(k_1+\cdots+k_n)}
		\bigl(G(\mathbf{k})-1\bigr)\\
		&
		=O(q^L).
	\end{aligned}
	\]
	To evaluate the remaining sum, set
	\[
	h_1=k_1,\qquad h_r=k_r-k_{r-1}\quad(2\le r\le n).
	\]
	This is a bijection with \(h_1,\ldots,h_n\ge0\), and
	\[
	k_1+\cdots+k_n
	=nh_1+(n-1)h_2+\cdots+h_n.
	\]
	Hence
	\[
	\begin{aligned}
		\sum_{0\le k_1\le\cdots\le k_n}
		q^{d(k_1+\cdots+k_n)}
		&=
		\sum_{h_1,\ldots,h_n\ge0}
		q^{d(nh_1+(n-1)h_2+\cdots+h_n)}\\
		&=
		\prod_{r=1}^{n}
		\left(\sum_{h\ge0}q^{drh}\right)\\
		&=
		\prod_{r=1}^{n}\frac{1}{1-q^{dr}}
		=\frac{1}{(q^d;q^d)_n}.
	\end{aligned}
	\]
	It follows that
	\[
	\Phi_{d,c}[z_0,\ldots,z_n]
	=\frac{1}{(q^d;q^d)_n}+O(q^L).
	\]
	
	Finally, when \(n=0\), the divided difference is the
	single evaluated product:
	\[
	\Phi_{d,c}[z_0]
	=\Phi_{d,c}(z_0)
	=\prod_{j=0}^{\infty}d_0(j)^{-1}
	=1+O(q^L)
	=\frac{1}{(q^d;q^d)_0}+O(q^L).
	\]
	This proves the asserted estimate also at order zero.
	
	For fixed \(0<|q|<1\), provided that the evaluated denominators and
	the node differences are nonzero, the expansion is also valid
	analytically. Indeed, the factors \(G_N(\mathbf{k})\) and
	\(G(\mathbf{k})\) are uniformly bounded, while
	\[
	\sum_{0\le k_1\le\cdots\le k_n}
	|q|^{d(k_1+\cdots+k_n)}
	=
	\prod_{r=1}^{n}\frac{1}{1-|q|^{dr}}
	<\infty.
	\]
	The result therefore follows by dominated convergence.
	\end{proof}

The notation $\operatorname{Res}_{t=t_0}f(t)$ means the coefficient of
$(t-t_0)^{-1}$ in the Laurent expansion of $f$.
If $z=z(t)$ is locally analytic and $z'(t_0)\ne0$, then
\begin{equation}\label{eq:residue-change}
 \operatorname{Res}_{t=t_0}\bigl(H(z(t))z'(t)\bigr)
 =\operatorname{Res}_{z=z(t_0)}H(z).
\end{equation}
In particular, if
\[
H(z)=\frac{\Phi_{d,c}(z)}{\prod_i(z-z_i)},
\]
then the pullback differential
\[
H(z(t))z'(t)\,dt
\]
has, at each simple preimage of a node, the corresponding
Lagrange summand as its residue.

\begin{lemma} \label{lem:finite-residue}
Fix $0<q<1$ and an integer $d\ge1$. Let $\Lambda\subset \mathbb C\setminus \{0\}$
be invariant under multiplication by $q^d$ and its inverse. Suppose
that $\Psi$ is meromorphic on $  \mathbb C\setminus \{0\}$, has no pole at a
point of $\Lambda$, and that
$\sigma(t)=-q^d\Psi(q^dt)/\Psi(t)$ is rational. For rational $R(t)$, put
\[
 \mathcal L_\Lambda(R):=\frac12\sum_{t_0\in\Lambda}
           \operatorname{Res}_{t=t_0}\bigl(\Psi(t)R(t)\bigr).
\]
Then the sum has finite support and
\begin{equation}\label{eq:general-residue-shift}
 \mathcal L_\Lambda\bigl(R(t)+\sigma(t)R(q^dt)\bigr)=0.
\end{equation}
\end{lemma}

\begin{proof}
	Fix a rational function \(R(t)\). We first verify that the
	sum defining \(\mathcal{L}_{\Lambda}(R)\) has finite support.
	Let
	\[
	P_R=\{t_0\in\Lambda:R \text{ has a pole at } t_0\}.
	\]
	Since \(R\) is rational, the set \(P_R\) is finite. At every
	point \(t_0\in\Lambda\setminus P_R\), both \(R\) and \(\Psi\)
	are holomorphic. Their product is therefore holomorphic
	there, and
	\[
	\operatorname{Res}_{t=t_0}\bigl(\Psi(t)R(t)\bigr)=0.
	\]
	Consequently,
	\[
	\mathcal{L}_{\Lambda}(R)
	=
	\frac12\sum_{t_0\in P_R}
	\operatorname{Res}_{t=t_0}\bigl(\Psi(t)R(t)\bigr),
	\]
	which is a finite sum. The same argument applies to every
	rational function, so \(\mathcal{L}_{\Lambda}\) is a
	well-defined linear functional on the space of rational
	functions.
	
	Define
	\[
	F(t):=\Psi(t)R(t),
	\qquad
	\widetilde R(t):=\sigma(t)R(q^dt).
	\]
	Because \(\sigma\) is rational, \(\widetilde R\) is rational
	as well. Hence the residue sums associated with
	\(\widetilde R\) and \(R+\widetilde R\) also have finite
	support. By the definition of \(\sigma\),
	\[
		\Psi(t)\widetilde R(t)
		=\Psi(t)\sigma(t)R(q^dt)
		=-q^d\Psi(q^dt)R(q^dt)
		=-q^dF(q^dt).
	\]
	This is an identity of meromorphic functions; in particular,
	it remains valid meromorphically at zeros of \(\Psi\).
	
	We now verify explicitly how the local residue changes
	under \(u=q^dt\). Fix \(t_0\in\Lambda\), and put
	$
	u_0=q^dt_0.
$
	The invariance of \(\Lambda\) implies that \(u_0\in\Lambda\).
	Since \(F\) is meromorphic near \(u_0\), it has a convergent
	Laurent expansion
	\[
	F(u)=\sum_{\nu=-M}^{\infty}a_\nu(u-u_0)^\nu,
	\qquad 0<|u-u_0|<\varepsilon,
	\]
	for some integer \(M\ge0\). Here \(a_{-1}\) is understood
	to be zero if \(M=0\). As
	\[
	q^dt-u_0=q^d(t-t_0),
	\]
	substitution gives
	\[
	\begin{aligned}
		q^dF(q^dt)
		&=q^d\sum_{\nu=-M}^{\infty}
		a_\nu\bigl(q^d(t-t_0)\bigr)^\nu=\sum_{\nu=-M}^{\infty}
		a_\nu q^{d(\nu+1)}(t-t_0)^\nu.
	\end{aligned}
	\]
	The coefficient of \((t-t_0)^{-1}\) in this expansion is
	\(a_{-1}q^{d(-1+1)}=a_{-1}\). Therefore
	\[
	\operatorname{Res}_{t=t_0}\bigl(q^dF(q^dt)\bigr)
	=
	a_{-1}
	=
	\operatorname{Res}_{u=u_0}F(u).
	\]
	Combining this equality with the meromorphic identity above
	yields
	\[
	\operatorname{Res}_{t=t_0}
	\bigl(\Psi(t)\widetilde R(t)\bigr)
	=
	-\operatorname{Res}_{u=q^dt_0}F(u).
	\]
	
	The assumptions on \(\Lambda\) give
$
	q^d\Lambda=\Lambda,
$
	so the map \(t_0\mapsto q^dt_0\) is a bijection from
	\(\Lambda\) to itself, with inverse \(u_0\mapsto q^{-d}u_0\).
	We may therefore reindex the finite-support residue sum:
	\[
	\begin{aligned}
		\mathcal{L}_{\Lambda}(\widetilde R)
		&=\frac12\sum_{t_0\in\Lambda}
		\operatorname{Res}_{t=t_0}
		\bigl(\Psi(t)\widetilde R(t)\bigr)\\
		&=-\frac12\sum_{t_0\in\Lambda}
		\operatorname{Res}_{u=q^dt_0}F(u)\\
		&=-\frac12\sum_{u_0\in\Lambda}
		\operatorname{Res}_{u=u_0}F(u)\\
		&=-\frac12\sum_{u_0\in\Lambda}
		\operatorname{Res}_{u=u_0}\bigl(\Psi(u)R(u)\bigr)\\
		&=-\mathcal{L}_{\Lambda}(R).
	\end{aligned}
	\]
	Finally, linearity gives
	\[
	\begin{aligned}
		\mathcal{L}_{\Lambda}
		\bigl(R(t)+\sigma(t)R(q^dt)\bigr)
		&=\mathcal{L}_{\Lambda}(R)
		+\mathcal{L}_{\Lambda}(\widetilde R)=0,
	\end{aligned}
	\]
	which proves \eqref{eq:general-residue-shift}.
\end{proof}

%\begin{proof}
%Only the finitely many nonzero finite poles of $R$ can contribute to
%the first sum. The second integrand is
%$-q^d\Psi(q^dt)R(q^dt)$. The substitution $u=q^dt$ identifies its
%residue at $t_0$ with the negative of the first integrand's residue at
%$q^dt_0$. Invariance of $\Lambda$ permits reindexing the finite sum.
%The argument allows poles of arbitrary finite order. No limit of an
%infinite contour integral is involved.
%\end{proof}

\begin{lemma}\label{lem:boundary-unique}
Suppose two sequences of $m$-component formal power-series vectors
$\mathbf U_n,\mathbf V_n$ satisfy the same recurrence
\[
 \mathbf U_n=M_n\mathbf U_{n+1},\qquad
 \mathbf V_n=M_n\mathbf V_{n+1},
 \qquad M_n\in\operatorname{Mat}_m(\mathbb C[[q]]),
\]
where  
 $ \operatorname{Mat}_m\bigl(\mathbb C[[q]]\bigr) $
 denotes  the ring of all \(m\times m\) matrices whose entries belong to $\mathbb C[[q]]$. 
If $\mathbf U_n-\mathbf V_n=O(q^{e_n})$ with $e_n\to\infty$, then
$\mathbf U_n=\mathbf V_n$ for every $n$.
\end{lemma}
\begin{proof}
For fixed $n$ and every $N>n$,
\[
 \mathbf U_n-\mathbf V_n
 =M_nM_{n+1}\cdots M_{N-1}(\mathbf U_N-\mathbf V_N).
\]
Multiplication by these matrices cannot lower formal order. The left
side therefore belongs to $q^{e_N}\mathbb C[[q]]^m$ for every $N$.
Every coefficient vanishes by choosing $N$ sufficiently large.
\end{proof}

\begin{remark}
When $\sum_{j=0}^k a_{n,j}u_{n+j}=0$ has all coefficients in
$\mathbb C[[q]]$ and $a_{n,0}$ is a unit, its companion matrix also
has entries in that ring. The lemma consequently applies to such
scalar recurrences. In the residue arguments we first fix real
$0<q<1$. For each fixed order, the Lagrange representation is meromorphic
near $q=0$ with at most a finite-order pole. Lemma~\ref{lem:dd-expansion}
removes its negative Laurent coefficients. Analytic equality for real
$q$ then yields equality of Taylor series, making the formal uniqueness
argument applicable. 
\end{remark}

\section{The first three identities: sum and product sequences}\label{sec:first-three}

\subsection{Recurrences of the   three  sums }

Define
\[
 U_n:=S(3n,6n),\quad V_n:=S(3n+1,6n+3),\quad W_n:=S(3n+2,6n+3).
\]
Applying \eqref{2-1} and \eqref{2-2} gives
\begin{equation}
 \begin{aligned}
 V_n-W_n&=q^{3n+2}U_{n+1},\\
 (1+q^{3n+1})V_n-U_n&=q^{9n+7}W_{n+1},\\
 W_n&=U_{n+1}+q^{3n+3}V_{n+1}+q^{6n+6}W_{n+1}.
 \end{aligned}                                                    \label{a:eq:3}
\end{equation}
Indeed, setting \((a,b)=(3n+1,6n+3)\) in~(2.1) gives the
first equation of~(3.1). Next, by~(2.1) and~(2.2),
\[
\begin{aligned}
	U_n-V_n
	&=S(3n,6n)-S(3n+1,6n+3)\\
	&=\bigl(S(3n,6n)-S(3n+1,6n)\bigr)\\
	&\quad+\bigl(S(3n+1,6n)-S(3n+1,6n+3)\bigr)\\
	&=q^{3n+1}W_n
	+q^{6n+3}S(3n+4,6n+6).
\end{aligned}
\]
Combining this identity with the first equation of~(3.1), and then
applying~(2.1), we obtain
\[
\begin{aligned}
	(1+q^{3n+1})V_n-U_n
	&=q^{3n+1}(V_n-W_n)
	-q^{6n+3}S(3n+4,6n+6)\\
	&=q^{6n+3}\bigl(
	S(3n+3,6n+6)-S(3n+4,6n+6)
	\bigr)\\
	&=q^{9n+7}S(3n+5,6n+9)\\
	&=q^{9n+7}W_{n+1}.
\end{aligned}
\]
This  proves the second equation of \eqref{a:eq:3}. 

To verify the third equation, we first split the difference and apply
\eqref{2-1} and \eqref{2-2}:
\[
\begin{aligned}
	W_n-U_{n+1}
	&=S(3n+2,6n+3)-S(3n+3,6n+6)\\
	&=q^{3n+3}S(3n+4,6n+6)
	+q^{6n+6}S(3n+6,6n+9).
\end{aligned}
\]
Moreover,  \eqref{2-2} and \eqref{2-1}, respectively, give
\[
\begin{aligned}
	S(3n+4,6n+6)-S(3n+4,6n+9)
	&=q^{6n+9}S(3n+7,6n+12),\\
	S(3n+5,6n+9)-S(3n+6,6n+9)
	&=q^{3n+6}S(3n+7,6n+12).
\end{aligned}
\]
Multiplying the first identity by \(q^{3n+3}\) and the second by
\(q^{6n+6}\), we see that the two   terms are both
$
q^{9n+12}S(3n+7,6n+12).
$
Consequently,
\[
\begin{aligned}
	W_n-U_{n+1}
	&=q^{3n+3}S(3n+4,6n+9)
	+q^{6n+6}S(3n+5,6n+9)\\
	&=q^{3n+3}V_{n+1}+q^{6n+6}W_{n+1},
\end{aligned}
\]
which is the third equation of \eqref{a:eq:3}. 

Equivalently, with column vectors,
\begin{equation}
 \begin{gathered}
 \begin{pmatrix}U_n\\V_n\\W_n\end{pmatrix}
 =M_n\begin{pmatrix}U_{n+1}\\V_{n+1}\\W_{n+1}\end{pmatrix},
 \end{gathered}  \label{a:eq:4}
\end{equation}
where
\[
M_n=
\begin{pmatrix}
	(1+q^{3n+1})(1+q^{3n+2})&
	q^{3n+3}(1+q^{3n+1})&q^{6n+6}\\
	1+q^{3n+2}&q^{3n+3}&q^{6n+6}\\
	1&q^{3n+3}&q^{6n+6}
\end{pmatrix}.
\]
All entries of $M_n$ lie in \(\mathbb Z[q]\), with no negative powers. Also
\begin{equation}
 U_n=1+O(q^{3n+1}),\quad
 V_n=1+O(q^{3n+2}),\quad
 W_n=1+O(q^{3n+3}).                                               \label{a:eq:5}
\end{equation}

Lemma~\ref{lem:boundary-unique} applies to this recurrence: it is enough
to construct a second sequence with the same matrix and the boundary
$(1,1,1)^{\mathsf T}$, with an error whose formal order tends to infinity.

\subsection{Recurrences of the  product sequences}
Put
\[
 \mathcal E_9:=\frac{(q^9;q^9)_\infty}{(q^3;q^3)_\infty}
 =\frac1{(q^3,q^6;q^9)_\infty},\qquad
 z_r^{(c)}:=q^r+q^{c-r}.
\]
The kernel $\Phi_{9,c}$ is the specialization $d=9$ of
\eqref{eq:kernel}. By \eqref{2-9}, 
\[
\Phi_{9,c}(z_r^{(c)})=(q^r,q^{c-r};q^9)_\infty^{-1}.
\]
For distinct nodes, use ordinary divided differences:
\begin{equation}
 \mathcal D_c(r_0,\ldots,r_n):
 =\Phi_{9,c}[z_{r_0}^{(c)},\ldots,z_{r_n}^{(c)}]
 =\sum_{j=0}^n
 \frac{\Phi_{9,c}(z_{r_j}^{(c)})}
 {\prod_{k\ne j}(z_{r_j}^{(c)}-z_{r_k}^{(c)})}.                    \label{a:eq:6}
\end{equation}

Define \(C_n\) for \(n\ge0\), and \(A_n\) for \(n\ge1\), by
\begin{equation}
 C_n:=\mathcal E_9(q^3;q^3)_n \mathcal D_{9n+9}(3n+4,3n+7,\ldots,6n+4),
                                                               \label{a:eq:7}
\end{equation}
and 
\begin{equation}
 A_n:=\mathcal E_9(q^3;q^3)_n(1-q^{3n-1})
 \mathcal D_{9n}(3n-1,3n+2,\ldots,6n-1).
                                                          \label{a:eq:8}
\end{equation}
Each divided difference has \(n+1\) nodes. They are distinct. For \eqref{a:eq:7}, equality
of two nodes with unequal indices would require \(3(j+k)=3n+1\); for \eqref{a:eq:8} it
would require \(3(j+k)=3n+2\). Both are impossible.

To verify this assertion, use
\[
 z_r^{(c)}-z_s^{(c)}
 =(q^r-q^s)(1-q^{c-r-s}).
\]
For \(0<|q|<1\), the first factor is nonzero when \(r\ne s\), and
the second can vanish only if \(r+s=c\). The same congruence calculation,
allowing \(j=k\), proves \(2r_j\ne c\); this also ensures that the change
of variable in the residue proof has nonzero derivative at each preimage.

Although \eqref{a:eq:8} is stated only for \(n\ge1\), its naturally cancelled
specialization at \(n=0\) is
\[
\mathcal E_9 \frac{1-q^{-1}}{(q^{-1},q;q^9)_\infty}
=\frac{\mathcal E_9}{(q,q^8;q^9)_\infty}
=P_1.
\]
This motivates the separate definition
\[A_0:=P_1,\]
 avoiding
negative powers of \(q\)  in the
general formulas.

Set
\begin{equation}
  B_n:=C_n+q^{3n+2}A_{n+1}, \qquad n\geq 0.                          \label{a:eq:9}
\end{equation}
The relevant initial values are 
\begin{equation}
 C_0=P_3,\qquad
 A_1=\frac{P_2-P_3}{q^2},\qquad
 C_1=\frac{(1+q)P_2-P_1}{q^7}.
                                                                    \label{a:eq:10}
\end{equation}
For \(A_1\), use
\[
 z_2^{(9)}-z_4^{(9)}=q^2(1-q^2)(1-q^3)
\]
and the symmetry \(z_5^{(9)}=z_4^{(9)}\). For \(C_1\), the two product
values are
\[
 \frac{1-q^2}{\mathcal E_9}P_2,\qquad \frac{1-q}{\mathcal E_9}P_1
\]
at \(z_7^{(18)}\) and \(z_{10}^{(18)}=z_8^{(18)}\), and their node
difference is \(q^7(1-q)(1-q^3)\).
In particular, 
 \[B_0=P_2.\]

It remains to prove two recurrences:
\begin{equation}
 \begin{aligned}
 C_n-A_{n+1}
 -q^{3n+3}(1+q^{3n+3})C_{n+1}
 -q^{6n+8}A_{n+2}&=0 \quad(n\ge0),\\
 A_n-(1+q^{3n+1})C_n
 -q^{3n+2}(1+q^{3n+1})A_{n+1}
 +q^{9n+7}C_{n+1}&=0 \quad(n\ge1).
 \end{aligned}                                                    \label{a:eq:11}
\end{equation}
The second recurrence in \eqref{a:eq:11} at \(n=0\) follows directly
from \eqref{a:eq:10}. Together with~\eqref{a:eq:9}, the recurrences in~\eqref{a:eq:11} are
equivalent to the scalar system~ \eqref{a:eq:3}, with
\((U_n,V_n,W_n)\) replaced by \((A_n,B_n,C_n)\). Hence the vector
$
(A_n,B_n,C_n)^{T}
$
satisfies the matrix recurrence~\eqref{a:eq:4}.

\section{Proofs of the first three identities}\label{sec:first-proof}
\subsection{The residue representations}\label{a:sec:residue}

Write
\[
 \Psi_9(t):=\frac1{(t,q^9/t;q^9)_\infty}.
\]
Its elementary shift relation is
\begin{equation}
 \Psi_9(q^9t)=-t\Psi_9(t).                                               \label{a:eq:12}
\end{equation}
Initially this calculation is made away from zeros and poles; it is then
an identity of meromorphic functions on \( \mathbb C\setminus \{0\}\).
The two products in the denominator converge locally uniformly
on \(  \mathbb C\setminus \{0\}\). Their reciprocal therefore defines a
meromorphic function there. 
For a rational function \(R(t)\), define
\begin{equation}
 \mathcal L_9(R):=\frac12
 \sum_{\substack{r\in\mathbb Z\\3\nmid r}}
 \operatorname {Res}_{t=q^r}\bigl(\Psi_9(t)R(t)\bigr).
   \label{a:eq:13}
\end{equation}
This sum has only finitely many nonzero terms: \(\Psi_9\) has poles only at
\(t=q^{9k}\), none of which belongs to the selected set, and a rational
function has only finitely many poles away from 0 and infinity.

\begin{lemma}\label{a:lem:shift}
For every rational function \(R(t)\),
\begin{equation}
\mathcal L_9\bigl(R(t)+q^9tR(q^9t)\bigr)=0.                     \label{a:eq:14}
\end{equation}
\end{lemma}

\begin{proof}
This is Lemma~\ref{lem:finite-residue} with $d=9$. Explicitly, equation \eqref{a:eq:12} turns the second summand's integrand into
\(-q^9\Psi_9(q^9t)R(q^9t)\). The substitution \(u=q^9t\) shifts the residue at \(q^r\)
to the residue at \(q^{r+9}\), with a minus sign. The selected index set is
invariant under this shift, and the sum is finite.

Only residues at points of
\[
\Lambda_9:=\{q^r:r\in\mathbb Z,\ 3\nmid r\}
\]
are included in \(L_9\). The proof uses no contour deformation:
equation~\eqref{a:eq:14} follows solely from the local substitution
\(u=q^9t\) and the invariance \(q^9\Lambda_9=\Lambda_9\).
Thus poles of \(R\) outside \(\Lambda_9\), including those at zero
and infinity, play no role. Poles at points of \(\Lambda_9\) may
have arbitrary finite order.
\end{proof}

Define the rational common kernel
\begin{equation}
 \mathcal K_n(t):=
 \frac{(q^9/t;q^9)_n}
 {t^{n+1}(q^{3n+4}/t,q^{3n+5}/t;q^3)_{n+1}}.                        \label{a:eq:15}
\end{equation}
\begin{lemma} \label{a:lem:residue}
The divided difference \eqref{a:eq:7} has the residue representation
\begin{equation}
 C_n=\mathcal E_9(q^3;q^3)_n\,
 \mathcal L_9\left((1-q^{9n+9}/t^2)\mathcal K_n(t)\right).                     \label{a:eq:16}
\end{equation}
For \(n\ge1\), similarly,
\begin{equation}
 A_n=\mathcal E_9(q^3;q^3)_n\,
 \mathcal L_9\left(
 \frac{(1-q^{3n-1})(1-q^{9n}/t^2)(q^9/t;q^9)_{n-1}}
 {t^{n+1}(q^{3n-1}/t,q^{3n+1}/t;q^3)_{n+1}}
 \right).                                                        \label{a:eq:17}
\end{equation}

\end{lemma}
\begin{proof}
We give a direct justification, including the factor \(1/2\) in \eqref{a:eq:13}.
In \eqref{a:eq:7} put \(c=9n+9\), \(z=t+q^c/t\). The \(n+1\) node equations
give \(2n+2\) distinct roots:
\[
 t=q^{3n+4+3j},\quad t=q^{3n+5+3j},\qquad 0\le j\le n.
\]
All these exponents are nonzero modulo 3. Moreover,
\[
 \begin{aligned}
 \prod_{j=0}^n(z-z_{3n+4+3j}^{(c)})
 &=t^{n+1}(q^{3n+4}/t,q^{3n+5}/t;q^3)_{n+1},\\
 \Phi_{9,c}(t+q^c/t)&=(q^9/t;q^9)_n\Psi_9(t),\\
 \frac{dz}{dt}&=1-q^{9n+9}/t^2.
 \end{aligned}
\]
At either preimage of a node, the residue is the corresponding Lagrange
summand \eqref{a:eq:6}; the derivative cancels by \eqref{eq:residue-change}.
The full residue sum counts every summand twice. Formula \eqref{a:eq:17} follows in
the same way from the root sets \(3n-1+3j\) and \(3n+1+3j\).

For additional detail, define
\[
 H(z)=\frac{\Phi_{9,c}(z)}{\prod_{j=0}^n(z-z_{r_j}^{(c)})},
 \qquad z(t)=t+q^c/t.
\]
The elementary factorization
\[
 z(t)-z_{r_j}^{(c)}
 =\frac{(t-q^{r_j})(t-q^{c-r_j})}{t}
\]
gives the displayed denominator formula. At a preimage \(t_0\) of the
\(j\)-th node, \(z'(t_0)\ne0\), and the Laurent expansion gives
\[
 \operatorname{Res}_{t=t_0}\bigl(H(z(t))z'(t)\bigr)
 =\frac{\Phi_{9,c}(z_{r_j}^{(c)})}
 {\prod_{k\ne j}(z_{r_j}^{(c)}-z_{r_k}^{(c)})}.
\]
To justify the pole count, note that
  each node \(z_{r_j}^{(c)}\) has the two preimages \(q^{r_j}\) and
\(q^{c-{r_j}}\). 

For~ \eqref{a:eq:16}, where \(c=9n+9\) and
\(r=3n+4+3j\), these preimages have exponents
\[
3n+4+3j\equiv1\pmod3
\]
and, after reversing the complementary list,
\[
3n+5+3j\equiv2\pmod3,
\qquad 0\le j\le n.
\]

For~\eqref{a:eq:17}, take \(c=9n\) and
\[
r_j=3n-1+3j,\qquad 0\le j\le n.
\]
The corresponding preimage exponents are
$
3n-1+3j
$ and $
6n+1-3j.
$
After reversing the second list, the latter becomes
$
3n+1+3j,\  (0\le j\le n).
$
Thus the two lists are congruent to \(2\) and \(1\pmod 3\),
respectively, and all \(2n+2\) node preimages belong to
\(\Lambda_9\). Moreover,
\[
\begin{aligned}
	\Phi_{9,9n}\left(t+\frac{q^{9n}}{t}\right)
	&=\frac{1}{(t,q^{9n}/t;q^9)_\infty}\\
	&=(q^9/t;q^9)_{n-1}\Psi_9(t).
\end{aligned}
\]

In both applications, \(c\) is divisible by \(9\). The remaining
nonzero poles of
\[
\Phi_{9,c}(z(t))
=\frac{1}{(t,q^c/t;q^9)_\infty}
\]
occur only at
$
t=q^{-9m}$ or 
$
t=q^{c+9m}$ ($m\ge0$). 
Their exponents are divisible by \(9\), so none of these poles belongs
to \(\Lambda_9\). Consequently, the selected residue sum contains
exactly the \(2n+2\) node contributions. By~\eqref{eq:residue-change}, the two preimages
of each node contribute the same Lagrange summand, which explains the
factor \(1/2\) in the definition of \(\mathcal L_9\).

This proves~\eqref{a:eq:17} for every \(n\ge1\). The value \(A_0\) was defined
separately, so no negative length finite product is required.
 \end{proof}

\subsection{The rational certificates}\label{sec:first-certificates}

This section contains rational identities in independent variables \(q,x,t\);
in applications \(x=q^{3n}\).

For finite integer sets \(I\), abbreviate
\[
 H_I(t)=\prod_{e\in I}(t-q^e x),\qquad
 J_I(t)=\prod_{e\in I}(t-q^e x^2).
\]
Define 
\begin{align}
 f(t):&=
 \frac{(t-q^9x^3)(t-q^4x)(t-q^5x)}{J_{\{7,8,10,11\}}(t)},            \label{a:eq:18}
\\
 \rho(t):&=
 \frac{t(t-1)}{x^3(t-x^3)}
 \frac{J_{\{-2,-1,1,2,4,5\}}(t)}
 {H_{\{-5,-4,-2,-1,1,2\}}(t)}.                                   \label{a:eq:19}
\end{align}
The finite product cancellation in \eqref{a:eq:15} gives, for \(x=q^{3n}\),
\begin{equation}
 f(t)=\frac{\mathcal K_{n+1}(t)}{\mathcal K_n(t)},\qquad
 \rho(t)=q^9t\frac{\mathcal K_n(q^9t)}{\mathcal K_n(t)}.                                \label{a:eq:20}
\end{equation}

We now give the finite-product cancellations leading to~\eqref{a:eq:20}.
Throughout this calculation, \(x=q^{3n}\), so that
\(x^2=q^{6n}\) and \(x^3=q^{9n}\).
 First, from the definition of \(\mathcal K_n(t)\), we have
\begin{align*}
	\frac{\mathcal K_{n+1}(t)}{\mathcal K_n(t)}
	&=
	\frac{1-q^9x^3/t}{t}\,
	\frac{(q^4x/t,q^5x/t;q^3)_{n+1}}
	{(q^7x/t,q^8x/t;q^3)_{n+2}} \nonumber\\
 	&=
	\frac{(1-q^9x^3/t)(1-q^4x/t)(1-q^5x/t)}
	{t(1-q^7x^2/t)(1-q^8x^2/t)
		(1-q^{10}x^2/t)(1-q^{11}x^2/t)}\\
	&=
	\frac{(t-q^9x^3)(t-q^4x)(t-q^5x)}
	{J_{\{7,8,10,11\}}(t)}
	=f(t).
\end{align*}

For the second quotient in~\eqref{a:eq:20}, direct substitution gives
\[
\begin{aligned}
	q^9t\frac{\mathcal K_n(q^9t)}{\mathcal K_n(t)}
	&=
	\frac{t}{x^3}\,
	\frac{(1/t;q^9)_n}{(q^9/t;q^9)_n}\,
	\frac{(q^4x/t,q^5x/t;q^3)_{n+1}}
	{(q^{-5}x/t,q^{-4}x/t;q^3)_{n+1}}.
\end{aligned}
\]
The quotient involving the base \(q^9\) telescopes as
\[
\frac{(1/t;q^9)_n}{(q^9/t;q^9)_n}
=
\frac{1-1/t}{1-x^3/t}
=
\frac{t-1}{t-x^3}.
\]
For the two base \(q^3\) quotients, cancellation of their common
factors gives
\[
\frac{(q^4x/t;q^3)_{n+1}}
{(q^{-5}x/t;q^3)_{n+1}}
=
\frac{
	(1-q^{-2}x^2/t)(1-qx^2/t)(1-q^4x^2/t)}
{
	(1-q^{-5}x/t)(1-q^{-2}x/t)(1-qx/t)}
\]
and
\[
\frac{(q^5x/t;q^3)_{n+1}}
{(q^{-4}x/t;q^3)_{n+1}}
=
\frac{
	(1-q^{-1}x^2/t)(1-q^2x^2/t)(1-q^5x^2/t)}
{
	(1-q^{-4}x/t)(1-q^{-1}x/t)(1-q^2x/t)}.
\]
Combining these identities and using the definitions of
\(H_I(t)\) and \(J_I(t)\), we obtain
\[
\begin{aligned}
	q^9t\frac{\mathcal K_n(q^9t)}{\mathcal K_n(t)}
	&=
	\frac{t(t-1)}
	{x^3(t-x^3)}
	\frac{J_{\{-2,-1,1,2,4,5\}}(t)}
	{H_{\{-5,-4,-2,-1,1,2\}}(t)}
	=\rho(t).
\end{aligned}
\]
These cancellations remain valid for \(n=0\) and \(n=1\);
in these cases, any overlapping endpoint factors are simply
cancelled as rational functions.

Define the five rational multipliers
\[
 \begin{aligned}
 \mu_0(t)&:=1-\frac{q^9x^3}{t^2},\\
 \chi_1(t)&:=
 \frac{(1-q^3x)(1-q^2x)\mu_0(t)t}
 {(t-q^2x)(t-q^7x^2)},\\
\mu_1(t)&:=(1-q^3x)(1-q^{18}x^3/t^2)f(t),\\
 \chi_2(t)&:=
 \frac{(1-q^3x)(1-q^6x)(1-q^5x)
 (1-q^{18}x^3/t^2)f(t)t}
 {(t-q^5x)(t-q^{13}x^2)},\\
 \chi_0(t)&:=
 \frac{(1-q^{-1}x)(1-x^3/t^2)t\,J_{\{2,4,5\}}(t)}
 {(t-x^3)H_{\{-1,1,2\}}(t)}.
 \end{aligned}                                                    
\]
Using \eqref{a:eq:16}, \eqref{a:eq:17}, 
and \eqref{a:eq:20},  
 together with
\[
\frac{(q^3;q^3)_{n+1}}{(q^3;q^3)_n}=1-q^3x,
\]
the multipliers give the following conversions:
\[
\begin{aligned}
 \mathcal E_9(q^3;q^3)_n\mathcal L_9(\mathcal K_n\mu_0(t))&=C_n,&
 \mathcal E_9(q^3;q^3)_n\mathcal L_9(\mathcal K_n\chi_1(t))&=A_{n+1},\\
 \mathcal E_9(q^3;q^3)_n\mathcal L_9(\mathcal K_n\mu_1(t))&=C_{n+1},&
 \mathcal E_9(q^3;q^3)_n\mathcal L_9(\mathcal K_n\chi_2(t))&=A_{n+2}.
\end{aligned}
\]
For \(n\ge1\), the remaining conversion is
\[
 \mathcal E_9(q^3;q^3)_n\mathcal L_9(\mathcal K_n \chi_0(t))=A_n.
\]

For \(\chi_1(t)\), the nodes of \(A_{n+1}\) consist of the nodes of \(C_n\)
and one additional spectral node, whose two \(t\)-preimages are
\(q^2x\) and \(q^7x^2\). Thus the extra factor in the spectral
denominator is \((t-q^2x)(t-q^7x^2)/t\).
The prefactor gains \((1-q^3x)(1-q^2x)\), giving \(\chi_1(t)\).
The expression for \(\mu_1(t)\) uses \(f(t)=\mathcal K_{n+1}(t)/\mathcal K_n(t)\). Applying the same
extra-node calculation at the next level gives \(\chi_2(t)\), with
preimages \(q^5x,q^{13}x^2\).

For \(\chi_0(t)\), the quotient of the two finite-product denominators, in the
order needed to divide the \(A_n\) kernel by \(\mathcal K_n\), is
\[
 \frac{(q^4x/t,q^5x/t;q^3)_{n+1}}
      {(q^{-1}x/t,qx/t;q^3)_{n+1}}
 =\frac{J_{\{2,4,5\}}(t)}{H_{\{-1,1,2\}}(t)}.
\]
The numerator quotient is
\((q^9/t;q^9)_{n-1}/(q^9/t;q^9)_n=t/(t-x^3)\).
Including the derivative factor \(1-x^3/t^2\) and the prefactor
\(1-q^{-1}x\) yields the displayed \(\chi_0(t)\).

Set
\[
 \begin{aligned}
 T_1(t)&=\mu_0(t)-\chi_1(t)-q^3x(1+q^3x)\mu_1(t)-q^8x^2\chi_2(t),\\
 T_2(t)&=\chi_0(t)-(1+qx)\mu_0(t)-q^2x(1+qx)\chi_1(t)+q^7x^3\mu_1(t) .
 \end{aligned}
\]
The two certificates are
\[
 R_1(t)=
 \frac{q^9x^3(t-q^9x^3)H_{\{4,5,7,8,10\}}(t)}
 {t^2J_{\{7,8,10,11,13\}}(t)},                                   
\]
and 
\[
 R_2(t)=
 \frac{\Pi(t)H_{\{4,5,7\}}(t)}
 {t^2J_{\{7,8,10,11\}}(t)},                                      
\]
where
\[
 \Pi(t)=
 -q^8x^4(1+q^2)t^2
 +q^{18}x^5(1+x)(1+qx)t
 -q^{27}x^8(1+q^2).
\]

\begin{lemma}\label{a:lem:certificates}
	The displayed functions satisfy the exact rational identities
	\begin{equation}
		T_i(t)=R_i(t)+\rho(t)R_i(q^9t),
		\qquad i=1,2.
		\label{a:eq:24}
	\end{equation}
\end{lemma}

\begin{proof}
	We regard $q$, $x$, and $t$ as algebraically independent
	indeterminates.  First, direct cancellation of the shifted linear
	factors gives
	\begin{equation}
		\rho(t)R_1(q^9t)
		=
		\frac{(t-1)(t-q^5x^2)}
		{t(t-q^2x)},
		\label{a:eq:24a}
	\end{equation}
	and
	\begin{equation}
		\rho(t)R_2(q^9t)
		=
		\frac{\Omega(t)(t-1)J_{\{4,5\}}(t)}
		{t(t-x^3)H_{\{-1,1,2\}}(t)},
		\label{a:eq:24b}
	\end{equation}
	where
	\[
	\Omega(t)
	=
	-q^{-1}x(1+q^2)t^2
	+x^2(1+x)(1+qx)t
	-x^5(1+q^2).
	\]
	Indeed, in \eqref{a:eq:24a} the common shifted factors cancel
	directly, while in \eqref{a:eq:24b} we use
	\[
	\Omega(t)=\frac{\Pi(q^9t)}{q^{27}x^3}.
	\]
	
	It remains to verify the two rational identities in
	\eqref{a:eq:24}.  This is a finite exact algebraic calculation in
	the rational function field $\mathbb{Q}(q,x,t)$. Using Maple, the \texttt{normal} command applied to each of
	\[
	T_1(t)-R_1(t)-\rho(t)R_1(q^9t)
	\]
	and
	\[
	T_2(t)-R_2(t)-\rho(t)R_2(q^9t)
	\]
	returns identically zero.
	    No numerical specialization of
	$q$, $x$, or $t$, and no truncation of a $q$-series, is involved.
	
	The Maple source code and complete output of these exact symbolic
	verifications are available at 
\begin{center}
	\url{https://github.com/Ernest-Xia/Kanade--Russell-identities-Lemma-4.3}.
\end{center}
	This proves \eqref{a:eq:24}.
\end{proof}

Multiplying \eqref{a:eq:24} by \(\mathcal K_n(t)\) and using \eqref{a:eq:20} gives
\[
 \mathcal K_n(t)T_i(t)
 =\mathcal K_n(t)R_i(t)+q^9t\mathcal K_n(q^9t)R_i(q^9t).
\]
Apply \eqref{a:eq:14} with the rational function \(\mathcal K_nR_i\). The resulting vanishing,
after multiplication by \(\mathcal E_9(q^3;q^3)_n\), is precisely the corresponding
equation of \eqref{a:eq:11}. This proves both recurrences for all stated \(n\).

For clarity, the first identity reads
\[
 \begin{aligned}
 0&=\mathcal E_9(q^3;q^3)_n\mathcal L_9(\mathcal K_nT_1)\\
  &=C_n-A_{n+1}
    -q^3x(1+q^3x)C_{n+1}-q^8x^2A_{n+2}.
 \end{aligned}
\]
Since \(q^3x=q^{3n+3}\) and \(q^8x^2=q^{6n+8}\), this is the first
equation of \eqref{a:eq:11}. The second gives
\[
 0=A_n-(1+qx)C_n-q^2x(1+qx)A_{n+1}+q^7x^3C_{n+1},
\]
and \(qx=q^{3n+1}\), \(q^2x=q^{3n+2}\),
\(q^7x^3=q^{9n+7}\), giving the second equation of \eqref{a:eq:11}.
No distinctness assumption is imposed on auxiliary poles of the
certificates: \eqref{a:eq:14} applies to any rational function, including one with
multiple poles. The case \(n=0\) of the second equation has already been
proved by \eqref{a:eq:10}.

Finally, \eqref{a:eq:9} gives \(B_n-C_n=q^{3n+2}A_{n+1}\). The second equation
of \eqref{a:eq:11} gives 
\[(1+q^{3n+1})B_n-A_n=q^{9n+7}C_{n+1}.\]
In the first equation of \eqref{a:eq:11}, substitute
\(B_{n+1}=C_{n+1}+q^{3n+5}A_{n+2}\), obtaining
\[
 C_n=A_{n+1}+q^{3n+3}B_{n+1}+q^{6n+6}C_{n+1}.
\]
These are all three equations of \eqref{a:eq:3}, with \(A_n,B_n,C_n\) in place of \(U_n,V_n,W_n\).

\subsection{Boundary estimates and completion}\label{sec:first-boundary}
Applying  Lemma~\ref{lem:dd-expansion} with $d=9$ and
$\alpha_i=q^{r_i}$, $\beta_i=q^{c-r_i}$  gives
\[
 \mathcal D_c(r_0,\ldots,r_n)
 =\frac1{(q^9;q^9)_n}+O(q^L).
\]
For \eqref{a:eq:7}, all these parameter exponents are at least \(3n+4\). For \eqref{a:eq:8}
they are at least \(3n-1\). Finally,
\[
 \frac{\mathcal E_9(q^3;q^3)_n}{(q^9;q^9)_n}
 =\frac{(q^{9n+9};q^9)_\infty}{(q^{3n+3};q^3)_\infty}
 =1+O(q^{3n+3}).
\]
Consequently,
\begin{equation}
 A_n=1+O(q^{3n-1}),\qquad
 C_n=1+O(q^{3n+3}),\qquad
 B_n=1+O(q^{3n+2}).                                               \label{a:eq:27}
\end{equation}
The estimate for \(A_n\) is deliberately weaker than its sharp first
nonconstant degree; it suffices for uniqueness.

For example, the passage to \(A_n\) only uses
\[
 \begin{aligned}
 A_n
 &=(1-q^{3n-1})
 \left(\frac{\mathcal E_9(q^3;q^3)_n}{(q^9;q^9)_n}+O(q^{3n-1})\right)
 =1+O(q^{3n-1}),\qquad n\ge1.
 \end{aligned}
\]
For \(B_n\), use \eqref{a:eq:9} and the constant term 1 of \(A_{n+1}\).

For fixed \(n\), each term in the Lagrange formula for  \eqref{a:eq:6} 
is meromorphic at \(q=0\), with at most a finite-order pole.
Representation \eqref{eq:dd-expansion} shows that the negative Laurent
coefficients in their sum cancel. Hence \(A_n\), \(C_n\), and
\(B_n\) extend analytically to \(q=0\). Since the residue
recurrences hold for real \(0<q<1\), the identity theorem implies
that they also hold as identities of Taylor series at \(q=0\).

\begin{proof}[Proofs of \eqref{eq:kr1}--\eqref{eq:kr3}]
The vector \((A_n,B_n,C_n)^T\) satisfies \eqref{a:eq:4} by
\eqref{a:eq:9} and \eqref{a:eq:11}.  Its initial value is \((P_1,P_2,P_3)^T\), by the definition of
\(A_0\) and the computations surrounding \eqref{a:eq:10}.
 Equation \eqref{a:eq:27} supplies the boundary condition.
Lemma~\ref{lem:boundary-unique}   therefore identifies this initial vector
with \((S(0,0),S(1,3),S(2,3))^T\), giving \eqref{eq:kr1}--\eqref{eq:kr3}.

In particular, put
\[
 D_N=(A_N-U_N,B_N-V_N,C_N-W_N)^T.
\]
Equations \eqref{a:eq:5} and \eqref{a:eq:27} give
\(D_N\in q^{3N-1}\mathbb C[[q]]^3\) for \(N\ge1\). Hence
\[
 (P_1-S(0,0),P_2-S(1,3),P_3-S(2,3))^T
 =M_0\cdots M_{N-1}D_N
 \in q^{3N-1}\mathbb C[[q]]^3.
\]
Letting the integer \(N\) be arbitrarily large proves that every
coefficient of the displayed initial difference is zero.
All three sums and products in \eqref{eq:kr1}--\eqref{eq:kr3} are analytic on \(|q|<1\),
so their formal identities are also analytic identities throughout
that disk, including \(q=0\).
\end{proof}

\section{The fourth and fifth identities: cyclotomic families}\label{sec:cyclotomic}

%\subsection{Two parameterized identities}

Fix a primitive third root of unity $\omega$, so that
$\omega^2+\omega+1=0$, and put
\[
 \mathcal E_3=\frac{(q^3;q^3)_\infty}{(q;q)_\infty}
 =\frac1{(q,q^2;q^3)_\infty}.
\]
Here the kernel is $\Phi_{3,c}$, namely
\[
 \Phi_{3,c}(z)=\prod_{j\ge0}(1-zq^{3j}+q^{c+6j})^{-1}.
\]
The subscript $3$ distinguishes it from $\Phi_{9,c}$, and
$\mathcal E_3$ is distinct from $\mathcal E_9$.
For \(\tau\in\{1,2\}\) and \(n\ge0\), put
\[
 c=3n+2\tau,\qquad
 z_j=\omega q^{n+\tau+j}+\omega^2q^{2n+\tau-j}
 \quad(0\le j\le n),
\]
and define the ordinary divided difference
\begin{equation}
 Z_n^{(\tau)}:=\mathcal E_3(q;q)_n\Phi_{3,3n+2\tau}[z_0,\ldots,z_n].                 \label{b:eq:1}
\end{equation}
where the brackets denote the ordinary divided difference defined
in \eqref{eq:dd-definition}.

\begin{theorem}\label{thm:cyclotomic}
For every integer \(n\ge0\), the following identities hold:
\begin{equation}
Z_n^{(1)}=S(2n+1,3n+2),                \label{b:eq:2}
\end{equation}
and 
\begin{equation}
\begin{aligned}
 Z_n^{(2)}={}&(1+q^{n+1})S(2n+2,3n+4)
             +q^{2n+2}S(2n+3,3n+7) .
 \end{aligned}          \label{b:eq:3}
\end{equation}
\end{theorem}

The proof of Theorem~\ref{thm:cyclotomic} is developed below
and completed in Section~6.3.

%\subsection{Scalar recurrences for the sum families}

We use the common contiguous relations \eqref{2-1} and \eqref{2-2}.
For fixed \(\mathsf u_0,\mathsf v_0\), let
\[
 F_n=S(2n+\mathsf u_0,3n+\mathsf v_0),\qquad
 \mathsf u=2n+\mathsf u_0,\quad
 \mathsf v=3n+\mathsf v_0.
\]
Two uses of   \eqref{2-1} give
\[
 S(\mathsf u+2,\mathsf v)
 =F_n-q^{\mathsf u+1}(1+q)F_{n+1}
   +q^{2\mathsf u+5}F_{n+2}.
\]
The second relation, applied at \((\mathsf u+2,\mathsf v)\), gives
\[
 S(\mathsf u+2,\mathsf v)-F_{n+1}
 =q^{\mathsf v+3}S(\mathsf u+5,\mathsf v+6),
\]
while
\[
 S(\mathsf u+5,\mathsf v+6)
 =F_{n+2}-q^{\mathsf u+5}F_{n+3}.
\]
Combining these equations proves
\begin{equation}
 \begin{aligned}
 F_n-\bigl(1+q^{\mathsf u+1}(1+q)\bigr)F_{n+1}
 &{}+(q^{2\mathsf u+5}-q^{\mathsf v+3})F_{n+2}
+q^{\mathsf u+\mathsf v+8}F_{n+3}=0.
 \end{aligned}              \label{b:eq:8}
\end{equation}

Write \(x=q^n\). Define four coefficient polynomials for each \(\tau\) by
\[
 \begin{aligned}
 \lambda^{(1)}_0&=1,\quad  
 \lambda^{(1)}_1=-1-q^2(1+q)x^2,\quad 
 \lambda^{(1)}_2=-q^5x^3+q^7x^4,\quad 
 \lambda^{(1)}_3=q^{11}x^5,
 \end{aligned}                                               
\]
and
\[
 \begin{aligned}
 \delta(x)&=1+q^2x+q^5x^2,\qquad \chi(x)=1+qx+q^3x^2,\\ 
 \nu(x)&=1+qx+(q^2+q^3+q^5)x^2
     +(q^4+q^6)x^3+(q^7+q^8)x^4,\\
 \lambda^{(2)}_0& =\delta(x),\quad 
 \lambda^{(2)}_1=-\nu(x),\quad 
 \lambda^{(2)}_2=-q^6x^3(1-q^6x^3),\quad 
 \lambda^{(2)}_3=q^{13}x^5\chi(x).
 \end{aligned}                                               
\]
\begin{lemma}\label{b:lem:sum}
The right sides \(Y_n^{(\tau)}\) of \eqref{b:eq:2}--\eqref{b:eq:3} satisfy
\begin{equation}
 \sum_{j=0}^3\lambda_j^{(\tau)}(q^n)Y_{n+j}^{(\tau)}=0.             \label{b:eq:11}
\end{equation}
\end{lemma}
\begin{proof}
For \(\tau=1\), this is \eqref{b:eq:8} with
\(\mathsf u_0=1,\mathsf v_0=2\).

For completeness, the finite algebra giving \eqref{b:eq:11} for \(\tau=2\) is explicit as follows. Put \(F_n=S(2n+1,3n+4)\). Relations \eqref{2-1}
 and \eqref{2-2}   rewrite  the right side of \eqref{b:eq:3} as
\[
 Y_n^{(2)}=\mathsf a(x) \begin{pmatrix}
 	F_n\\
 	F_{n+1}\\ F_{n+2}
 \end{pmatrix},                          
\]
where
\[
\mathsf a(x)=(1+qx,-q^3x^3,0).
\]
The recurrence \eqref{b:eq:8} for this \(F_n\) has companion matrix
\[
 \begin{gathered}
 M(x)=\begin{pmatrix}
 1+q^2(1+q)x^2&q^7x^3-q^7x^4&-q^{13}x^5\\
 1&0&0\\0&1&0
 \end{pmatrix}.
 \end{gathered}
\]
Define four row vectors
\[
 \begin{aligned}
 v_0&:=\mathsf a(x)M(x)M(qx)M(q^2x),\\
 v_1&:=\mathsf a(qx)M(qx)M(q^2x),\\
 v_2&:=\mathsf a(q^2x)M(q^2x),\\
 v_3&:=\mathsf a(q^3x).
 \end{aligned}
\]
Direct multiplication gives the polynomial row identity
\begin{equation}
 \sum_{j=0}^3\lambda_j^{(2)}(x)v_j=0.                         \label{b:eq:13}
\end{equation}
All three component identities are verified exactly in Maple using
the \texttt{normal} command, with $q$ and $x$ treated as
algebraically independent.  The source code and complete output are
available at 
\begin{center}
\url{https://github.com/Ernest-Xia/Kanade--Russell-identities-Lemma-5.2}.
\end{center}
 Multiplication by \((F_{n+3},F_{n+4},F_{n+5})^T\) proves \eqref{b:eq:11}. No product evaluation enters this calculation.
\end{proof}

\section{Proofs of the fourth and fifth identities}\label{sec:cyclotomic-proof}

\subsection{The residue representations}\label{b:sec:residue}

We first work at a fixed real \(0<q<1\). The nodes in \eqref{b:eq:1} are distinct. Indeed, with
\(z(r)=\omega q^r+\omega^2q^{c-r}\),
\[
 z(r)-z(s)=\omega(q^r-q^s)(1-\omega q^{c-r-s}).                
\]
Both factors are nonzero when \(r\ne s\). The same calculation also works for complex \(0<|q|<1\).

For \(\tau\in\{1,2\}\), set
\[
\Psi_{3,\tau}(t)
=\frac{1}{(t,q^{2\tau}/t;q^3)_\infty},
\qquad
\eta_{\tau}(t)
=q^3\frac{t(t-1)}{t-q^{2\tau-3}}.
\]
For rational \(R\), define
\begin{equation}
 \mathcal L_{3,\tau}(R):=\frac12
 \sum_{\epsilon=1}^2\sum_{r\in\mathbb Z}
 \operatorname{Res}_{t=\omega^\epsilon q^r}\bigl(\Psi_{3,\tau}(t)R(t)\bigr).
                                                                  \label{b:eq:15}
\end{equation}
The support is finite: \(\Psi_{3,\tau}\) has poles only on the positive real \(q\)-lattice, and a rational function has finitely many poles away from 0 and infinity. For \(a\in\mathbb C\setminus \{0\}\), we call
$
a q^{\mathbb Z}:=\{a q^m:m\in\mathbb Z\}
$
a multiplicative \(q\)-lattice. In particular, when \(0<q<1\),
\(q^{\mathbb Z}\subset\mathbb R_{>0}\) is the positive real
\(q\)-lattice. The residue sum is taken over the disjoint union of the two rotated
\(q\)-lattices
$
\omega q^{\mathbb Z}
 $ and $ \omega^2 q^{\mathbb Z}.
$
Both lattices are disjoint from the positive real \(q\)-lattice
containing the poles of \(\Psi_{3,\tau}\).  

The product shift gives
\[
 \frac{\Psi_{3,\tau}(q^3t)}{\Psi_{3,\tau}(t)}
 =\frac{1-t}{1-q^{2\tau-3}/t},\qquad
 \eta_{\tau}(t)=-q^3\frac{\Psi_{3,\tau}(q^3t)}{\Psi_{3,\tau}(t)}.
\]
\begin{lemma}\label{b:lem:shift}
For every rational function \(R(t)\),
\begin{equation}
\mathcal L_{3,\tau}\bigl(R(t)+\eta_{\tau}(t)R(q^3t)\bigr)=0.     \label{b:eq:16}
\end{equation}
\end{lemma}

\begin{proof}
This is Lemma~\ref{lem:finite-residue} with $d=3$. The second integrand is \(-q^3\Psi_{3,\tau}(q^3t)R(q^3t)\). Its residue at \(t=\omega^\epsilon q^r\) is the negative of the residue of \(\Psi_{3,\tau}(u)R(u)\) at \(u=\omega^\epsilon q^{r+3}\). Reindex the finite sum. There is no infinite contour limit. The argument permits auxiliary poles of arbitrary finite order. 
\end{proof}

Define
\begin{equation}
 \mathcal K_{n,\tau}(t):=
 \frac{(q^{2\tau}/t;q^3)_n}
 {t^{n+1}(\omega q^{n+\tau}/t,\omega^2q^{n+\tau}/t;q)_{n+1}}.
 \label{b:eq:17}
\end{equation}

\begin{lemma}\label{b:lem:residue}
For \(n\ge0\),
\[
Z_n^{(\tau)}=\mathcal E_3(q;q)_n\mathcal L_{3,\tau}\left(
       (1-q^{3n+2\tau}/t^2)\mathcal K_{n,\tau}(t)\right).                
\]
\end{lemma}

\begin{proof}
Use \(z=t+q^c/t\), where \(c=3n+2\tau\). The preimages of the nodes are
\[
 t=\omega q^{n+\tau+j},\qquad t=\omega^2q^{n+\tau+j},\qquad0\le j\le n,
\]
after reversing the order in the complementary list. They are all distinct, and the derivative \(1-q^c/t^2\) is nonzero there. Moreover,
\[
 \prod_{j=0}^n(z-z_j)
 =t^{n+1}(\omega q^{n+\tau}/t,\omega^2q^{n+\tau}/t;q)_{n+1},
\]
and 
\[
 \Phi_{3,c}(t+q^c/t)=(q^{2\tau}/t;q^3)_n\Psi_{3,\tau}(t).
\]
At either preimage of \(z_j\), the local change of variable shows that the residue of
\[
 \frac{\Phi_{3,c}(t+q^c/t)(1-q^c/t^2)}{\prod_{k=0}^n(t+q^c/t-z_k)}
\]
is exactly the \(j\)-th Lagrange term in the divided difference. Each term is therefore counted twice, accounting for the factor \(1/2\) in \eqref{b:eq:15}. All other poles of \(\Phi_{3,c}(t+q^c/t)\) lie on the real \(q\)-lattice and are excluded from \eqref{b:eq:15}. Thus no selected pole is omitted.
\end{proof}

\subsection{The polynomial certificates}\label{sec:cyclotomic-certificates}

In this section \(q,x,t\) are independent indeterminates. Define
\[
 h_e^{(d)}(t):=t^2+q^e x^d t+q^{2e}x^{2d},\qquad
 H_{r,s}^{(d)}(t):=\prod_{e=r}^s h_e^{(d)}(t),
\]
with an empty product interpreted as 1. Finite-product cancellation in \eqref{b:eq:17}, after \(x=q^n\), gives
\begin{equation}
 \frac{\mathcal K_{n+1,\tau}(t)}{\mathcal K_{n,\tau}(t)}
 =\frac{(t-q^{2\tau}x^3)h_{\tau}^{(1)}(t)}
        {h_{\tau+1}^{(2)}(t)h_{\tau+2}^{(2)}(t)},                   \label{b:eq:19}
\end{equation}
and 
\begin{equation}
 \eta_{\tau}(t)\frac{\mathcal K_{n,\tau}(q^3t)}{\mathcal K_{n,\tau}(t)}
 =\rho_{\tau}(t):=\frac{t(t-1)}{x^3(t-q^{2\tau-3}x^3)}
             \frac{H_{\tau-2,\tau}^{(2)}(t)}{H_{\tau-3,\tau-1}^{(1)}(t)}.
                                                                  \label{b:eq:20}
\end{equation}
For \eqref{b:eq:19}, the numerator gains \(1-q^{2\tau}x^3/t\); the old two starting denominator factors are removed and four ending factors are added. Pair the factors with coefficients \(\omega\) and \(\omega^2\), using
\((t-\omega u)(t-\omega^2u)=t^2+ut+u^2\).
For \eqref{b:eq:20}, shifting the \(q\)-factorial strings by three places leaves exactly the three paired factors at each end. The numerator quotient is
\[
 \frac{(q^{2\tau-3}/t;q^3)_n}{(q^{2\tau}/t;q^3)_n}
 =\frac{1-q^{2\tau-3}/t}{1-q^{2\tau-3}x^3/t}.
\]
This cancels the corresponding factor in \(\eta_{\tau}\). The power of \(t\) contributes \(q^{-3n-3}\), and \(q^{3n}=x^3\). These are finite-product identities for every \(n\ge0\); shortened strings at \(n=0,1\) are handled by rational cancellation.

Set
\begin{align*}
 K_j(x)&:=\prod_{r=1}^j(1-q^r x),\\
 V_{j,\tau}(t)&:=\prod_{i=0}^{j-1}(t-q^{2\tau+3i}x^3),
\\
 \Delta_{\tau}(t)&:=t^2H_{\tau+1,\tau+6}^{(2)}(t),
\end{align*}
and 
 define the degree-\(14\) polynomial
\[
 \begin{aligned}
 L_{\tau}(t):={}&\sum_{j=0}^3\lambda_j^{(\tau)}(x)K_j(x)
 (t^2-q^{2\tau+3j}x^3)V_{j,\tau}(t) H_{\tau,\tau+j-1}^{(1)}(t)H_{\tau+2j+1,\tau+6}^{(2)}(t).
 \end{aligned}                
\]
Indeed, the \(j\)-th summand has \(t\)-degree
\[
2+j+2j+2(6-2j)=14-j.
\]
Since the \(j=0\) summand has nonzero leading coefficient
\(\lambda_0^{(\tau)}(x)\), it follows that
\(\deg_t L_\tau(t)=14\).

Finally put
\[
 B_{\tau}(t):=(t-q^{2\tau}x^3)H_{\tau,\tau+2}^{(1)}(t),\qquad
 C_{\tau}(t):=t(t-1)H_{\tau+4,\tau+6}^{(2)}(t).
\]
\begin{lemma}\label{b:lem:certificates}
	There are explicit degree-six polynomials $\mathcal P_\tau(t)$,
	listed in Appendix~\ref{app:coefficients}, such that
\begin{equation}
 L_{\tau}(t)=B_{\tau}(t)\mathcal P_{\tau}(t)
       +\frac{C_{\tau}(t)\mathcal P_{\tau}(q^3t)}{q^{21}x^3},\qquad \tau=1,2.     \label{b:eq:22}
\end{equation}
\end{lemma}

\begin{proof}
	We regard $q$, $x$, and $t$ as algebraically independent
	indeterminates.
		For completeness, the seven coefficients of
	$\mathcal P_\tau(t)$ can  be recovered directly from
	$L_\tau(t)$.  Start with
	\[
	R_6(t)=L_\tau(t).
	\]
	For $k=6,5,\ldots,0$, define
	\[
	\kappa_{\tau,k}
	:=
	q^{21-3k}x^3[t^{k+8}]R_k(t),
	\]
	and
	\[
	R_{k-1}(t)
	:=
	R_k(t)
	-
	\kappa_{\tau,k}
	\left(
	t^kB_\tau(t)
	+
	\frac{q^{3k}t^kC_\tau(t)}
	{q^{21}x^3}
	\right).
	\]
	Then
	\[
	\mathcal P_\tau(t)
	=
	\sum_{k=0}^6\kappa_{\tau,k}t^k.
	\]
	This descending procedure reproduces exactly the polynomials listed
	in   Appendix~\ref{app:coefficients}.

	Substituting these explicit polynomials into \eqref{b:eq:22},
	the identity is verified exactly in Maple using the
	\texttt{normal} command.  More precisely, for each $\tau=1,2$,
	Maple returns zero for
	\[
	\texttt{normal}\!\left(
	L_\tau(t)
	-
	B_\tau(t)\mathcal P_\tau(t)
	-
	\frac{C_\tau(t)\mathcal P_\tau(q^3t)}
	{q^{21}x^3}
	\right).
	\]
	Hence \eqref{b:eq:22} holds identically in the rational function
	field $\mathbb Q(q,x,t)$.  No numerical specialization of
	$q$, $x$, or $t$, and no truncated $q$-series computation, is
	involved.  The Maple source code and complete output are available
	at 
	\begin{center}
		\url{https://github.com/Ernest-Xia/Kanade--Russell-identities-Lemma-6.3}.
	\end{center}
\end{proof}

To connect \eqref{b:eq:22} to residues, define
\[
 \mathcal R_{\tau}(t):=\frac{B_{\tau}(t)\mathcal P_{\tau}(t)}{\Delta_{\tau}(t)}.
\]
The elementary identity
\[
 h_e^{(d)}(q^3t)=q^6h_{e-3}^{(d)}(t)
\]
and \eqref{b:eq:20} give
\begin{equation}
 \rho_{\tau}(t)\mathcal R_{\tau}(q^3t)
 =\frac{C_{\tau}(t)\mathcal P_{\tau}(q^3t)}{q^{21}x^3\Delta_{\tau}(t)}.                \label{b:eq:24}
\end{equation}

Let
\[
 T_{\tau}(t)=\sum_{j=0}^3\lambda_j^{(\tau)}(x)K_j(x)
 (1-q^{2\tau+3j}x^3/t^2)\frac{\mathcal K_{n+j,\tau}(t)}{\mathcal K_{n,\tau}(t)},
 \qquad x=q^n.
\]
Iterating \eqref{b:eq:19} shows \(\Delta_{\tau} (t) T_{\tau}(t)=L_{\tau}(t)\). Therefore \eqref{b:eq:22}--\eqref{b:eq:24} give
\[
 T_{\tau}(t)=\mathcal R_{\tau}(t)+\rho_{\tau}(t)\mathcal R_{\tau}(q^3t).
\]
Multiplication by \(\mathcal K_{n,\tau}(t)\), followed by \eqref{b:eq:16}, proves
\begin{equation}
 \sum_{j=0}^3\lambda_j^{(\tau)}(q^n)Z_{n+j}^{(\tau)}=0
 \quad(n\ge0).                                               \label{b:eq:25}
\end{equation}
In this last step,
\((q;q)_{n+j}/(q;q)_n=K_j(q^n)\), which accounts for the normalization. This proves the recurrence at every order, rather than at finitely many sampled orders.

\subsection{Boundary estimates, boundary uniqueness and completion}

For the nodes defining $Z_n^{(\tau)}$, apply
Lemma~\ref{lem:dd-expansion} with
\[
 d=3,\quad c=3n+2\tau,\quad
 \alpha_i=\omega q^{n+\tau+i},\quad
 \beta_i=\omega^2q^{2n+\tau-i}.
\]
Thus the evaluated factors are
\[
 d_i(j)=(1-\omega q^{n+\tau+i+3j})
        (1-\omega^2q^{2n+\tau-i+3j}).
\]

The finite-product argument and the coefficientwise passage to the
limit are the same as in Lemma~2.2. Indeed, for
\(0\le i\le n\) and \(j\ge0\), both exponents
$
n+\tau+i+3j$ and $
2n+\tau-i+3j
$
are at least \(n+\tau\). Hence
\[
d_i(j)^{-1}=1+O(q^{n+\tau}).
\]
Consequently, replacing every evaluated reciprocal factor by \(1\)
up to \(q\)-order \(n+\tau-1\) gives
\[
 \begin{aligned}
 \Phi_{3,3n+2\tau}[z_0,\ldots,z_n]
 &=\sum_{0\le k_1\le\cdots\le k_n}q^{3(k_1+\cdots+k_n)}
  +O(q^{n+\tau})\\
 &=\frac1{(q^3;q^3)_n}+O(q^{n+\tau}).
 \end{aligned}                                
\]
The last equality follows on writing \(h_1=k_1\), \(h_j=k_j-k_{j-1}\) and summing the resulting independent geometric series. Complex coefficients \(\omega,\omega^2\) do not affect these formal order estimates.

Since
\[
 \frac{\mathcal E_3(q;q)_n}{(q^3;q^3)_n}
 =\frac{(q^{3n+3};q^3)_\infty}{(q^{n+1};q)_\infty}
 =1+O(q^{n+1}),
\]
we obtain
\begin{equation}
 Z_n^{(\tau)}=1+O(q^{n+1}),\qquad \tau=1,2.                        \label{b:eq:28}
\end{equation}
The two sum-side families satisfy the same sufficient bounds:
\begin{equation}
 Y_n^{(1)}=1+O(q^{2n+2}),\qquad
 Y_n^{(2)}=1+O(q^{n+1}).                                    \label{b:eq:29}
\end{equation}

 For fixed \(n\), the Lagrange formula for \eqref{b:eq:1} is
 meromorphic at \(q=0\), with at most a finite-order pole.
 Representation \eqref{eq:dd-expansion} shows that its negative Laurent
 coefficients cancel, so \(Z_n^{(\tau)}\) extends analytically
 to \(q=0\). Since both sides of \eqref{b:eq:25} are analytic near
 \(q=0\) and agree for all sufficiently small positive real \(q\),
 the identity theorem implies that \eqref{b:eq:25} also holds as an
 identity in \(\mathbb C[[q]]\).

\begin{proof}[Proof of Theorem~\ref{thm:cyclotomic}]
For either value of \(\tau\), let \(e_n=Z_n^{(\tau)}-Y_n^{(\tau)}\). Equations \eqref{b:eq:11} and \eqref{b:eq:25} imply
\[
 \begin{pmatrix}e_n\\e_{n+1}\\e_{n+2}\end{pmatrix}
 =\mathsf C_n^{(\tau)}\begin{pmatrix}e_{n+1}\\e_{n+2}\\e_{n+3}\end{pmatrix},
\]
where the first row of \(\mathsf C_n^{(\tau)}\) is
\[
 -\frac1{\lambda_0^{(\tau)}(q^n)}
 (\lambda_1^{(\tau)}(q^n),\lambda_2^{(\tau)}(q^n),\lambda_3^{(\tau)}(q^n)),
\]
and its last two rows are \((1,0,0)\), \((0,1,0)\).
All entries belong to \(\mathbb C[[q]]\): for \(\tau=1\) the leading coefficient is 1; for \(\tau=2\) it is \(1+q^{n+2}+q^{2n+5}\), a formal unit. Consequently matrix multiplication cannot lower the least \(q\)-degree.

By \eqref{b:eq:28} and \eqref{b:eq:29}, 
 we have
\[
e_n=O(q^{n+1}).
\]
Consequently,
\[
\begin{pmatrix}
	e_N\\ e_{N+1}\\ e_{N+2}
\end{pmatrix}
=
\begin{pmatrix}
	O(q^{N+1})\\ O(q^{N+2})\\ O(q^{N+3})
\end{pmatrix}
=O(q^{N+1}),
\]
where the final order estimate is understood componentwise.
Iterating the recurrence from any fixed initial level \(n\) to an arbitrarily large level \(N\) shows that \(e_n\) vanishes to arbitrarily high order. 
 Every coefficient of \(e_n\) is therefore zero. This proves
 Theorem~\ref{thm:cyclotomic}. 
 \end{proof}
 
 \begin{proof}[Proofs of \eqref{eq:kr4} and \eqref{eq:kr5}]
 Setting   \(n=0\) in \eqref{b:eq:1} and using  the factorization
 \[
 (1-u)(1-\omega u)(1-\omega^2u)=1-u^3
 \]
 gives
 \begin{equation}
 	Z_0^{(\tau)}=\mathcal E_3 \Phi_{3, 2\tau}[z_0]
 	=\frac{\mathcal E_3}{(\omega q^{\tau},\omega^2q^{\tau};q^3)_\infty}
 	=\mathcal E_3\frac{(q^{\tau};q^3)_\infty}{(q^{3\tau};q^9)_\infty}.                   \label{b:eq:4}
 \end{equation}
  The  specialization of Theorem~\ref{thm:cyclotomic} at \(n=0\),
 together with \eqref{b:eq:4}, gives  
 \begin{equation}
 	S(1,2)=\frac1{(q^2;q^3)_\infty(q^3;q^9)_\infty}
 	=\frac1{(q^2,q^3,q^5,q^8;q^9)_\infty},                   \label{b:eq:5}
 \end{equation}
 and 
 \begin{equation}
 	\begin{aligned}
 		(1+q)S(2,4)+q^2S(3,7)
 		&=\frac1{(q,q^4,q^6,q^7;q^9)_\infty}.
 	\end{aligned}                          \label{b:eq:6}
 \end{equation}
 Equation \eqref{b:eq:5} is precisely  \eqref{eq:kr4}.  
 
 We note from
 Lemma~\ref{lem:contiguous} with $(a,b)=(1,4)$ that
 \[
 S(1,4)-S(2,4)=q^2S(3,7),
 \]
 and hence
 \begin{align}\label{eq:fifth-equivalent}
 S(1,4)+qS(2,4)
 =
 (1+q)S(2,4)+q^2S(3,7).
 \end{align}
  Identity \eqref{eq:kr5} follows from 
   \eqref{b:eq:6} and  \eqref{eq:fifth-equivalent}.

All final sums and products are analytic throughout \(|q|<1\). Their formal equalities thus extend as analytic equalities to that disk. No numerical limit of a growing matrix product is needed.
\end{proof}

\section{Proofs of the three Wang--Wang dual identities}\label{sec:wang-wang}

Throughout the dual sections below, we work with the negative-mixed-term
quadratic form
\[
Q(r,s):=r^2-3rs+3s^2.
\]
Thus
\[
\mathfrak{N}(a,b)
=
\sum_{r,s\geq0}
\frac{q^{Q(r,s)+ar+bs}}
{(q;q)_r(q^3;q^3)_s}.
\]
This quadratic form is the dual counterpart of the positive-mixed-term
form $r^2+3rs+3s^2$ occurring in the definition of $S(a,b)$ in
\eqref{eq:S}.
Moreover,
\[
Q(r,s)
=
\left(r-\frac32s\right)^2+\frac34s^2,
\]
so it is positive definite.

The three identities \eqref{eq:ww1}--\eqref{eq:ww3} follow from
bilinear identities between $S(a,b)$ and $\mathfrak N(a,b)$.
We prove these bilinear identities first, using only contiguous
relations and the boundary uniqueness principle of
Lemma~\ref{lem:boundary-unique}. Product evaluations enter only in
Section~\ref{sec:wang-wang-products}.

\subsection{The bilinear identities and the common recurrence}

\begin{theorem}\label{ww:bilinear}
For every integer $n\ge0$, define
\begin{align*}
 \alpha_n&:=S(3n,3n),&\beta_n&:=S(3n+1,3n+3),&\gamma_n&:=S(3n+2,3n+3),\\
 \widehat\alpha_n&:=S(0,3n),&\widehat\beta_n&:=S(1,3n+3),&\widehat\gamma_n&:=S(2,3n+3).
\end{align*}
Then
\begin{align}
 \mathfrak N(0,3n)&=\alpha_n\widehat\alpha_n+q^{3n+1}\beta_n\widehat\gamma_n,\label{ww:bilinX}\\
 \mathfrak N(-1,3n+3)&=\beta_n\widehat\beta_n+\gamma_n\widehat\alpha_n,\label{ww:bilinY}\\
 \mathfrak N(1,3n)&=\alpha_n\widehat\beta_n-q^{3n+1}\gamma_n\widehat\gamma_n.\label{ww:bilinZ}
\end{align}
In particular,
\begin{align}
 \mathfrak N(0,0)&=S(0,0)^2+qS(1,3)S(2,3),\label{ww:at0X}\\
 \mathfrak N(-1,3)&=S(1,3)^2+S(0,0)S(2,3),\label{ww:at0Y}\\
 \mathfrak N(1,0)&=S(0,0)S(1,3)-qS(2,3)^2.\label{ww:at0Z}
\end{align}
\end{theorem}

The proof of Theorem~\ref{ww:bilinear} will be completed
at the end of Section~7.3, after establishing the common recurrence,
the row-vector certificates, and the required boundary estimates.

 The positive-mixed-term relations \eqref{2-1} and \eqref{2-2} were established in
 Section 2. We first record the corresponding contiguous relations for
 the negative-mixed-term family.
 
 \begin{lemma}
 	For all integers $a$ and $b$, direct index shifts give
 		\begin{align}
 			\mathfrak{N}(a,b)-\mathfrak{N}(a+1,b)
 			&=q^{a+1}\mathfrak{N}(a+2,b-3),
 			\label{eq:dual-contiguous-a}\\
 			\mathfrak{N}(a,b)-\mathfrak{N}(a,b+3)
 			&=q^{b+3}\mathfrak{N}(a-3,b+6).
 			\label{eq:dual-contiguous-b}
 		\end{align}
 \end{lemma}
 
 \begin{proof}
 	By the definition of $\mathfrak{N}(a,b)$,
 	\begin{align*}
 		\mathfrak{N}(a,b)-\mathfrak{N}(a+1,b)
 		&=
 		\sum_{r,s\geq 0}
 		\frac{q^{r^2-3rs+3s^2+ar+bs}}
 		{(q;q)_r(q^3;q^3)_s}
 		-
 		\sum_{r,s\geq 0}
 		\frac{q^{r^2-3rs+3s^2+(a+1)r+bs}}
 		{(q;q)_r(q^3;q^3)_s}\\
 		&=
 		\sum_{\substack{r\geq 1\\ s\geq 0}}
 		\frac{q^{r^2-3rs+3s^2+ar+bs}(1-q^r)}
 		{(q;q)_r(q^3;q^3)_s}\\
 		&=
 		\sum_{\substack{r\geq 1\\ s\geq 0}}
 		\frac{q^{r^2-3rs+3s^2+ar+bs}}
 		{(q;q)_{r-1}(q^3;q^3)_s}\\
 		&=
 		\sum_{r,s\geq 0}
 		\frac{q^{(r+1)^2-3(r+1)s+3s^2+a(r+1)+bs}}
 		{(q;q)_r(q^3;q^3)_s}.
 	\end{align*}
 	Since
 	\[
 	(r+1)^2-3(r+1)s+3s^2+a(r+1)+bs
 	=
 	r^2-3rs+3s^2+(a+2)r+(b-3)s+a+1,
 	\]
 	we obtain
 	\begin{align*}
 		\mathfrak{N}(a,b)-\mathfrak{N}(a+1,b)
 		&=
 		q^{a+1}
 		\sum_{r,s\geq 0}
 		\frac{q^{r^2-3rs+3s^2+(a+2)r+(b-3)s}}
 		{(q;q)_r(q^3;q^3)_s}\\
 		&=
 		q^{a+1}\mathfrak{N}(a+2,b-3),
 	\end{align*}
 	which proves \eqref{eq:dual-contiguous-a}.
 	
 	Similarly,
 	\begin{align*}
 		\mathfrak{N}(a,b)-\mathfrak{N}(a,b+3)
 		&=
 		\sum_{r,s\geq 0}
 		\frac{q^{r^2-3rs+3s^2+ar+bs}}
 		{(q;q)_r(q^3;q^3)_s}
 		-
 		\sum_{r,s\geq 0}
 		\frac{q^{r^2-3rs+3s^2+ar+(b+3)s}}
 		{(q;q)_r(q^3;q^3)_s}\\
 		&=
 		\sum_{\substack{r\geq 0\\ s\geq 1}}
 		\frac{q^{r^2-3rs+3s^2+ar+bs}(1-q^{3s})}
 		{(q;q)_r(q^3;q^3)_s}\\
 		&=
 		\sum_{\substack{r\geq 0\\ s\geq 1}}
 		\frac{q^{r^2-3rs+3s^2+ar+bs}}
 		{(q;q)_r(q^3;q^3)_{s-1}}\\
 		&=
 		\sum_{r,s\geq 0}
 		\frac{q^{r^2-3r(s+1)+3(s+1)^2+ar+b(s+1)}}
 		{(q;q)_r(q^3;q^3)_s}.
 	\end{align*}
 	Since
 	\[
 	r^2-3r(s+1)+3(s+1)^2+ar+b(s+1)
 	=
 	r^2-3rs+3s^2+(a-3)r+(b+6)s+b+3,
 	\]
 	it follows that
 	\begin{align*}
 		\mathfrak{N}(a,b)-\mathfrak{N}(a,b+3)
 		&=
 		q^{b+3}
 		\sum_{r,s\geq 0}
 		\frac{q^{r^2-3rs+3s^2+(a-3)r+(b+6)s}}
 		{(q;q)_r(q^3;q^3)_s}\\
 		&=
 		q^{b+3}\mathfrak{N}(a-3,b+6),
 	\end{align*}
 	which proves \eqref{eq:dual-contiguous-b}.
 \end{proof}
  
  Put
\[
 \mathsf X_n:=\mathfrak N(0,3n),\qquad \mathsf Y_n:=\mathfrak N(-1,3n+3),\qquad \mathsf Z_n:=\mathfrak N(1,3n).
\]
Equations \eqref{eq:dual-contiguous-a} and \eqref{eq:dual-contiguous-b} give
\begin{align}
 \mathsf Y_n&=\mathsf X_{n+1}+\mathsf Z_n,\label{ww:NX1}\\
 \mathsf Z_n&=q^{3n+2}\mathsf X_{n+1}+q^{3n+3}\mathsf Y_{n+1}+\mathsf Z_{n+1},\label{ww:NX2}\\
 \mathsf X_n&=(1+q^{3n+1})\mathsf Y_n-\mathsf Z_{n+1}.\label{ww:NX3}
\end{align}
For \eqref{ww:NX2}, use
\[
\mathsf Z_n-\mathsf Z_{n+1}=q^{3n+3}\mathfrak N(-2,3n+6)
\]
 and
\[
\mathfrak N(-2,3n+6)=\mathsf Y_{n+1}+q^{-1}\mathsf X_{n+1}.
\]
For \eqref{ww:NX3}, the identities
\[
 \mathsf X_n-\mathsf Z_n=q\mathfrak N(2,3n-3),\ \ 
 \mathfrak N(2,3n-3)-\mathfrak N(2,3n)=q^{3n}\mathsf Y_n,\ \ 
 q\mathfrak N(2,3n)=\mathsf X_{n+1}-\mathsf Z_{n+1}
\]
together with \eqref{ww:NX1} suffice. Thus
\begin{equation}\label{ww:Drec}
 \begin{pmatrix}\mathsf X_n\\\mathsf Y_n\\\mathsf Z_n\end{pmatrix}
 =\mathsf D_n\begin{pmatrix}\mathsf X_{n+1}\\\mathsf Y_{n+1}\\\mathsf Z_{n+1}\end{pmatrix},
\end{equation}
where
\begin{equation} \label{D_n-def}
 \mathsf D_n:=\begin{pmatrix}
		(1+q^{3n+1})(1+q^{3n+2})&q^{3n+3}(1+q^{3n+1})&q^{3n+1}\\
		1+q^{3n+2}&q^{3n+3}&1\\
		q^{3n+2}&q^{3n+3}&1
	\end{pmatrix}.
\end{equation}
Subtracting the third row from the second and expanding along
the resulting row $(1,0,0)$ gives
\begin{align}
\det\mathsf D_n
=-\det\begin{pmatrix}
	q^{3n+3}(1+q^{3n+1})&q^{3n+1}\\
	q^{3n+3}&1
\end{pmatrix}
=-q^{3n+3}.\label{7-14}
\end{align}
All entries of $\mathsf D_n$ belong to $\mathbb C[[q]]$.

\subsection{Two exact row-vector certificates}

\begin{lemma}\label{lem:row-vector-certificates}
	Let
	\[
	r_n=(-\gamma_n,\alpha_n,-\beta_n),
	\qquad
	s_n=(-\widehat{\beta}_n,
	q^{3n+1}\widehat{\gamma}_n,
	\widehat{\alpha}_n).
	\]
	 Then, for every $n\geq 0$,
	\begin{equation}\label{eq:row-vector-certificates}
		r_n \mathsf D_n=-q^{3n+3}r_{n+1},
		\qquad
		s_n\mathsf D_n=s_{n+1},
	\end{equation}
	where $\mathsf D_n$ is defined by \eqref{D_n-def}.  
\end{lemma}

\begin{proof}
	We verify the two identities separately, using only   \eqref{2-1} and \eqref{2-2}.
	
	For the first identity, the following three relations are sufficient:
	\begin{align}
		\alpha_n-\beta_n
		&=
		q^{3n+1}\gamma_n+q^{3n+3}\beta_{n+1},
		\label{eq:first-cert-1}\\
		\gamma_n
		&=
		\alpha_{n+1}+q^{3n+3}\beta_{n+1},
		\label{eq:first-cert-2}\\
		\beta_n
		&=
		(1+q^{3n+2})\alpha_{n+1}
		+q^{3n+3}\gamma_{n+1}.
		\label{eq:first-cert-3}
	\end{align}
	Indeed, \eqref{eq:first-cert-1} follows by inserting
	$S(3n+1,3n)$ between $\alpha_n$ and $\beta_n$.
	Relation \eqref{eq:first-cert-2} is \eqref{2-1} with
	$(a,b)=(3n+2,3n+3)$.
	
%	To prove \eqref{eq:first-cert-3}, put
%	\[
%	T_n=S(3n+3,3n+6),
%	\qquad
%	J_n=S(3n+6,3n+9).
%	\]

	Then \eqref{2-1} and \eqref{2-2} give
	\begin{align*}
	\beta_n-\gamma_n&=q^{3n+2}S(3n+3,3n+6),
	\\
	\alpha_{n+1}-S(3n+3,3n+6)&=q^{3n+6}S(3n+6,3n+9),
	\\
	\beta_{n+1}-\gamma_{n+1}&=q^{3n+5}S(3n+6,3n+9).
	\end{align*}
	Consequently,
	\begin{align*}
		\beta_n-\alpha_{n+1}
		&=
		q^{3n+2}S(3n+3,3n+6) +q^{3n+3}\beta_{n+1}\\
		&=
		q^{3n+2}\alpha_{n+1}
		+q^{3n+3}\gamma_{n+1}\\
		&\quad+
		\bigl(q^{3n+3}q^{3n+5}
		-q^{3n+2}q^{3n+6}\bigr)S(3n+6,3n+9)\\
		&=
		q^{3n+2}\alpha_{n+1}
		+q^{3n+3}\gamma_{n+1},
	\end{align*}
	which proves \eqref{eq:first-cert-3}.
	
	We now multiply $r_n$ by $\mathsf D_n$. The third component is
	\begin{align*}
		-q^{3n+1}\gamma_n+\alpha_n-\beta_n
		&=
		q^{3n+3}\beta_{n+1},
	\end{align*}
	by \eqref{eq:first-cert-1}. The second component is
	\begin{align*}
		&q^{3n+3}
		\bigl(
		\alpha_n-\beta_n-(1+q^{3n+1})\gamma_n
		\bigr) =
		q^{3n+3}
		\bigl(q^{3n+3}\beta_{n+1}-\gamma_n\bigr) =
		-q^{3n+3}\alpha_{n+1},
	\end{align*}
	where \eqref{eq:first-cert-1} and
	\eqref{eq:first-cert-2} were used.
	
	For the first component, we have
	\begin{align*}
		 -(1+q^{3n+1})(1+q^{3n+2})\gamma_n
		&+(1+q^{3n+2})\alpha_n
		-q^{3n+2}\beta_n\\
		&=
		(1+q^{3n+2})
		\bigl(\alpha_n-(1+q^{3n+1})\gamma_n\bigr)
		-q^{3n+2}\beta_n.
	\end{align*}
	By \eqref{eq:first-cert-1},
	\[
	\alpha_n-(1+q^{3n+1})\gamma_n
	=
	\beta_n-\gamma_n+q^{3n+3}\beta_{n+1}.
	\]
	Hence the first component equals
	\begin{align*}
		&(1+q^{3n+2})
		\bigl(\beta_n-\gamma_n+q^{3n+3}\beta_{n+1}\bigr)
		-q^{3n+2}\beta_n\\
		&\qquad \qquad =
		\beta_n-(1+q^{3n+2})\gamma_n
		+q^{3n+3}(1+q^{3n+2})\beta_{n+1}.
	\end{align*}
	Using \eqref{eq:first-cert-2} and
	\eqref{eq:first-cert-3}, this becomes
	$
	q^{3n+3}\gamma_{n+1}.
	$
	Therefore, 
	\[
	r_n \mathsf D_n
	=
	\bigl(
	q^{3n+3}\gamma_{n+1},
	-q^{3n+3}\alpha_{n+1},
	q^{3n+3}\beta_{n+1}
	\bigr)
	=
	-q^{3n+3}r_{n+1}.
	\]
	
	We next prove the second identity. The required relations are
	\begin{align}
		\widehat{\alpha}_{n+1}
		&=
		\widehat{\alpha}_n
		+q^{3n+1}
		\bigl(
		\widehat{\gamma}_n-\widehat{\beta}_n
		\bigr),
		\label{eq:second-cert-1}\\
		\widehat{\alpha}_{n+1}-\widehat{\beta}_n
		&=
		q\widehat{\gamma}_{n+1},
		\label{eq:second-cert-2}\\
		\widehat{\alpha}_n-\widehat{\beta}_{n+1}
		&=
		(q+q^{3n+3})\widehat{\gamma}_{n+1}.
		\label{eq:second-cert-3}
	\end{align}
	For \eqref{eq:second-cert-1}, the contiguous relations give
	\[
	\widehat{\alpha}_n-\widehat{\alpha}_{n+1}
	=
	q^{3n+3}S(3,3n+6),
	\]
	and
	\[
	\widehat{\beta}_n-\widehat{\gamma}_n
	=
	q^2S(3,3n+6).
	\]
Combining the above two identities gives 
 \eqref{eq:second-cert-1}.
	
	Equation \eqref{eq:second-cert-2} is \eqref{2-1} with
	$(a,b)=(0,3n+3)$.

Equation \eqref{2-2} gives 
	\begin{align*}
	\widehat{\alpha}_n-\widehat{\alpha}_{n+1}
	&=
	q^{3n+3}S(3,3n+6),
\\
	\widehat{\beta}_n-\widehat{\beta}_{n+1}
	&=
	q^{3n+6}S(4,3n+9),
	\end{align*}
	and
	\[
	\widehat{\gamma}_{n+1}=S(3,3n+6)+q^3S(4,3n+9).
	\]
	Combining these identities with
	\eqref{eq:second-cert-2} gives
	\[
	\widehat{\alpha}_n-\widehat{\beta}_{n+1}
	=
	(q+q^{3n+3})\widehat{\gamma}_{n+1},
	\]
	as required.
	
	We now compute the three components of $s_n \mathsf D_n$.
	The third component is
	\begin{align*}
		-q^{3n+1}\widehat{\beta}_n
		+q^{3n+1}\widehat{\gamma}_n
		+\widehat{\alpha}_n
		=
		\widehat{\alpha}_{n+1},
	\end{align*}
	by \eqref{eq:second-cert-1}.
	
	The second component is
	\begin{align*}
		&q^{3n+3}
		\Bigl(
		\widehat{\alpha}_n
		-(1+q^{3n+1})\widehat{\beta}_n
		+q^{3n+1}\widehat{\gamma}_n
		\Bigr)\\
		&=
		q^{3n+3}
		\bigl(
		\widehat{\alpha}_{n+1}
		-\widehat{\beta}_n
		\bigr)\\
		&=
		q^{3n+4}\widehat{\gamma}_{n+1},
	\end{align*}
	by \eqref{eq:second-cert-1} and
	\eqref{eq:second-cert-2}.
	
	Finally, the first component is
	\begin{align*}
		&-(1+q^{3n+1})(1+q^{3n+2})
		\widehat{\beta}_n
		+q^{3n+1}(1+q^{3n+2})
		\widehat{\gamma}_n
		+q^{3n+2}\widehat{\alpha}_n.
	\end{align*}
	From \eqref{eq:second-cert-1} and
	\eqref{eq:second-cert-2},
	\[
	(1+q^{3n+1})\widehat{\beta}_n
	-q^{3n+1}\widehat{\gamma}_n
	=
	\widehat{\alpha}_n-q\widehat{\gamma}_{n+1}.
	\]
	Therefore the first component becomes
	\begin{align*}
		-(1+q^{3n+2})
		\bigl(
		\widehat{\alpha}_n-q\widehat{\gamma}_{n+1}
		\bigr)
		+q^{3n+2}\widehat{\alpha}_n
		&=
		-\widehat{\alpha}_n
		+q(1+q^{3n+2})\widehat{\gamma}_{n+1}\\
		&=
		-\widehat{\alpha}_n
		+(q+q^{3n+3})\widehat{\gamma}_{n+1}\\
		&=
		-\widehat{\beta}_{n+1},
	\end{align*}
	where the last equality follows from
	\eqref{eq:second-cert-3}. Hence
	\[
	s_n \mathsf D_n
	=
	\bigl(
	-\widehat{\beta}_{n+1},
	q^{3n+4}\widehat{\gamma}_{n+1},
	\widehat{\alpha}_{n+1}
	\bigr)
	=
	s_{n+1}.
	\]
	This completes the proof.
\end{proof}

\subsection{The cross product and the boundary}

For row vectors
$r=(r_1,r_2,r_3)$ and $s=(s_1,s_2,s_3)$, write
their usual cross product as the column vector
\[
\operatorname{cr}(r,s)
=
\begin{pmatrix}
	r_2s_3-r_3s_2\\
	r_3s_1-r_1s_3\\
	r_1s_2-r_2s_1
\end{pmatrix}.
\]
Thus, from
\[
r_n=(-\gamma_n,\alpha_n,-\beta_n),\qquad
s_n=(-\widehat{\beta}_n,
q^{3n+1}\widehat{\gamma}_n,\widehat{\alpha}_n),
\]
we obtain
\begin{equation}\label{ww:cross}
 \widetilde{V}_n:=\operatorname{cr}(r_n,s_n)
 =\begin{pmatrix}
 \alpha_n\widehat\alpha_n+q^{3n+1}\beta_n\widehat\gamma_n\\
 \beta_n\widehat\beta_n+\gamma_n\widehat\alpha_n\\
 \alpha_n\widehat\beta_n-q^{3n+1}\gamma_n\widehat\gamma_n
 \end{pmatrix}.
\end{equation}
The elementary identity
\[
 \operatorname{cr}(rM,sM)=\operatorname{adj}(M)\operatorname{cr}(r,s)
\]
for $3\times3$ matrices, together with \eqref{eq:row-vector-certificates}, gives
\[
 -q^{3n+3}\widetilde{V}_{n+1}=\operatorname{adj}(\mathsf D_n)\widetilde{V}_n.
\]
Multiplying by $\mathsf D_n$ and using   $\det \mathsf D_n=-q^{3n+3}$ from
   \eqref{7-14} yields
\begin{equation}\label{ww:crossrec}
 \widetilde{V}_n=\mathsf D_n\widetilde{V}_{n+1}.
\end{equation}
Cancellation of the nonzero monomial $q^{3n+3}$ is valid in the integral domain
$\mathbb C[[q]]$; no assertion about negative powers is needed.

 Separating the
$s=0$ terms in \eqref{eq:dual-sum} gives
\begin{equation}\label{ww:Nbound}
 \begin{pmatrix}\mathsf X_n\\\mathsf Y_n\\\mathsf Z_n\end{pmatrix}
 =\begin{pmatrix}\mathcal G(q)\\\mathcal G(q)+\mathcal H(q)\\\mathcal H(q)\end{pmatrix}+O(q^{3n+1}),
\end{equation}
where $\mathcal G(q)$ and $\mathcal H(q)$ are defined by \eqref{RR-1}
 and \eqref{RR-2}, respectively.
 
 To justify \eqref{ww:Nbound}, first separate the $s=0$
 terms in the three dual sums.  For $\mathsf X_n=\mathfrak{N}(0,3n)$,
 the contribution with $s=0$ is
 $\mathcal G(q)$. 
 For $\mathsf Z_n=\mathfrak{N}(1,3n)$, the corresponding contribution is
  $\mathcal H(q)$. 
 For the middle component,
 \[
 \sum_{r\geq0}\frac{q^{r(r-1)}}{(q;q)_r}=\mathcal G(q)+\mathcal H(q).
 \]
 Indeed, subtracting $\mathcal G(q)$ from the left-hand side gives
 \[
 \sum_{r\geq1}
 \frac{q^{r(r-1)}(1-q^r)}{(q;q)_r}
 =
 \sum_{r\geq1}
 \frac{q^{r(r-1)}}{(q;q)_{r-1}}
 =
 \mathcal H(q),
 \]
 after shifting $r\mapsto r+1$.
 
 It remains to estimate the terms with $s\geq1$.  Their least possible
 $q$-orders in $\mathsf X_n$, $\mathsf Y_n$, and $\mathsf Z_n$ are at least
 $3n+1$, $3n+2$, and $3n+2$, respectively.  Hence
 \[
 \mathsf X_n=\mathcal G(q)+O(q^{3n+1}),\qquad
 \mathsf Y_n=\mathcal G(q)+\mathcal H(q)+O(q^{3n+2}),\qquad
 \mathsf Z_n=\mathcal H(q)+O(q^{3n+2}),
 \]
 which yields the common vector estimate
 \eqref{ww:Nbound}.

On the positive-cross-term side,
\begin{align*}
 \alpha_n&=1+O(q^{3n+1}),&
 \beta_n&=1+O(q^{3n+2}),&
 \gamma_n&=1+O(q^{3n+3}),\\
 \widehat\alpha_n&=\mathcal G(q)+O(q^{3n+3}),&
 \widehat\beta_n&=\mathcal H(q)+O(q^{3n+6}),&
 \widehat\gamma_n&\in\mathbb C[[q]].
\end{align*}
These estimates follow directly from the definition of $S(a,b)$.
For $\alpha_n$, $\beta_n$, and $\gamma_n$, the term $(r,s)=(0,0)$
contributes $1$, while the least positive $q$-degrees are
$3n+1$, $3n+2$, and $3n+3$, respectively. For the hatted families, separate the $s=0$ terms. 
We have
\[
\widehat{\alpha}_n=\mathcal G(q)+O(q^{3n+3}),
\qquad
\widehat{\beta}_n=\mathcal H(q)+O(q^{3n+6}),
\]
because the least degrees among the terms with $s\geq1$ are
$3n+3$ and $3n+6$, respectively.  Finally,
$\widehat{\gamma}_n=S(2,3n+3)\in\mathbb{C}[[q]]$
(and in fact $\widehat{\gamma}_n=1+O(q^3)$).

Hence \eqref{ww:cross} satisfies the same boundary estimate:
\begin{equation}\label{ww:crossbound}
 \widetilde{V}_n=\begin{pmatrix}\mathcal G(q)\\\mathcal G(q)+\mathcal H(q)\\\mathcal H(q)\end{pmatrix}+O(q^{3n+1}).
\end{equation}

 \begin{proof}[Proof of Theorem~\ref{ww:bilinear}]
 	Set
 	\[
 	 \mathsf E_n:=
 	\begin{pmatrix}
 		 \mathsf X_n\\
 		 \mathsf Y_n\\
 	 \mathsf	Z_n
 	\end{pmatrix}
 	-\widetilde V_n.
 	\]
 	The dual-sum vector and the bilinear vector satisfy the same
 	recurrence.  By \eqref{ww:Drec} and  \eqref{ww:crossrec}, 
 	\[
 	 \mathsf E_n= \mathsf D_n  \mathsf E_{n+1}.
 	\]
 	
 	Moreover, from \eqref{ww:Nbound} and \eqref{ww:crossbound}, we know 
 	the two vectors have the same boundary behavior. 
 	Hence
 	\[
 	 \mathsf E_n=O(q^{3n+1}).
 	\]
 	
 	For any $M>n$, iteration of the common recurrence gives
 	\[
 	\mathsf E_n
 	=
 	\mathsf D_n \mathsf D_{n+1}\cdots \mathsf D_{M-1} \mathsf E_M.
 	\]
 	All entries of $\mathsf D_j$ belong to $\mathbb C[[q]]$, so multiplication
 	by these matrices cannot lower the $q$-adic order.  Since
 	\[
 	\mathsf E_M\in q^{3M+1}\mathbb C[[q]]^3,
 	\]
 	we obtain
 	\[
 	\mathsf E_n\in q^{3M+1}\mathbb C[[q]]^3
 	\]
 	for every $M>n$.  Letting $M$ be arbitrarily large shows that
 	every coefficient of $\mathsf E_n$ vanishes.  Thus
 	\[
 	\begin{pmatrix}
 		\mathsf X_n\\
 		\mathsf Y_n\\
 		\mathsf	Z_n
 	\end{pmatrix}
 	=\widetilde V_n.
 	\]
 	 These are precisely the three assertions of
 	Theorem~\ref{ww:bilinear}.
 	
 	Finally, setting $n=0$ and using
 	\[
 	\alpha_0=\widehat{\alpha}_0=S(0,0),\qquad
 	\beta_0=\widehat{\beta}_0=S(1,3),\qquad
 	\gamma_0=\widehat{\gamma}_0=S(2,3),
 	\]
 	gives the three stated special cases.
 \end{proof}

\subsection{Recovery of Conjecture 3.6 of Wang and Wang}
\label{sec:wang-wang-products}

\begin{proof}  
	 The first three dual identities \eqref{eq:ww1}--\eqref{eq:ww3} settle
	 Conjecture~3.6 of Wang and Wang.  Their original product forms
	 will be recovered at the end of this section.
 The first three cases of Theorem~\ref{thm:main}, proved in
Sections~\ref{sec:first-three}--\ref{sec:first-proof}, give
$S(0,0)=P_1$, $S(1,3)=P_2$, and $S(2,3)=P_3$.
Substitution into \eqref{ww:at0X}--\eqref{ww:at0Z} proves
\eqref{eq:ww1}--\eqref{eq:ww3}. Expanding $P_i$ in the first two identities recovers the first two
product forms stated in Conjecture~3.6 of Wang and Wang~\cite{wangwang}.
To recover their third displayed product form, put
\[
 \vartheta_r:=(q^r,q^{9-r};q^9)_\infty\quad(1\le r\le4),
 \qquad \mathcal E_9=\vartheta_3^{-1}.
\]
We need the theta identity
\begin{equation}\label{ww:cubic}
 \vartheta_2^2\vartheta_4-\vartheta_1\vartheta_4^2=q\vartheta_1^2\vartheta_2.
\end{equation}
Here is an explicit standard derivation. Write
$\theta(z;p)=(z,p/z;p)_\infty$, with
$\theta(z_1,\ldots,z_k;p)=\prod_{j=1}^k\theta(z_j;p)$, and use
the Weierstrass addition formula
\cite[Eq.~(2.7)]{koornwinder} in the form
\begin{align}\label{W-identity}
 &\theta(xy,x/y,uv,u/v;p)-\theta(xv,x/v,uy,u/y;p)=\frac{u}{y}\theta(yv,y/v,xu,x/u;p).
\end{align}
Set $p=q^9$, $x=q$, $y=q^3$, $u=q^2$, and $v=q^4$.
Using  
\begin{align}\label{t-1}
\theta(z^{-1};p)=-z^{-1}\theta(z;p),
\qquad
\theta(p/z;p)=\theta(z;p),
\end{align}
this becomes
\[
 q^{-4}(\vartheta_2^2\vartheta_4-\vartheta_1\vartheta_4^2)(q^3,q^6;q^9)_\infty
 =q^{-3}\vartheta_1^2\vartheta_2(q^3,q^6;q^9)_\infty.
\]
Cancellation proves \eqref{ww:cubic}. Since $P_1=\mathcal E_9/\vartheta_1$, $P_2=\mathcal E_9/\vartheta_2$, and
$P_3=\mathcal E_9/\vartheta_4$, it follows that
\begin{align*}
 P_1P_2-qP_3^2
 &=\mathcal E_9^2\left(\frac1{\vartheta_1\vartheta_2}-\frac{q}{\vartheta_4^2}\right)\\
 &=\mathcal E_9^2\left(\frac{\vartheta_2}{\vartheta_1^2\vartheta_4}-\frac{2q}{\vartheta_4^2}\right)\\
 &=\frac{(q^2,q^7;q^9)_\infty}
 {(q,q^3,q^6,q^8;q^9)_\infty^2(q^4,q^5;q^9)_\infty}
 -\frac{2q}{(q^3,q^4,q^5,q^6;q^9)_\infty^2}.
\end{align*}
This is exactly the third product form in Conjecture~3.6
of Wang and Wang~\cite{wangwang}. The formal series identities
also hold analytically for $|q|<1$, by absolute convergence.
\end{proof}

\section{Proof of Li--Wang's dual identity}\label{sec:li-wang}

\subsection{The identity and a parameterized family}

We now prove \eqref{eq:li-wang-conjecture}, the dual of the fourth
Kanade--Russell identity. Unlike the bilinear proof in
Section~\ref{sec:wang-wang}, this proof transfers the finite
certificate already used for \eqref{eq:kr4}.     Recall
\[
 \mathcal E_3=\frac{(q^3;q^3)_\infty}{(q;q)_\infty}
 =\frac1{(q,q^2;q^3)_\infty},
 \qquad \omega^2+\omega+1=0.
\]
Either primitive third root of unity may be used.

The remaining identity of Theorem~\ref{thm:dual-main} can be written as
\begin{equation}\label{lw:eq:target}
 \sum_{r,s\ge0}
 \frac{q^{r^2-3rs+3s^2+s}}{(q;q)_r(q^3;q^3)_s}
 =\frac{(q^6;q^9)_\infty}
 {(q;q^3)_\infty(q^2;q^3)_\infty^2}.
\end{equation}
Equivalently, the denominator on the right is
\((q,q^2,q^2,q^4,q^5,q^5;q^6)_\infty\), so
\eqref{lw:eq:target} is exactly \eqref{eq:li-wang-conjecture}.

For $n\geq0$, define the one-parameter specialization
\begin{equation}\label{lw:eq:Nn}
	\mathfrak N_n(q):=\mathfrak N(n,1),
\end{equation}
so that the linear term in the exponent is $nr+s$.

For \( \iota \in\mathbb C\), put
\begin{equation}\label{lw:eq:theta}
 \Theta_\iota(z):=\prod_{k\ge0}(1-zq^{3k}+\iota q^{6k}).
\end{equation}
Equivalently, if \(z=\alpha+\beta\) and
\(\iota=\alpha\beta\), then
\[
\Theta_\iota(z)=(\alpha,\beta;q^3)_\infty.
\]
For fixed \(0<|q|<1\), this is entire in \(z\) and \(\iota \).  Set
\begin{equation}\label{lw:eq:Wn}
 \begin{aligned}
 c_n^\vee&:=q^{4-3n},\\
 \xi_{n,j}&:=\omega q^{2-2n+j}+\omega^2q^{2-n-j}
                  \qquad(0\le j\le n),\\
 \mathfrak W_n(q)
 &:=\mathcal E_3(-1)^nq^{-3\binom n2}(q;q)_n
   \Theta_{c_n^\vee}[\xi_{n,0},\ldots,\xi_{n,n}].
 \end{aligned}
\end{equation}
The brackets denote the ordinary divided difference defined in
\eqref{eq:dd-definition}.  The nodes are distinct.  Indeed, if
\(\alpha_j=\omega q^{2-2n+j}\), then
\[
 \xi_{n,j}-\xi_{n,k}
 =(\alpha_j-\alpha_k)
 \left(1-\frac{c_n^\vee}{\alpha_j\alpha_k}\right),
 \qquad
 \frac{c_n^\vee}{\alpha_j\alpha_k}
 =\omega q^{n-j-k}.
\]
For \(j\ne k\), neither factor vanishes when \(0<|q|<1\).

\begin{theorem}\label{lw:thm:family}
For every \(n\ge0\) and \(0<|q|<1\),
\begin{equation}\label{lw:eq:family}
                         \mathfrak N_n(q)=\mathfrak W_n(q).
\end{equation}
Both sides extend to formal power series at \(q=0\), with constant
term one.
\end{theorem}

Since the quadratic form $Q(r,s)$ is positive definite, 
the sums in \eqref{lw:eq:Nn} converge absolutely and locally uniformly
on \(|q|<1\).  They are also well-defined formal power series.  

%\subsection{The recurrence for the dual sums}

We use the common notation $\mathfrak N(a,b)$ from
\eqref{eq:dual-sum} and the contiguous relations
\eqref{eq:dual-contiguous-a} and \eqref{eq:dual-contiguous-b} established in Section~\ref{sec:wang-wang}.

\begin{lemma}\label{lw:lem:Nrec}
For \(n\ge0\), the sequence \(\mathfrak N_n(q)\) satisfies
\begin{equation}\label{lw:eq:mainrec}
 q\mathfrak N_n(q)-(1+q+q^{2n+3})\mathfrak N_{n+1}(q)
 +(1-q^{n+2})\mathfrak N_{n+2}(q)
 +q^{n+2}\mathfrak N_{n+3}(q)=0.
\end{equation}
\end{lemma}

\begin{proof}
Fix \(b\) and put \(F_a=\mathfrak N(a,b)\).  Applying the first
relation   \eqref{eq:dual-contiguous-a} twice gives
\[
 \mathfrak N(a,b-3)=q^{1-a}(F_{a-2}-F_{a-1})
\]
and
\[
 \mathfrak N(a,b-6)
 =q^{3-2a}\{qF_{a-4}-(1+q)F_{a-3}+F_{a-2}\}.
\]
The second relation   \eqref{eq:dual-contiguous-b}, applied at
\((a,b-6)\), says that the difference of these expressions is
\(q^{b-3}F_{a-3}\).  Multiplication by \(q^{2a-3}\), followed by
\(a=n+4\), yields
\[
 qF_n-(1+q+q^{2n+b+2})F_{n+1}
 +(1-q^{n+2})F_{n+2}+q^{n+2}F_{n+3}=0.
\]
Taking \(b=1\) proves \eqref{lw:eq:mainrec}.
\end{proof}

\subsection{Transfer of the fourth-identity certificate}

We use the same finite-product notation as in Section~\ref{sec:cyclotomic-certificates},
with the algebraic parameter \(q\) replaced by an independent
indeterminate \(u\). More precisely, let \(u,x,t\) be independent
indeterminates and define
\begin{align*}
h_e^{(d)}(t;u)
&:=t^2+u^ex^dt+u^{2e}x^{2d},
\\ 
H_{r,s}^{(d)}(t;u)
&:=\prod_{e=r}^{s}h_e^{(d)}(t;u) 
\end{align*}
with empty products equal to one.
Only the resulting finite algebraic identity will later be
specialized at \(u=q^{-1}\); all infinite products retain the
convergent base \(q^3\). 
Set 
\[
 K_j(x;u)=\prod_{\ell=1}^j(1-u^\ell x),\qquad
 V_j(t;u)=\prod_{\ell=0}^{j-1}(t-u^{2+3\ell}x^3)
\]
and
\begin{equation}\label{lw:eq:lambdas}
 \lambda_0(x;u)=1,\ 
 \lambda_1(x;u)=-1-u^2(1+u)x^2,\ 
 \lambda_2(x;u)=u^7x^4-u^5x^3,\ 
 \lambda_3(x;u)=u^{11}x^5.
\end{equation}
 The following polynomials are precisely the \(\tau=1\) objects
of Section~\ref{sec:cyclotomic-certificates} with the algebraic parameter \(q\) replaced by \(u\);
 we write \(D(t)\) for the specialization of \(C_\tau(t)\) at
 \(\tau=1\), in order to avoid conflict with the sequence \(C_n\).  Define 
\begin{equation}\label{lw:eq:certificate-data}
 \begin{aligned}
 \Delta(t)&:=t^2H_{2,7}^{(2)}(t;u),\\
 B(t)&:=(t-u^2x^3)H_{1,3}^{(1)}(t;u),\\
 D(t)&:=t(t-1)H_{5,7}^{(2)}(t;u),\\
 L(t)&:=\sum_{j=0}^3\lambda_j(x;u)K_j(x;u)
 (t^2-u^{2+3j}x^3)V_j(t;u)  H_{1,j}^{(1)}(t;u)
 H_{2j+2,7}^{(2)}(t;u).
 \end{aligned}
\end{equation}
Let
\[
 \mathcal P^\vee(t;u,x)=\left.\mathcal P_1(t)\right|_{q=u},
\]
where the seven coefficients of \(\mathcal P_1(t)\) are listed in
Appendix~\ref{app:coefficients}.  Since \(q,x,t\) were independent
in Section~\ref{sec:cyclotomic-certificates}, the specialization
\(\tau=1\) of \eqref{b:eq:22}, with its algebraic base renamed
from \(q\) to \(u\), gives the polynomial identity
\begin{equation}\label{lw:eq:certificate}
 L(t)=B(t)\mathcal P^\vee(t;u,x)
 +\frac{D(t)\mathcal P^\vee(u^3t;u,x)}{u^{21}x^3}.
\end{equation}
Thus \eqref{lw:eq:certificate} is a finite algebraic identity in the independent
indeterminates $u,x,t$.  In particular, the later specialization
$u=q^{-1}$ is made only in this finite identity; no substitution
$q\mapsto q^{-1}$ is performed in any infinite product, all of which
retain the convergent base $q^3$.

Define 
\begin{equation}\label{lw:eq:Fkernel}
 \mathcal F_n(t;u):=
 \frac{(u^2/t;u^3)_n}
 {t^{n+1}(\omega u^{n+1}/t,\omega^2u^{n+1}/t;u)_{n+1}},
 \qquad
 J_n(t;u):=1-\frac{u^{3n+2}}{t^2},
\end{equation}
and
\[
 \eta(t;u):=u^3\frac{t(t-1)}{t-u^{-1}},\qquad
 \rho(t;u,x):=\frac{t(t-1)}{x^3(t-u^{-1}x^3)}
 \frac{H_{-1,1}^{(2)}(t;u)}{H_{-2,0}^{(1)}(t;u)}.
\]
At \(x=u^n\), cancellation of finite product strings gives
\begin{equation}\label{lw:eq:kernel-shifts}
 \frac{\mathcal F_{n+1}(t;u)}{\mathcal F_n(t;u)}
 =\frac{(t-u^2x^3)h_1^{(1)}(t;u)}
 {h_2^{(2)}(t;u)h_3^{(2)}(t;u)},
 \qquad
 \eta(t;u)\frac{\mathcal F_n(u^3t;u)}{\mathcal F_n(t;u)}
 =\rho(t;u,x).
\end{equation}
The first quotient gains \(1-u^{3n+2}/t\), removes two initial
linear factors, and gains four terminal factors; the pairing
\((t-\omega v)(t-\omega^2v)=t^2+vt+v^2\) gives the displayed
quadratics.  In the second quotient, shifting by three leaves three
factors at each endpoint.  Hence these are rational identities for
every \(n\ge0\), including the cases of overlapping short products.

Put 
\[R(t):=B(t)\mathcal P^\vee(t;u,x)/\Delta(t).\]
  The scaling identity
\[
 h_e^{(d)}(u^3t;u)=u^6h_{e-3}^{(d)}(t;u)
\]
gives
\[
 \rho(t;u,x)R(u^3t)
 =\frac{D(t)\mathcal P^\vee(u^3t;u,x)}{u^{21}x^3\Delta(t)}.
\]
We now derive the rational telescoping identity from
\eqref{lw:eq:certificate} and \eqref{lw:eq:kernel-shifts}.  Recall that $x=u^n$.
Iterating the first identity in \eqref{lw:eq:kernel-shifts}, with
$x$ replaced successively by $x,ux,\ldots,u^{j-1}x$, gives
\[
\frac{\mathcal{F}_{n+j}(t;u)}
{\mathcal{F}_n(t;u)}
=
\frac{
	V_j(t;u)H^{(1)}_{1,j}(t;u)}
{H^{(2)}_{2,\,2j+1}(t;u)},
\qquad 0\leq j\leq3.
\]
Indeed, at the $\ell$-th step the numerator contributes
$
t-u^{2+3\ell}x^3$ and 
$
h^{(1)}_{\ell+1}(t;u),
$
whereas the denominator contributes
$
h^{(2)}_{2+2\ell}(t;u)
h^{(2)}_{3+2\ell}(t;u).
$
Their products over $0\leq\ell\leq j-1$ are precisely
$V_j(t;u)$, $H^{(1)}_{1,j}(t;u)$, and
$H^{(2)}_{2,\,2j+1}(t;u)$, respectively.

Moreover,
\[
J_{n+j}(t;u)
=
1-\frac{u^{3(n+j)+2}}{t^2}
=
\frac{t^2-u^{2+3j}x^3}{t^2}.
\]
Since
\[
\Delta(t)=t^2H^{(2)}_{2,7}(t;u),
\]
we obtain
\begin{align*}
	\Delta(t)J_{n+j}(t;u)
	\frac{\mathcal{F}_{n+j}(t;u)}
	{\mathcal{F}_n(t;u)}
	&=
	t^2H^{(2)}_{2,7}(t;u)
	\frac{t^2-u^{2+3j}x^3}{t^2}
	\frac{
		V_j(t;u)H^{(1)}_{1,j}(t;u)}
	{H^{(2)}_{2,\,2j+1}(t;u)}
	\\
	&=
	\bigl(t^2-u^{2+3j}x^3\bigr)
	V_j(t;u)H^{(1)}_{1,j}(t;u)
	H^{(2)}_{2j+2,7}(t;u).
\end{align*}
Multiplying by $\lambda_j(x;u)K_j(x;u)$ and summing over
$j=0,1,2,3$, the definition of $L(t)$ in \eqref{lw:eq:certificate-data}
therefore gives
\begin{equation}\label{eq:L-from-F}
	\Delta(t)
	\sum_{j=0}^3
	\lambda_j(x;u)K_j(x;u)J_{n+j}(t;u)
	\frac{\mathcal{F}_{n+j}(t;u)}
	{\mathcal{F}_n(t;u)}
	=
	L(t).
\end{equation}

By \eqref{lw:eq:certificate}, equation \eqref{eq:L-from-F} becomes
\[
\sum_{j=0}^3
\lambda_j(x;u)K_j(x;u)J_{n+j}(t;u)
\frac{\mathcal{F}_{n+j}(t;u)}
{\mathcal{F}_n(t;u)}
=
R(t)+\rho(t;u,x)R(u^3t).
\]

Multiplying by $\mathcal{F}_n(t;u)$ gives
\begin{align*}
	\sum_{j=0}^3
	\lambda_j(x;u)K_j(x;u)J_{n+j}(t;u)
	\mathcal{F}_{n+j}(t;u)
	&=
	\mathcal{F}_n(t;u)R(t)+
	\mathcal{F}_n(t;u)\rho(t;u,x)R(u^3t).
\end{align*}
Finally, the second identity in \eqref{lw:eq:kernel-shifts} gives
\[
\mathcal{F}_n(t;u)\rho(t;u,x)
=
\eta(t;u)\mathcal{F}_n(u^3t;u).
\]
Consequently,
\begin{equation}\label{lw:eq:telescoping}
 \begin{aligned}
 &\sum_{j=0}^3\lambda_j(x;u)K_j(x;u)J_{n+j}(t;u)
          \mathcal F_{n+j}(t;u)\\
 &\qquad=\mathcal F_n(t;u)R(t)
 +\eta(t;u)\mathcal F_n(u^3t;u)R(u^3t),
 \qquad x=u^n.
 \end{aligned}
\end{equation}

\subsection{A convergent residue weight and the product recurrence}

Initially let \(0<q<1\) and set \(u=q^{-1}\) only in the finite
identity \eqref{lw:eq:telescoping}. 
Define the convergent residue weight 
\begin{equation}\label{lw:eq:weight}
 \Psi_{\rm LW}(t):=(q^3t,q/t;q^3)_\infty.
\end{equation}
It is holomorphic on \(\mathbb C\setminus \{0\}\), and
\begin{equation}\label{lw:eq:weight-shift}
 \frac{\Psi_{\rm LW}(u^3t)}{\Psi_{\rm LW}(t)}
 =\frac{1-t}{1-q/t},\qquad
 \eta(t;u)=-u^3\frac{\Psi_{\rm LW}(u^3t)}
                         {\Psi_{\rm LW}(t)}.
\end{equation}
Let
\[
 \Lambda_{\rm LW}:
 =\{\omega^\epsilon q^m:\epsilon\in\{1,2\},\ m\in\mathbb Z\}.
\]
For a rational function \(T(t)\), define
\[
 \mathcal L_{\rm LW}(T):=\frac12
 \sum_{t_0\in\Lambda_{\rm LW}}
 \operatorname{Res}_{t=t_0}\bigl(\Psi_{\rm LW}(t)T(t)\bigr).
\]
Only finitely many rational poles contribute.  Since
\(\Lambda_{\rm LW}\) is invariant under multiplication by
\(u^3=q^{-3}\), the local change of variable and finite reindexing
used in Lemma~\ref{lem:finite-residue} give
\begin{equation}\label{lw:eq:residue-cancel}
 \mathcal L_{\rm LW}\bigl(T(t)+\eta(t;u)T(u^3t)\bigr)=0.
\end{equation}

At \(u=q^{-1}\), direct finite-product cancellation yields
\begin{equation}\label{lw:eq:weight-kernel}
 \begin{aligned}
 \Psi_{\rm LW}(t)(u^2/t;u^3)_n
 &=(q^3t,q^{1-3n}/t;q^3)_\infty\\
 &=\Theta_{c_n^\vee}\!
 \left(q^3\left(t+\frac{u^{3n+2}}t\right)\right).
 \end{aligned}
\end{equation}
Set \(\zeta=t+u^{3n+2}/t\).  The two preimages of
\[
 \zeta_j=\omega u^{n+1+j}+\omega^2u^{2n+1-j},
 \qquad 0\le j\le n,
\]
are the corresponding points in
\(t=\omega u^{n+1+j}\) and \(t=\omega^2u^{n+1+j}\), after one list
is reversed.  Moreover,
\[
 \prod_{j=0}^n(\zeta-\zeta_j)
 =t^{n+1}(\omega u^{n+1}/t,\omega^2u^{n+1}/t;u)_{n+1},
 \qquad
 \frac{d\zeta}{dt}=J_n(t;u).
\]
The local residue change therefore produces the corresponding
Lagrange summand for \(\Theta_{c_n^\vee}(q^3\zeta)\).
Each summand occurs at two preimages, which accounts for the factor
\(1/2\).  Reversing the list shows that \(q^3\zeta_j\) are precisely
the nodes \(\xi_{n,j}\) in \eqref{lw:eq:Wn}.  Scaling a divided
difference of order \(n\) consequently gives
\begin{equation}\label{lw:eq:Ldd}
 \mathcal L_{\rm LW}\bigl(J_n\mathcal F_n\bigr)
 =q^{3n}\Theta_{c_n^\vee}
       [\xi_{n,0},\ldots,\xi_{n,n}].
\end{equation}

\begin{lemma}\label{lw:lem:Wrec}
The sequence \(\mathfrak W_n\) satisfies \eqref{lw:eq:mainrec}, with
\(\mathfrak W\) in place of \(\mathfrak N\).
\end{lemma}

\begin{proof}
Apply \eqref{lw:eq:residue-cancel} to
\eqref{lw:eq:telescoping}.  Since
\[
 K_j(u^n;u)=\frac{(u;u)_{n+j}}{(u;u)_n},
\]
the sequence
\[
 \mathfrak Z_n:=\mathcal E_3^{-1}(u;u)_n
       \mathcal L_{\rm LW}(J_n\mathcal F_n)
\]
satisfies
\(\sum_{j=0}^3\lambda_j(u^n;u)\mathfrak Z_{n+j}=0\).
Using \eqref{lw:eq:Ldd} and
\[
 (q^{-1};q^{-1})_n
 =(-1)^nq^{-n(n+1)/2}(q;q)_n
\]
gives
\begin{equation}\label{lw:eq:normalization}
 \mathfrak Z_n=\mathcal E_3^{-2}q^{n(n+1)}\mathfrak W_n.
\end{equation}
After substituting \eqref{lw:eq:normalization}, dividing out the
common factor, and then multiplying by \(q\), the four recurrence
coefficients are
\[
 q,\qquad -(1+q+q^{2n+3}),\qquad
 1-q^{n+2},\qquad q^{n+2}.
\]
This proves the lemma.
\end{proof}

\subsection{A coefficient estimate for the direct product}

Let 
\[
t_m(\iota)=[z^m]\Theta_\iota (z),\]
 and define
\begin{equation}\label{lw:eq:dm}
 \begin{aligned}
 d_m(\gamma):
 &=(-1)^mq^{-3\binom m2}(q^3;q^3)_m
       t_m(\gamma q^{-3m}),\\
 \mathcal R(\gamma;q):
 &=\sum_{j\ge0}\frac{\gamma^jq^{3j(j-1)}}{(q^3;q^3)_j}.
 \end{aligned}
\end{equation}

\begin{lemma}\label{lw:lem:coefficient}
There are series \(r_{m,j}\in\mathbb Z[[q^3]]\) such that
\begin{equation}\label{lw:eq:rmj}
 d_m(\gamma)=\sum_{j\ge0}\gamma^jq^{3j(j-1)}r_{m,j},
 \quad
 r_{m,j}-\frac1{(q^3;q^3)_j}
 \in q^{3(m+1)}\mathbb Z[[q^3]]
 \quad(j\ge1),
\end{equation}
and \(r_{m,0}=1\).  Hence, for every positive integer \(s\),
\begin{equation}\label{lw:eq:dm-estimate}
 d_m(q^s)=\mathcal R(q^s;q)+O(q^{3(m+1)+s}).
\end{equation}
\end{lemma}

\begin{proof}
In \eqref{lw:eq:theta}, choose \(m\) factors contributing
\(-zq^{3k}\) and \(j\) disjoint factors contributing \(\iota q^{6k}\).
After the normalization in \eqref{lw:eq:dm}, the minimum remaining
exponent of \(q^3\) is \(j(j-1)\).  It is attained by taking the
first \(m+j\) sites and placing the \(j\) sites of weight two first.
This proves the expansion in \eqref{lw:eq:rmj}, with
\[
 r_{m,0}=1,\qquad r_{0,j}=\frac1{(q^6;q^6)_j}.
\]
The product relation
\[
 \Theta_{\iota}(z)=(1-z+\iota)\Theta_{\iota q^6}(q^3z)
\]
gives, for \(m\ge1\),
\[
 d_m(\gamma)
 =(q^{3m}+\gamma)d_m(\gamma q^6)
 +(1-q^{3m})d_{m-1}(\gamma q^3).
\]
Comparison of the coefficients of \(\gamma^j\) gives
\begin{equation}\label{lw:eq:r-recurrence}
 (1-q^{3m+6j})r_{m,j}
 =r_{m,j-1}+q^{3j}(1-q^{3m})r_{m-1,j}.
\end{equation}
Let
\(\delta_{m,j}=(q^3;q^3)_j^{-1}-r_{m,j}\).  Rearranging
\eqref{lw:eq:r-recurrence} yields
\begin{equation}\label{lw:eq:delta-recurrence}
 \begin{aligned}
 (1-q^{3m+6j})\delta_{m,j}
 ={}&\delta_{m,j-1}
 +q^{3j}(1-q^{3m})\delta_{m-1,j}+\frac{q^{3m+3j}}{(q^3;q^3)_{j-1}}.
 \end{aligned}
\end{equation}
For \(m=0\), one has
\(\delta_{0,j}\in q^3\mathbb Z[[q^3]]\), and
\(\delta_{m,0}=0\).  Induction on \(m\), and then along each row
in \(j\), proves
\(\delta_{m,j}\in q^{3(m+1)}\mathbb Z[[q^3]]\).
Indeed, the three terms on the right of
\eqref{lw:eq:delta-recurrence} have orders at least
\(3(m+1)\), \(3(m+j)\), and \(3(m+j)\), while the factor on the
left is a unit.  Substitution of \(\gamma=q^s\) shows that the
smallest error order occurs at \(j=1\), which proves
\eqref{lw:eq:dm-estimate}.
\end{proof}

\subsection{Matching the boundary}

Let \(h_k(\xi_{n,0},\ldots,\xi_{n,n})\) be the complete homogeneous
symmetric polynomial of degree \(k\), with \(h_0=1\).  The monomial
divided-difference formula gives
\[
 z^m[\xi_{n,0},\ldots,\xi_{n,n}]
 =\begin{cases}
 0,&m<n,\\
 h_{m-n}(\xi_{n,0},\ldots,\xi_{n,n}),&m\ge n.
 \end{cases}
\]
Expanding \eqref{lw:eq:Wn} by \eqref{lw:eq:dm} therefore gives the
exact series
\begin{equation}\label{lw:eq:W-expansion}
 \mathfrak W_n
 =\mathcal E_3(q;q)_n\sum_{k\ge0}
 \frac{(-1)^kq^{3nk+3\binom k2}
 d_{n+k}(q^{4+3k})}
 {(q^3;q^3)_{n+k}}\,
 h_k(\xi_{n,0},\ldots,\xi_{n,n}).
\end{equation}
The smallest exponent occurring in a node is \(2-2n\).
Consequently the \(k\)-th summand of
\eqref{lw:eq:W-expansion} has order at least
\begin{equation}\label{lw:eq:k-order}
 3nk+3\binom k2+(2-2n)k
 =nk+\frac{3k^2+k}{2}.
\end{equation}
Thus the expansion is coefficientwise convergent and proves
\(\mathfrak W_n\in\mathbb C[[q]]\); in particular, the negative
Laurent powers in its finite Lagrange expression cancel.

Define
\begin{equation}\label{lw:eq:boundary-series}
 \mathcal R_0(q):=\mathcal R(q^4;q)
 =\sum_{s\ge0}\frac{q^{3s^2+s}}{(q^3;q^3)_s},
 \qquad
 \mathcal R_1(q):=\mathcal R(q;q)
 =\sum_{s\ge0}\frac{q^{3s^2-2s}}{(q^3;q^3)_s}.
\end{equation}

\begin{lemma}\label{lw:lem:boundary}
For every \(n\ge0\), the two sequences have the same boundary
expansion:
\begin{equation}\label{lw:eq:boundary}
\begin{aligned}
 \mathfrak N_n
 &=\mathcal R_0(q)+\frac{q^{n+1}}{1-q}\mathcal R_1(q)
   +O(q^{2n+2}),\\
 \mathfrak W_n
 &=\mathcal R_0(q)+\frac{q^{n+1}}{1-q}\mathcal R_1(q)
   +O(q^{2n+2}).
\end{aligned}
\end{equation}
\end{lemma}

\begin{proof}
For \(\mathfrak N_n\), the terms \(r=0\) and \(r=1\) in
\eqref{lw:eq:Nn} give the two displayed terms exactly.  If \(r\ge2\),
then \(Q(r,s)+s\ge2\).  Indeed, for \(s=0\) the value is at least
four; for \(s=1\) it is \((r-1)(r-2)+2\); and for \(s\ge2\),
completion of the square gives
\[
 Q(r,s)+s\ge\frac34s^2+s\ge5.
\]
All remaining summands therefore have order at least \(2n+2\).

For \(\mathfrak W_n\), every term of
\eqref{lw:eq:W-expansion} with \(k\ge2\) has order at least
\(2n+7\), by \eqref{lw:eq:k-order}.  Put
\[
 \mathcal U_n:
 =\frac{\mathcal E_3(q;q)_n}{(q^3;q^3)_n}
 =\frac{(q^{3n+3};q^3)_\infty}{(q^{n+1};q)_\infty}.
\]
The sum of the nodes is
\[
 h_1(\xi_{n,0},\ldots,\xi_{n,n})
 =-q^{2-2n}\frac{1-q^{n+1}}{1-q}.
\]
Hence the first two terms of \eqref{lw:eq:W-expansion} are
\begin{equation}\label{lw:eq:W-two}
 \begin{aligned}
 \mathfrak W_n={}&\mathcal U_nd_n(q^4)+\mathcal U_n
 \frac{q^{n+2}(1-q^{n+1})}
 {(1-q)(1-q^{3n+3})}\,d_{n+1}(q^7)
 +O(q^{2n+7}).
 \end{aligned}
\end{equation}
The reciprocal product expansion gives
\[
 \mathcal U_n
 =1+\frac{q^{n+1}}{1-q}+O(q^{2n+2}).
\]
Using \eqref{lw:eq:dm-estimate} in \eqref{lw:eq:W-two} now gives
\[
 \mathfrak W_n
 =\mathcal R_0(q)+\frac{q^{n+1}}{1-q}
 \{\mathcal R(q^4;q)+q\mathcal R(q^7;q)\}
 +O(q^{2n+2}).
\]
Finally, direct coefficient comparison gives
\begin{equation}\label{lw:eq:R-contiguous}
 \mathcal R(\gamma;q)
 =\mathcal R(\gamma q^3;q)
 +\gamma\mathcal R(\gamma q^6;q).
\end{equation}
Taking \(\gamma=q\) identifies the expression in braces with
\(\mathcal R_1(q)\), proving \eqref{lw:eq:boundary}.
\end{proof}

\subsection{Boundary uniqueness and evaluation at order zero}

\begin{proof}[Proof of Theorem~\ref{lw:thm:family}]
Let \(e_n=\mathfrak W_n-\mathfrak N_n\).  Lemmas
\ref{lw:lem:Nrec}, \ref{lw:lem:Wrec}, and
\ref{lw:lem:boundary} give
\[
 qe_n-(1+q+q^{2n+3})e_{n+1}
 +(1-q^{n+2})e_{n+2}+q^{n+2}e_{n+3}=0,
 \qquad e_n=O(q^{2n+2}).
\]
The product recurrence, initially obtained for \(0<q<1\), is also a
formal identity by the removable-singularity argument following
\eqref{lw:eq:k-order}. Since the two sides agree for real \(0<q<1\) and extend
analytically to \(q=0\), the identity theorem gives the
corresponding formal identity. Its companion form is
\[
 \begin{pmatrix}e_n\\e_{n+1}\\e_{n+2}\end{pmatrix}
 =  \overline{M}_n
 \begin{pmatrix}e_{n+1}\\e_{n+2}\\e_{n+3}\end{pmatrix},
 \quad
 \overline{M}_n=
 \begin{pmatrix}
 q^{-1}(1+q+q^{2n+3})&-q^{-1}(1-q^{n+2})&-q^{n+1}\\
 1&0&0\\
 0&1&0
 \end{pmatrix}.
\]
Every entry of \( \overline{M}_n\) belongs to \(q^{-1}\mathbb C[[q]]\).
Fix \(n\) and iterate to any \(N>n\).  The product of \(N-n\)
matrices lowers the least exponent by at most \(N-n\), whereas the
vector at level \(N\) has order at least \(2N+2\).  Therefore
\[
 \operatorname{ord}_q(e_n)
 \ge2N+2-(N-n)=N+n+2.
\]
Letting \(N\) tend to infinity proves \(e_n=0\) coefficientwise.
Analyticity on \(|q|<1\) then proves \eqref{lw:eq:family} throughout
the disk.
\end{proof}

\begin{proof}[Proof of \eqref{eq:li-wang-conjecture}]
At \(n=0\), one has
\(\xi_{0,0}=-q^2\) and \(c_0^\vee=q^4\).  Theorem
\ref{lw:thm:family} gives
\[
 \mathfrak N_0
 =\mathcal E_3\Theta_{q^4}(-q^2)
 =\mathcal E_3(\omega q^2,\omega^2q^2;q^3)_\infty.
\]
Applying
\((1-v)(1-\omega v)(1-\omega^2v)=1-v^3\) factorwise yields
\[
 (\omega q^2,\omega^2q^2;q^3)_\infty
 =\frac{(q^6;q^9)_\infty}{(q^2;q^3)_\infty}.
\]
Together with
\(\mathcal E_3=1/(q,q^2;q^3)_\infty\), this proves
\eqref{lw:eq:target}.  Finally,
\[
 (q;q^3)_\infty(q^2;q^3)_\infty^2
 =(q,q^2,q^2,q^4,q^5,q^5;q^6)_\infty,
\]
so the result has exactly the form conjectured by Li and Wang.
\end{proof}

\section{A dual companion of the fifth Kanade--Russell identity}\label{sec:fifth-dual}

We now prove Theorem~\ref{thm:fifth-dual-companion}. The argument is
parallel to the proof of the Li--Wang identity in
Section~\ref{sec:li-wang}, but it uses the $\tau=2$ certificate from
Section~\ref{sec:cyclotomic-certificates}. Thus both cyclotomic
certificates used for the fourth and fifth Kanade--Russell identities
admit direct-product transfers.

\subsection{A parameterized dual family and its recurrence}

For $n\geq0$, define
\begin{equation}\label{d5:eq:Mfamily}
 \mathcal M_n(q):=\mathfrak N(n,2)+q^2\mathfrak N(n-2,5).
\end{equation}
Although $\mathfrak N(-2,5)$ itself has lowest $q$-degree -1, the
combination in \eqref{d5:eq:Mfamily} is a formal power series for every
$n\geq0$. The case $n=0$ is the left-hand side of
Theorem~\ref{thm:fifth-dual-companion}.

Define the recurrence coefficients, as polynomials in an independent
variable $x$, by
\begin{equation}\label{d5:eq:coefficients}
\begin{aligned}
 \varrho_0(x)&:=q(1+q^3x+q^5x^2),\\
 \varrho_1(x)&:=-\bigl(1+q+q^2x+q^4x
 +(q^3+q^5+q^6)x^2+q^7x^3+q^8x^4\bigr),\\
 \varrho_2(x)&:=1-q^6x^3,\\
 \varrho_3(x)&:=q^2x(1+q^2x+q^3x^2).
\end{aligned}
\end{equation}

\begin{lemma}\label{d5:lem:Mrec}
For every $n\geq0$,
\begin{equation}\label{d5:eq:mainrec}
 \sum_{j=0}^3\varrho_j(q^n)\mathcal M_{n+j}(q)=0.
\end{equation}
\end{lemma}

\begin{proof}
Put
\[
 \mathcal B_n^{(5)}:=\mathfrak N(n-2,5).
\]
The fixed-$b$ recurrence derived in the proof of
Lemma~\ref{lw:lem:Nrec}, with $b=5$ and the index shifted by two,
gives
\begin{equation}\label{d5:eq:Brec}
 q\mathcal B_n^{(5)}-(1+q+q^{2n+3})\mathcal B_{n+1}^{(5)}
 +(1-q^n)\mathcal B_{n+2}^{(5)}+q^n\mathcal B_{n+3}^{(5)}=0.
\end{equation}
Moreover, \eqref{eq:dual-contiguous-a} with $(a,b)=(n-2,5)$ gives
\[
 \mathcal B_n^{(5)}-\mathcal B_{n+1}^{(5)}
 =q^{n-1}\mathfrak N(n,2),
\]
and hence
\begin{equation}\label{d5:eq:MfromB}
 \mathcal M_n
 =q^{1-n}\bigl((1+q^{n+1})\mathcal B_n^{(5)}
               -\mathcal B_{n+1}^{(5)}\bigr).
\end{equation}

For $n\geq0$, set
\[
 \mathsf T_n:=
 \begin{pmatrix}
 q^{-1}(1+q+q^{2n+3})&-q^{-1}(1-q^n)&-q^{n-1}\\
 1&0&0\\
 0&1&0
 \end{pmatrix},
 \qquad
 \mathbf p_n:=(1+q^{n+1},-1,0).
\]
Then \eqref{d5:eq:Brec} is equivalent to
\[
 \begin{pmatrix}
 \mathcal B_n^{(5)}\\
 \mathcal B_{n+1}^{(5)}\\
 \mathcal B_{n+2}^{(5)}
 \end{pmatrix}
 =\mathsf T_n
 \begin{pmatrix}
 \mathcal B_{n+1}^{(5)}\\
 \mathcal B_{n+2}^{(5)}\\
 \mathcal B_{n+3}^{(5)}
 \end{pmatrix}.
\]
Define the four row vectors
\[
\begin{aligned}
 \mathbf p_{n,0}&:=\mathbf p_n\mathsf T_n\mathsf T_{n+1}\mathsf T_{n+2},\\
 \mathbf p_{n,1}&:=\mathbf p_{n+1}\mathsf T_{n+1}\mathsf T_{n+2},\\
 \mathbf p_{n,2}&:=\mathbf p_{n+2}\mathsf T_{n+2},\\
 \mathbf p_{n,3}&:=\mathbf p_{n+3}.
\end{aligned}
\]
A direct multiplication gives the exact row identity
\begin{equation}\label{d5:eq:row-certificate}
 \varrho_0(q^n)\mathbf p_{n,0}
 +q^{-1}\varrho_1(q^n)\mathbf p_{n,1}
 +q^{-2}\varrho_2(q^n)\mathbf p_{n,2}
 +q^{-3}\varrho_3(q^n)\mathbf p_{n,3}=0.
\end{equation}
By \eqref{d5:eq:MfromB}, for $0\leq j\leq3$,
\[
 \mathcal M_{n+j}
 =q^{1-n}q^{-j}\mathbf p_{n,j}
 \begin{pmatrix}
 \mathcal B_{n+3}^{(5)}\\
 \mathcal B_{n+4}^{(5)}\\
 \mathcal B_{n+5}^{(5)}
 \end{pmatrix}.
\]
Multiplying \eqref{d5:eq:row-certificate} by the displayed column
vector proves \eqref{d5:eq:mainrec}.
\end{proof}

\subsection{Transfer of the $\tau=2$ certificate}

We retain the notation $h_e^{(d)}(t;u)$,
$H_{r,s}^{(d)}(t;u)$, and $K_j(x;u)$ from
Section~\ref{sec:li-wang}. For each of the remaining $\tau=2$ objects of
Section~\ref{sec:cyclotomic-certificates}, appending the argument $u$
means that its algebraic base $q$ is renamed as the independent
indeterminate $u$. Thus
\[
 V_{j,2}(t;u),\quad \Delta_2(t;u),\quad B_2(t;u),\quad
 C_2(t;u),\quad L_2(t;u),\quad \rho_2(t;u,x),
\]
and $\lambda_j^{(2)}(x;u)$ are exactly the corresponding objects of
Section~\ref{sec:cyclotomic-certificates} after $q\mapsto u$. Likewise,
\[
 \mathcal P_2(t;u,x):=\left.\mathcal P_2(t)\right|_{q=u}.
\]
With this notation, the $\tau=2$ specialization of
\eqref{b:eq:22} is
\begin{equation}\label{d5:eq:certificate}
 L_2(t;u)=B_2(t;u)\mathcal P_2(t;u,x)
 +\frac{C_2(t;u)\mathcal P_2(u^3t;u,x)}{u^{21}x^3}.
\end{equation}
Thus no new polynomial certificate is required.

For $n\geq0$, define
\begin{equation}\label{d5:eq:Fkernel}
 \mathcal F_n^{(2)}(t;u):=
 \frac{(u^4/t;u^3)_n}
 {t^{n+1}(\omega u^{n+2}/t,\omega^2u^{n+2}/t;u)_{n+1}},
 \qquad
 J_n^{(2)}(t;u):=1-\frac{u^{3n+4}}{t^2},
\end{equation}
and put
\[
 \eta_2(t;u):=u^3\frac{t(t-1)}{t-u}.
\]
At $x=u^n$, the same finite-product cancellation as in
\eqref{lw:eq:kernel-shifts} gives
\begin{equation}\label{d5:eq:kernel-shifts}
 \frac{\mathcal F_{n+1}^{(2)}(t;u)}{\mathcal F_n^{(2)}(t;u)}
 =\frac{(t-u^4x^3)h_2^{(1)}(t;u)}
 {h_3^{(2)}(t;u)h_4^{(2)}(t;u)},
 \qquad
 \eta_2(t;u)
 \frac{\mathcal F_n^{(2)}(u^3t;u)}{\mathcal F_n^{(2)}(t;u)}
 =\rho_2(t;u,x).
\end{equation}

Put
\[
 R_2(t;u,x):=
 \frac{B_2(t;u)\mathcal P_2(t;u,x)}{\Delta_2(t;u)}.
\]
The scaling identity
$h_e^{(d)}(u^3t;u)=u^6h_{e-3}^{(d)}(t;u)$ gives
\[
 \rho_2(t;u,x)R_2(u^3t;u,x)
 =\frac{C_2(t;u)\mathcal P_2(u^3t;u,x)}
 {u^{21}x^3\Delta_2(t;u)}.
\]
Iteration of the first identity in \eqref{d5:eq:kernel-shifts} gives
\[
 \frac{\mathcal F_{n+j}^{(2)}(t;u)}{\mathcal F_n^{(2)}(t;u)}
 =\frac{V_{j,2}(t;u)H_{2,j+1}^{(1)}(t;u)}
 {H_{3,2j+2}^{(2)}(t;u)},
 \qquad 0\leq j\leq3.
\]
Together with
\[
 J_{n+j}^{(2)}(t;u)=\frac{t^2-u^{4+3j}x^3}{t^2},
\]
this gives, exactly as in \eqref{eq:L-from-F},
\[
 \Delta_2(t;u)
 \sum_{j=0}^3\lambda_j^{(2)}(x;u)K_j(x;u)
 J_{n+j}^{(2)}(t;u)
 \frac{\mathcal F_{n+j}^{(2)}(t;u)}{\mathcal F_n^{(2)}(t;u)}
 =L_2(t;u).
\]
Consequently, \eqref{d5:eq:certificate} and
\eqref{d5:eq:kernel-shifts} yield the rational telescoping identity
\begin{equation}\label{d5:eq:telescoping}
\begin{aligned}
 &\sum_{j=0}^3\lambda_j^{(2)}(x;u)K_j(x;u)
 J_{n+j}^{(2)}(t;u)\mathcal F_{n+j}^{(2)}(t;u)\\
 &\qquad=\mathcal F_n^{(2)}(t;u)R_2(t;u,x)
 +\eta_2(t;u)\mathcal F_n^{(2)}(u^3t;u)R_2(u^3t;u,x),
 \qquad x=u^n.
\end{aligned}
\end{equation}

\subsection{A convergent residue weight and the product family}

Initially take $0<q<1$ and specialize $u=q^{-1}$ only in the finite
identity \eqref{d5:eq:telescoping}. Define
\begin{equation}\label{d5:eq:weight}
 \Psi_{\rm D5}(t):=(q^3t,q^{-1}/t;q^3)_\infty.
\end{equation}
This is holomorphic on $\mathbb C\setminus \{0\}$, with convergent base $q^3$,
and
\begin{equation}\label{d5:eq:weight-shift}
 \frac{\Psi_{\rm D5}(u^3t)}{\Psi_{\rm D5}(t)}
 =\frac{1-t}{1-q^{-1}/t},
 \qquad
 \eta_2(t;u)
 =-u^3\frac{\Psi_{\rm D5}(u^3t)}{\Psi_{\rm D5}(t)}.
\end{equation}
On the same two rotated $q$-lattices used in Section~\ref{sec:li-wang},
put
\[
 \mathcal L_{\rm D5}(T):=\frac12
 \sum_{\epsilon=1}^2\sum_{m\in\mathbb Z}
 \operatorname{Res}_{t=\omega^\epsilon q^m}
 \bigl(\Psi_{\rm D5}(t)T(t)\bigr).
\]
Only finitely many rational poles contribute. Reindexing under
$t\mapsto u^3t=q^{-3}t$ gives
\begin{equation}\label{d5:eq:residue-cancel}
 \mathcal L_{\rm D5}
 \bigl(T(t)+\eta_2(t;u)T(u^3t)\bigr)=0.
\end{equation}

Set
\begin{equation}\label{d5:eq:Vfamily}
\begin{aligned}
 c_{n,2}^\vee&:=q^{2-3n},\\
 \xi_{n,j}^{(2)}&:=\omega q^{1-2n+j}+\omega^2q^{1-n-j}
 \qquad(0\leq j\leq n),\\
 \mathcal V_n(q)&:=\mathcal E_3(-1)^nq^{-3\binom n2}(q;q)_n
 \Theta_{c_{n,2}^\vee}
 [\xi_{n,0}^{(2)},\ldots,\xi_{n,n}^{(2)}].
\end{aligned}
\end{equation}
The nodes are distinct by the same factorization used after
\eqref{lw:eq:Wn}. Indeed, with $\alpha_j=\omega q^{1-2n+j}$,
\[
 \xi_{n,j}^{(2)}-\xi_{n,k}^{(2)}
 =(\alpha_j-\alpha_k)
 \left(1-\frac{c_{n,2}^\vee}{\alpha_j\alpha_k}\right),
 \qquad
 \frac{c_{n,2}^\vee}{\alpha_j\alpha_k}
 =\omega q^{n-j-k},
\]
so neither factor vanishes for $j\ne k$.

At $u=q^{-1}$,
\begin{equation}\label{d5:eq:weight-kernel}
\begin{aligned}
 \Psi_{\rm D5}(t)(u^4/t;u^3)_n
 &=(q^3t,q^{-3n-1}/t;q^3)_\infty\\
 &=\Theta_{c_{n,2}^\vee}
 \left(q^3\left(t+\frac{u^{3n+4}}t\right)\right).
\end{aligned}
\end{equation}
Set $\zeta=t+u^{3n+4}/t$. The same local change of variable as in
\eqref{lw:eq:Ldd}, now with the two preimage strings beginning at
$\omega u^{n+2}$ and $\omega^2u^{n+2}$, gives
\begin{equation}\label{d5:eq:Ldd}
 \mathcal L_{\rm D5}
 \bigl(J_n^{(2)}\mathcal F_n^{(2)}\bigr)
 =q^{3n}\Theta_{c_{n,2}^\vee}
 [\xi_{n,0}^{(2)},\ldots,\xi_{n,n}^{(2)}].
\end{equation}

\begin{lemma}\label{d5:lem:Vrec}
The sequence $\mathcal V_n$ satisfies the recurrence
\eqref{d5:eq:mainrec}.
\end{lemma}

\begin{proof}
Apply \eqref{d5:eq:residue-cancel} to
\eqref{d5:eq:telescoping}. Since
$K_j(u^n;u)=(u;u)_{n+j}/(u;u)_n$, the sequence
\[
 \mathfrak Z_n^{(2)}:=\mathcal E_3^{-1}(u;u)_n
 \mathcal L_{\rm D5}
 \bigl(J_n^{(2)}\mathcal F_n^{(2)}\bigr)
\]
satisfies
\[
 \sum_{j=0}^3\lambda_j^{(2)}(u^n;u)
 \mathfrak Z_{n+j}^{(2)}=0.
\]
By \eqref{d5:eq:Ldd} and the standard inversion formula for the finite
$q$-factorial used in Section~\ref{sec:li-wang},
\begin{equation}\label{d5:eq:normalization}
 \mathfrak Z_n^{(2)}
 =\mathcal E_3^{-2}q^{n(n+1)}\mathcal V_n.
\end{equation}
Substituting \eqref{d5:eq:normalization} and $u=q^{-1}$ into the
recurrence above and multiplying by $q^{2n+6}$ gives exactly
\eqref{d5:eq:mainrec}, with the coefficients in
\eqref{d5:eq:coefficients}. This proves the lemma.
\end{proof}

\subsection{Boundary matching and uniqueness}

We use the coefficient functions $d_m(\gamma)$ and
$\mathcal R(\gamma;q)$ introduced in \eqref{lw:eq:dm}. In particular,
Lemma~\ref{lw:lem:coefficient} and \eqref{lw:eq:R-contiguous} are
available without change.

The next boundary statement is the $\tau=2$ analogue of
Lemma~\ref{lw:lem:boundary}.  The two lemmas have the same structural
form, but the $\tau=2$ transfer leads to different specializations of
the same function $\mathcal R(\gamma;q)$.

\begin{lemma}\label{d5:lem:boundary}
For every $n\geq0$, the two sequences have the same boundary
expansion:
\begin{equation}\label{d5:eq:boundary}
\begin{aligned}
 \mathcal M_n
 &=\mathcal R(q^2;q)
 +\frac{q^{n+1}}{1-q}
  \bigl(\mathcal R(q^2;q)+\mathcal R(q^5;q)\bigr)
 +O(q^{2n+2}),\\
 \mathcal V_n
 &=\mathcal R(q^2;q)
 +\frac{q^{n+1}}{1-q}
  \bigl(\mathcal R(q^2;q)+\mathcal R(q^5;q)\bigr)
 +O(q^{2n+2}).
\end{aligned}
\end{equation}
\end{lemma}

\begin{proof}
For $\mathfrak N(n,2)$, separating the terms $r=0,1$ gives
\[
 \mathfrak N(n,2)
 =\mathcal R(q^5;q)
 +\frac{q^{n+1}}{1-q}\mathcal R(q^2;q)
 +O(q^{2n+3}).
\]
Indeed, for $r\geq2$ one has $Q(r,s)+2s\geq3$. Similarly,
\[
 q^2\mathfrak N(n-2,5)
 =q^2\mathcal R(q^8;q)
 +\frac{q^{n+1}}{1-q}\mathcal R(q^5;q)
 +O(q^{2n+2}),
\]
because $Q(r,s)-2r+5s+2\geq2$ for $r\geq2$.
Taking $\gamma=q^2$ in \eqref{lw:eq:R-contiguous} gives
\[
 \mathcal R(q^2;q)=\mathcal R(q^5;q)+q^2\mathcal R(q^8;q),
\]
which proves \eqref{d5:eq:boundary} for $\mathcal M_n$.

For $\mathcal V_n$, the monomial divided-difference formula and
\eqref{lw:eq:dm} give
\begin{equation}\label{d5:eq:V-expansion}
 \mathcal V_n
 =\mathcal E_3(q;q)_n\sum_{k\geq0}
 \frac{(-1)^kq^{3nk+3\binom k2}
 d_{n+k}(q^{2+3k})}
 {(q^3;q^3)_{n+k}}
 h_k(\xi_{n,0}^{(2)},\ldots,\xi_{n,n}^{(2)}).
\end{equation}
The smallest exponent occurring in a node is $1-2n$. Hence the
$k$-th summand has order at least
\[
 3nk+3\binom k2+(1-2n)k
 =nk+\frac{3k^2-k}{2}.
\]
In particular, all terms with $k\geq2$ have order at least $2n+5$.
Furthermore,
\[
 h_1(\xi_{n,0}^{(2)},\ldots,\xi_{n,n}^{(2)})
 =-q^{1-2n}\frac{1-q^{n+1}}{1-q}.
\]
Using the quantity $\mathcal U_n$ from the proof of
Lemma~\ref{lw:lem:boundary},
\[
 \mathcal U_n
 =1+\frac{q^{n+1}}{1-q}+O(q^{2n+2}),
\]
the first two terms of \eqref{d5:eq:V-expansion} give
\[
\begin{aligned}
 \mathcal V_n={}&\mathcal U_nd_n(q^2)
 +\mathcal U_n
 \frac{q^{n+1}(1-q^{n+1})}
 {(1-q)(1-q^{3n+3})}d_{n+1}(q^5)
 +O(q^{2n+5}).
\end{aligned}
\]
Lemma~\ref{lw:lem:coefficient} now yields
\[
 \mathcal V_n
 =\mathcal R(q^2;q)
 +\frac{q^{n+1}}{1-q}
 \bigl(\mathcal R(q^2;q)+\mathcal R(q^5;q)\bigr)
 +O(q^{2n+2}),
\]
as required.
\end{proof}

\begin{theorem}\label{d5:thm:family}
For every $n\geq0$ and $|q|<1$,
\begin{equation}\label{d5:eq:family}
 \mathcal M_n(q)=\mathcal V_n(q).
\end{equation}
\end{theorem}

\begin{proof}
Let $e_n=\mathcal V_n-\mathcal M_n$. By
Lemmas~\ref{d5:lem:Mrec}, \ref{d5:lem:Vrec}, and
\ref{d5:lem:boundary},
\[
 \sum_{j=0}^3\varrho_j(q^n)e_{n+j}=0,
 \qquad e_n=O(q^{2n+2}).
\]
Since
\[
 \varrho_0(q^n)=q(1+q^{n+3}+q^{2n+5}),
\]
solving the recurrence for $e_n$ gives a companion matrix all of
whose entries belong to $q^{-1}\mathbb C[[q]]$. Fix $n$ and iterate
to level $N>n$. The product of $N-n$ companion matrices can lower the
$q$-order by at most $N-n$, whereas the terminal vector has order at
least $2N+2$. Therefore
\[
 \operatorname{ord}_q(e_n)\geq2N+2-(N-n)=N+n+2.
\]
Letting $N\to\infty$ gives $e_n=0$ coefficientwise. The direct-product
expansion \eqref{d5:eq:V-expansion} removes the apparent singularity at
$q=0$, and absolute convergence then extends the identity throughout
$|q|<1$.
\end{proof}

\subsection{Evaluation at $n=0$}

\begin{proof}[Proof of Theorem~\ref{thm:fifth-dual-companion}]
At $n=0$,
\[
 c_{0,2}^\vee=q^2,
 \qquad
 \xi_{0,0}^{(2)}=\omega q+\omega^2q=-q.
\]
Theorem~\ref{d5:thm:family} gives
\[
 \mathfrak N(0,2)+q^2\mathfrak N(-2,5)
 =\mathcal E_3\Theta_{q^2}(-q)
 =\mathcal E_3(\omega q,\omega^2q;q^3)_\infty.
\]
Using
$ (1-v)(1-\omega v)(1-\omega^2v)=1-v^3 $
factorwise yields
\[
 (\omega q,\omega^2q;q^3)_\infty
 =\frac{(q^3;q^9)_\infty}{(q;q^3)_\infty}.
\]
Since
$\mathcal E_3=1/((q;q^3)_\infty(q^2;q^3)_\infty)$,
we conclude that
\[
 \mathfrak N(0,2)+q^2\mathfrak N(-2,5)
 =\frac{(q^3;q^9)_\infty}
 {(q;q^3)_\infty^2(q^2;q^3)_\infty}.
\]
This is \eqref{eq:fifth-dual-companion}.
\end{proof}

 \section{Concluding remarks}\label{sec:conclusion}
 
 The five modulo nine Kanade--Russell identities, the first four
 individual dual identities, and a dual companion of the fifth identity
 are established by explicit recurrences and coefficientwise boundary
 uniqueness. For the five original identities, the two divided-difference
 constructions use the kernels $\Phi_{9,c}$ and $\Phi_{3,c}$,
 respectively. Finite local residue sums and algebraic certificates
 provide the recurrences, while a common product expansion provides the
 boundary estimates. The cyclotomic construction also proves the
 parameterized identities of Theorem~\ref{thm:cyclotomic}.
 
 For the first three individual dual identities, the bilinear relations of
 Theorem~\ref{ww:bilinear} connect the two quadratic forms directly.
 The two row-vector certificates and their cross product explain why
 the bilinear expressions satisfy the dual recurrence. The resulting
 product formulas settle Conjecture~3.6 of Wang and Wang. For the fourth
 dual identity, the finite algebraic certificate for $\tau=1$ transfers
 to a direct-product divided-difference family, and a two-term boundary
 expansion proves Li--Wang's Conjecture~6.4.  The parallel transfer of
 the already established $\tau=2$ certificate gives
 Theorem~\ref{thm:fifth-dual-companion}, the dual companion of the fifth
 Kanade--Russell identity.  Thus both cyclotomic certificates admit
 direct-product dual transfers.   

 In particular, the first Kanade--Russell identity and its first dual
 identity prove the two $q$-series evaluations singled out by Sun and
 Wang~\cite{sun-wang} for the Dynkin-diagram pairs $(T_1,G_2)$ and
 $(G_2,T_1)$.  Thus, relative to the low-rank status recorded in their
 work, the modularity of these two cases is now unconditional.

 \appendix
\section{Coefficient data for the cyclotomic certificates}
\label{app:coefficients}
 
 Write
 \[
 \mathcal P_\tau(t)=\sum_{k=0}^6\kappa_{\tau,k}t^k,
 \qquad \tau=1,2.
 \]
 For completeness, we record here the explicit coefficients used in
 Lemma~\ref{b:lem:certificates}.

\subsection{The first certificate}

\begin{align*}
\kappa_{1,6}={}&q^{3} x^{3},\\
\kappa_{1,5}={}&q^{7} x^{4}\bigl(q^{3} x^{2} + q^{3} x + 1\bigr),\\
\kappa_{1,4}={}&- q^{11} x^{5}\bigl(- q^{6} x^{2} - q^{5} x^{3} - q^{5} x^{2} + q^{4} x^{4}\\
&\qquad{}- q^{4} x^{2} - q^{3} x^{3} - q^{3} x^{2} - q^{2} x^{3} - q^{2} x^{2} + q x^{2}+ q x - 1\bigr)
,\\
\kappa_{1,3}={}&q^{16} x^{8}\bigl(q^{9} x^{3} + q^{6} x - q^{5} x^{3} + q^{5}\\
&\qquad{}+ q^{4} x^{2} + 2 q^{4} x + q^{4} + q^{3} x + q x - 1\bigr),
\\
\kappa_{1,2}={}&- q^{22} x^{8}\bigl(q^{8} x^{6} - q^{8} x^{4} - q^{7} x^{4} - q^{6} x^{4}\\
&\qquad{}- q^{5} x^{4} - q^{5} x^{3} - q^{4} x^{4} + q^{3} x^{2} - q^{3} x - q^{2} x
- q x + 1\bigr),\\
\kappa_{1,1}={}&- q^{27} x^{11}\bigl(q^{8} x^{4} + q^{7} x^{4} - q^{7} x^{3} - q^{7} x^{2} - q^{6} x^{3} - 2 q^{6} x^{2} - q^{5} x^{3} \\
&\qquad{}- q^{5} x^{2} + q^{5} x + q^{4} x - q^{3} x^{2} + q^{3} x - q^{3} + q^{2} x + q x - 1\bigr),\\
\kappa_{1,0}={}&- q^{35} x^{14}\bigl(- q^{4} x + q^{3} x^{2} - q^{2} x - q x + 1\bigr).
\end{align*}

\subsection{The second certificate}

\begin{align*}
\kappa_{2,6}={}&q^{3} x^{3}\bigl(q^{5} x^{2} + q^{2} x + 1\bigr),\\
\kappa_{2,5}={}&q^{8} x^{4}\bigl(q^{8} x^{3} + q^{7} x^{4} + q^{7} x^{3} + q^{5} x^{2} + q^{4} x^{2} + q^{3} x^{2} + q^{3} x + q^{2} x + 1\bigr),\\
\kappa_{2,4}={}&- q^{13} x^{5}\bigl(- q^{11} x^{4} - q^{10} x^{5} - q^{10} x^{4} + q^{9} x^{6} - q^{9} x^{4} - q^{8} x^{5} - q^{8} x^{4} \\
&\qquad{} - q^{8} x^{3} - q^{7} x^{4} - q^{7} x^{3} - q^{6} x^{4} - q^{6} x^{3} - q^{6} x^{2} - q^{5} x^{3} - 2 q^{5} x^{2}  \\
&\qquad{} - q^{4} x^{3} - q^{4} x^{2} - q^{3} x^{3} - q^{3} x^{2} + q^{2} x^{2} - 1\bigr),\\
\kappa_{2,3}={}&q^{20} x^{8}\bigl(q^{12} x^{5} + q^{11} x^{3} - q^{10} x^{5} + q^{10} x^{4} + q^{10} x^{3} + q^{9} x^{4} + q^{9} x^{3} \\
&\qquad{}+ q^{8} x^{3} + q^{7} x^{3} + 2 q^{7} x^{2} - q^{6} x^{4} + q^{6} x^{3} + 2 q^{6} x + 2 q^{5} x^{2}  \\
&\qquad{} + 2 q^{5} x + 2 q^{4} x^{2} + q^{4} + q^{3} x + q^{2} x^{2} + q^{2} x + q^{2} - 1\bigr),\\
\kappa_{2,2}={}&- q^{26} x^{8}\bigl(q^{13} x^{8} - q^{12} x^{6} - q^{12} x^{5} + q^{11} x^{7} - 2 q^{11} x^{6} - q^{10} x^{5} - q^{9} x^{6} \\
&\qquad{} - 3 q^{9} x^{5} - q^{9} x^{4} - q^{8} x^{5} - q^{7} x^{5} - q^{7} x^{4} - q^{7} x^{3} + q^{6} x^{5} - q^{6} x^{3}  \\
&\qquad{}- 2 q^{5} x^{4} - q^{4} x^{4} - q^{4} x^{3} - q^{4} x^{2} + q^{3} x^{2} - q^{3} x - q^{2} x^{2}- q^{2} x + 1\bigr),\\
\kappa_{2,1}={}&- q^{33} x^{11}\bigl(q^{12} x^{6} + q^{11} x^{6} - q^{11} x^{5} - q^{11} x^{4} - q^{10} x^{5} - q^{10} x^{4} + q^{9} x^{5}  \\
&\qquad{} - 2 q^{9} x^{4} + q^{9} x^{3} - q^{8} x^{4} - q^{8} x^{3} - q^{8} x^{2} - q^{7} x^{3} - q^{6} x^{4} - q^{6} x^{3}  \\
&\qquad{} - q^{6} x^{2} + q^{6} x - q^{5} x^{3} - q^{3} x^{2} + q^{3} x - q^{3} + q^{2} x^{2} + q^{2} x - 1\bigr),\\
\kappa_{2,0}={}&- q^{43} x^{14}\bigl(- q^{7} x^{3} + q^{6} x^{4} - q^{6} x^{3} + q^{5} x^{2} - q^{4} x^{2} - q^{3} x - q^{2} x^{2} - q^{2} x + 1\bigr).
\end{align*}

\vspace{10mm}
\noindent{\bf Acknowledgments.}
This work was supported by the National Natural Science Foundation
of China (Grant No.~12371334) and the Qinglan Project. 

\medskip
\noindent{\bf Competing Interests.}
The author declares that he has no conflict of interest.

\medskip
\noindent{\bf Code and data availability.}
The exact symbolic Maple code and the resulting verification
output are available at 
\begin{center}
	\url{https://github.com/Ernest-Xia/Kanade--Russell-identities-Lemma-4.3},\\
	\url{https://github.com/Ernest-Xia/Kanade--Russell-identities-Lemma-5.2},\\
	\url{https://github.com/Ernest-Xia/Kanade--Russell-identities-Lemma-6.3}.
\end{center}

\medskip
\textbf{Declaration of AI usage.}
During the development of this work, the author used ChatGPT
(OpenAI) as a research-assistance tool to explore possible proof
strategies and to assist with language editing and proofreading.
All mathematical arguments presented in the paper were independently
verified, refined, and written by the author, who takes full
responsibility for the final content of the publication.

\end{document}